\documentclass[11pt,english]{article}
\usepackage{amsmath,amsthm,url,amsfonts,amssymb,amscd}
\pdfoutput=1
\usepackage{enumerate}
\usepackage{fullpage}
\usepackage{epsfig,graphics}

\usepackage[latin1]{inputenc}

\begin{document}

\theoremstyle{plain}
\newtheorem{theo}{Theorem}[section]
\newtheorem{prop}[theo]{Proposition}
\newtheorem{coro}[theo]{Corollary}
\newtheorem{conj}{Conjecture}
\newtheorem{propr}{Property}[subsection]
\newtheorem{lemm}[propr]{Lemma}
\newtheorem{fact}[propr]{Fact}
\newtheorem{proprf}[propr]{Fundamental property}

\theoremstyle{definition}
\newtheorem{exem}[propr]{Example}
\newtheorem{exems}[propr]{Examples}
\newtheorem{rema}[propr]{Remark}
\newtheorem{defi}[propr]{Definition}
\newtheorem{ques}[propr]{Question}

\title{Persistence of stratification of  normally expanded laminations }
\date {October 2007}
\author{Pierre Berger}
\maketitle
\thispagestyle{empty}

\def\IMSmarkvadjust{0 pt}
\def\IMSmarkhadjust{0 pt}
\def\IMSmarkhpadding{0 pt}
\def\IMSpubltext{Published in modified form:}
\def\SBIMSMark#1#2#3{
 \font\SBF=cmss10 at 10 true pt
 \font\SBI=cmssi10 at 10 true pt
 \setbox0=\hbox{\SBF \hbox to \IMSmarkhpadding{\relax}
                Stony Brook IMS Preprint \##1}
 \setbox2=\hbox to \wd0{\hfil \SBI #2}
 \setbox4=\hbox to \wd0{\hfil \SBI #3}
 \setbox6=\hbox to \wd0{\hss
             \vbox{\hsize=\wd0 \parskip=0pt \baselineskip=10 true pt
                   \copy0 \break%
                   \copy2 \break%
                   \copy4 \break}}
 \dimen0=\ht6   \advance\dimen0 by \vsize \advance\dimen0 by 8 true pt
                \advance\dimen0 by -\pagetotal
	        \advance\dimen0 by \IMSmarkvadjust
 \dimen2=\hsize \advance\dimen2 by .25 true in
	        \advance\dimen2 by \IMSmarkhadjust

%
%
  \openin2=publishd.tex
  \ifeof2\setbox0=\hbox to 0pt{}
  \else 
     \setbox0=\hbox to 3.1 true in{
                \vbox to \ht6{\hsize=3 true in \parskip=0pt  \noindent  
                {\SBI \IMSpubltext}\hfil\break
                \input publishd.tex 
                \vfill}}
  \fi
  \closein2
  \ht0=0pt \dp0=0pt
 \ht6=0pt \dp6=0pt
 \setbox8=\vbox to \dimen0{\vfill \hbox to \dimen2{\copy0 \hss \copy6}}
 \ht8=0pt \dp8=0pt \wd8=0pt
 \copy8
 \message{*** Stony Brook IMS Preprint #1, #2. #3 ***}
}

\SBIMSMark{2007/3}{October 2007}{}

\begin{abstract}

This manuscript complements the Hirsch-Pugh-Shub (HPS) theory on persistence of normally hyperbolic laminations and the theorem of Robinson on the structural stability of diffeomorphisms that satisfy Axiom A and the strong transversality condition (SA).

We generalize these results by introducing a geometric object: the stratification of laminations. It is a stratification whose strata are laminations. Our main theorem implies the persistence of some stratifications whose strata are normally expanded. The dynamics is a $C^r$-endomorphism of a manifold (which is possibly not invertible). The persistence means that for any $C^r$-perturbation of the dynamics, there exists a close $C^r$-stratification preserved by the perturbation.

This theorem in its elementary statement (the stratification is constituted by a unique stratum) gives the persistence of normally expanded laminations by endomorphisms, generalizing HPS theory. Another application of this theorem is the persistence, as stratifications, of submanifolds with boundary or corners normally expanded. 

Moreover, we remark that SA diffeomorphism gives a canonical stratifications: the stratification whose strata are the stable sets of basic pieces of the spectral decomposition. Our Main theorem then implies the persistence of some ``normally SA'' laminations which are not normally hyperbolic.
\end{abstract}
\newpage
\setlength{\parskip}{0pt}
\tableofcontents

\newpage\section*{Introduction}
\addtocontents{toc}{Introduction}
\subsection{Motivations}

In 1977, M. Hirsch, C. Pugh and M. Shub \cite{HPS} developed a theory which has been very useful for hyperbolic dynamical systems.
The central point of their work was to prove the $C^r$-persistence of manifolds, foliations, or more generally laminations which are $r$-normally hyperbolic and plaque-expansive, for all $r\ge 1$.

  We recall that a lamination is \emph{$r$-normally hyperbolic}, if the dynamics preserves the lamination (each leaf is sent into a leaf) and if the normal space to the leaves splits into two $Tf$-invariant subspaces, that $Tf$ contracts (or expands) $r$-times more sharply than the tangent space to the leaves. Plaque expansiveness is a generalization\footnote{For instance a normally hyperbolic lamination, whose leaves are the fibers of a bundle, is plaque-expansive.} of expansiveness to the concept of laminations. The $C^r$-persistence of such a lamination means that for any $C^r$-perturbation of the dynamics, there exists a lamination, $C^r$-close to the first, which is  preserved by the new dynamics, and such that the dynamics induced on the space of the leaves remains the same.

A direct application of this theory was the construction of an example of a robustly transitive  {diffeomorphism (every close diffeomorphism has a dense orbit) but not Anosov}.
Then their work was used for example by C. Robinson \cite{Rs} to prove the structural stability of  $C^1$-diffeomorphisms that satisfy Axiom A and the strong transversality condition.

Nowadays, this theory remains very useful in several mathematical areas such  as generic dynamical systems, differentiable dynamics, foliations theory or Lie group theory.

 Nevertheless, this theory is not optimal. There are laminations which are not normally hyperbolic but are stable. For example, let $S$ be the $2$-dimensional sphere and let $N$ be a compact manifold. Let $\mathcal L$ be the lamination structure on $N\times S$ whose leaves are the fibers of the canonical projection $N\times S\rightarrow S$. Let $f$ be the north-south dynamics on $S$. Let $F$ be the diffeomorphism on $N\times S$ equal to the product of the identity of $N$ with
$f$. One can  easily show that for any diffeomorphism $F'$ close to $F$ in the $C^1$-topology, there exists a lamination structure $\mathcal L'$ on $N\times S$ which is preserved by $F'$ and isomorphic to $\mathcal L$ by a map close to the identity. Here the lamination $\mathcal L$ is $C^1$-persistent, but is not $1$-normally hyperbolic.

Furthermore, in his thesis, M. Shub \cite{Shubthese} has shown that for a manifold $M$ and a $C^1$-endomorphism $f$, every compact set $K$ which is stable and on which $f$ is expanding, is structurally stable (any $C^1$-perturbation of $f$ preserves a compact subset, homeomorphic and $C^0$-close to $K$, such that via this homeomorphism the restriction of the dynamics to these compact sets are conjugate). By an endomorphism we mean a differentiable map, not necessarily bijective and possibly with some singularities. 

 Also, M. Viana \cite{V} has used a persistent normally expanded lamination of (co-)dimension one to build a robustly non-uniformly expanding map. However, to our knowledge, it has not been proved that a $r$-normally expanding and plaque-expansive lamination, {by an endomorphism}, is $C^r$-persistent. Yet, this result seems fundamental
in the study of endomorphisms, and should be helpful in order to reduce the gap between the understanding of endomorphisms and of diffeomorphisms (structural stability, existence of new non-uniformly expanding maps,...). 

Finally, since Ma\~ne's thesis \cite{Manethese}, we know that a  compact $C^1$-submanifold is $(1)$-normally hyperbolic
 if and only if it is
$C^1$-persistent and uniformly locally maximal ({\it i.e.} there exists a neighborhood $U$ of the submanifold $N$ such that the maximal invariant subset in $U$ of any $C^1$-perturbation is a submanifold $C^1$-close to $N$). 
However, a uniform locally maximal submanifold $N$ can be persistent as a stratified space without being normally hyperbolic. For example, assume that a planar diffeomorphism has a hyperbolic fixed point $P$ with a one-dimensional stable manifold $X$. We suppose that $X$ punctured by $\{P\}$ is contained in the repulsive basin of an expanding fixed point $R$. The set $\mathbb S$, equals to the union of $X$ and $\{R\}$, is homeomorphic to a circle. We may even define a stratification structure with $X$ and $\{R\}$ as strata. One easily shows that for any $C^1$-perturbation of the dynamics, there exists a hyperbolic fixed point $P'$ close to $P$ whose
one-dimensional stable manifold $X'$ punctured by $P'$ belongs to the repulsive basin of a fixed point $R'$ close to $R$. In particular, there is a stratification $(X',\{R'\})$ on $\mathbb S':= X'\cup \{R'\}$ which is preserved by the perturbation of the dynamics, such that $X'$ is $C^1$-close to $X$, $\{R'\}$ is close to $\{R\}$ and $\mathbb S'$ is $C^0$-close to $\mathbb S$. 

For these reasons, it seems useful and natural to ask the question of the persistence of stratifications of normally expanded laminations, in the endomorphism context. As the concept of stratification of laminations is new, we are going to define all the above terms. Then we will give several applications of the theory developed in this work. At the end of this introduction, we will formulate a more transparent special case of our theorem suitable for most of our applications.

\subsection{Stratifications of normally expanded laminations}

We recall that a lamination is a second-countable metric space $L$ locally modeled (via compatible charts) on the product of $\mathbb R^d$ with a locally compact space. The maximal set of compatible charts is denoted by $\mathcal L$.

Let $(L,\mathcal L)$ be a lamination $C^r$-embedded into a Riemannian manifold $M$. Let $f$ be a $C^r$-endomorphism of $M$, \emph{preserving} $(L,\mathcal L)$: $f$ sends each leaf of $\mathcal L$ into a leaf of $\mathcal L$, for some $r\ge 1$.
Let $T\mathcal L$ be the subbundle of $TM_{|L}$ whose fibers are the tangent spaces to the leaves of $\mathcal L$. We say that $f$ \emph{ $r$-normally expands} $(L,\mathcal L)$ if there exist $\lambda >1$ and a continuous positive function $C$ on $L$ such that for any $x\in L$, any unitary vectors $v_0\in T_x \mathcal L$ and $v_1\in (T_x\mathcal L)^\bot$, any $n\ge 0$, we have
\[\|p\circ Tf^n(v_1)\|\ge C(x)\cdot \lambda^n\cdot (1+\|Tf^n(v_0)\|^r),\]
with $p$ the orthogonal projection of $TM_{|L}$ onto $T\mathcal L^\bot$.

When $L$ is compact, it is consistent with the usual definitions of normal expansion by replacing $C$ with its minimum.

 A first result is:
\begin{theo} Let $(L,\mathcal L)$ be a lamination $C^r$-embedded into a Riemannian manifold $M$. Let $f$ be a $C^r$-endomorphism of $M$ which is $r$-normally expanding and plaque-expansive at $(L, \mathcal L)$. Let $L'$ be a precompact open subset of $L$ whose  closure is sent by $f$ into $L'$. Then the lamination structure induced by $(L, \mathcal L)$ on $L'$  is $C^r$-persistent.\end{theo}
 In particular, this theorem implies the $C^r$-persistence of compact $r$-normally expanded and plaque-expansive laminations. Actually, this theorem is a particular case of our main theorem.  In this memoir we will not give a direct proof of it. 

Let us define the stratifications of laminations. Following J. Mather \cite{Ma}, a \emph{stratified space} is the data of a second countable metric space $A$ with a locally finite partition $\Sigma$ of $A$ into locally closed subsets, satisfying
the axiom of the frontier:
  \[\forall (X,Y)\in \Sigma^2,\; cl(X)\cap Y\not=\emptyset\Rightarrow cl(X)\supset Y\]
The pair $(A,\Sigma)$ is called \emph{stratified space} with \emph{support} $A$ and \emph{stratification} $\Sigma$.

Following H. Whitney, R. Thom or J. Mather, we can endow a stratified space with some geometric structure. In such way, we define a \emph{laminar} structure on $(A,\Sigma)$ as a lamination structure on each stratum, such that if the closure of a stratum $X$ intersects a stratum $Y$, then the dimension of $X$ is at least equal to the dimension of $Y$. Then $\Sigma$ is called a \emph{stratification of laminations}. A \emph{(stratified) $C^r$-embedding} of this space into a manifold $M$ is a homeomorphism onto its image such that, its restriction to each stratum $X$ is a $C^r$-embedding of the lamination $X$ into $M$. We often identify the stratified space $(A,\Sigma)$ with its image by the embedding $i$.

\begin{exem} A Whitney's stratification is a laminar stratification whose strata consist of a single leaves.\end{exem}
\begin{exem}\label{intro:K}
Let $f$ be an endomorphism of a manifold $M$. Let $K$ be a compact subset of $M$, $f$-invariant $(f^{-1}(K)=K)$, nowhere dense and expanded. Then $K$ endowed with its $0$-dimensional lamination structure and $M\setminus K$ endowed with its manifold structure form a stratification of normally expanded laminations on $M$. This example will be useful for number of our results.  
\end{exem}
\begin{exem}\label{intro:axiom A} Given a diffeomorphism that satisfies Axiom A and the strong transversality condition, if we denote by $(\Lambda_i)_i$ the basic sets and $X_i:=W^s(\Lambda_i)$ the canonical lamination on the stable set of each $\Lambda_i$ (whose leaves are stable manifolds), then the partition $(X_i)_i$ is a stratification of normally expanded laminations.\end{exem}
 
  Given a manifold $M$, a stratification of laminations $\Sigma$ on $A\subset M$ and an endomorphism $f$ of $M$, we say that $f$ \emph{preserves} $(A,\Sigma)$ if $f$ preserves each stratum $X\in \Sigma$, as a lamination.
  
   A stratification of laminations $(A,\Sigma)$ preserved by $f\in End^r(M)$ is \emph{$C^r$-persistent}, if for any endomorphism $f'$ $C^r$-close to $f$, there exists a stratified embedding $i'$ $C^r$-close to the canonical inclusion $i$ such that $f'$ preserves the stratification $(A,\Sigma)$ embedded by $i'$, and for each stratum $X\in \Sigma$, every point $i'(x)\in i'(X)$ is sent by $f'$ into the image by $i'$ of a small plaque of $X$ which contains $f(x)$.

     The aim of this memoir is to present and prove a general theorem providing, for any $r\ge 1$, the $C^r$-persistence of stratifications of $r$-normally expanded laminations, under some extra geometric conditions.
    
    Let us illustrate, by some applications of our main result, the persistence of stratifications of laminations.

\subsubsection{Submanifolds with boundary}

\begin{theo} 
Let $(M,g)$ be a Riemannian manifold and let $N$ be a compact submanifold with boundary of $M$. Let $f$ be an endomorphism of $M$ which preserves and $1$-normally expands the boundary $\partial N$ and the interior $\mathring{N}$ of $N$.
Then the stratification $(\mathring N,\partial N)$ on $N$ is $C^1$-persistent.

In other words, for any endomorphism $f'$ $C^1$-close to $f$, there exist two submanifolds $\partial N'$ and $\mathring N'$ such that:
\begin{itemize}\item  $\mathring N'$ (resp. $\partial N'$) is preserved by $f'$, diffeomorphic and $C^1$-close to $\mathring N$ (resp. $\partial N$) for the compact-open topology,
\item the pair $(\mathring N',\partial N')$ is a stratification (of laminations) on $N':= \mathring{N'} \cup \partial N'$,
\item  the set $N'$ is  the image of $N$ by an embedding $C^0$-close to the canonical inclusion of $N$ into $M$.
\end{itemize}
\end{theo}
\begin{rema}
Usually, $N'$ is not a submanifold with boundary.\end{rema}
\begin{rema}
Our main theorem also implies the $C^r$-persistence of $r$-normally expanded submanifold as stratification, for any $r\ge 1$.\end{rema}

Let us generalize the above result to a larger context:

\subsubsection{Submanifolds with corners}
We recall that a compact manifold with corner $N$ is a differentiable manifold modeled on  $\mathbb R_+^d$. We denote by $\partial^{0_k}N$ the set of points of $N$ which, seen in a chart, have exactly $k$ coordinates equal to zero. The pair $(N,\Sigma:=\{\partial^{0_k} N\})$ is a stratified space. 
\begin{theo}
Let $i$ be a $C^1$-embedding of $N$ into a Riemannian manifold $(M,g)$. Let $f$ be a $C^1$-endomorphism of $N$, which preserves and $1$-normally expands each stratum  $\partial^{0_k} N$.
Then the stratification $\Sigma$ on $N$ is stable for $C^1$-perturbations of $f$.

 In other words, for every endomorphism $f'$ $C^1$-close to $f$, there exist submanifolds $(\partial^{0_k} N')_k$ such that:
 \begin{itemize}
\item for each $k$, $\partial^{0_k} N'$ is preserved by $f'$, is diffeomorphic and close to $\partial^{0_k} N$ in the $C^1$-compact-open topology,
\item the family $(\partial^{0_k} N')_k$ is a stratification (of laminations) on $N':=\cup_k \partial^{0_k} N'$,
\item the set $N'$ is the image of $N$ by an embedding $C^0$-close to the canonical inclusion of $N$ into $M$.
\end{itemize}
\end{theo}
\begin{rema}
Our main theorem also implies the $C^r$-persistence of $r$-normally expanded submanifold with corners as stratification, for any $r\ge 1$.\end{rema}

Our main result easily provides the persistence of many stratifications of  normally expanded laminations in product dynamics, as in the following examples.

\subsubsection{Invariant laminations of the Viana map in $\mathbb C\times \mathbb R$}
\label{intro:Viana}
\[\mathrm{Let}\; V\;:\; \mathbb C\times \mathbb R\rightarrow \mathbb C\times \mathbb R\]
\[(z,h)\mapsto (z^4,h^2+c)\]

 The map $z\mapsto z^4$ expands the unit circle $\mathbb S^1$ and preserves the interior of the unit disk $\mathbb D$. We endow $\mathbb S^1$ and $\mathbb D$ with the lamination structures of dimension 0 and 2 respectively.
 
  Let $c\in ]-2,1/4[$, then the map $h\mapsto h^2+c$ sends an open interval $I$ into it self and expands its boundary $\partial I$. 
  
We stratify the filled cylinder $C:= cl(\mathbb D\times I)$ by the laminations:
\begin{itemize}
\item $X_0:=\mathbb S^1\times \partial I$ of dimension 0,
\item $X_1:=\mathbb S^1\times I$ of dimension 1, whose leaves are $(\{\alpha\}\times I)_{\alpha\in \mathbb S^1}$,
\item $X_2:=\mathbb D\times \partial I$ of dimension 2,
\item $X_3:=\mathbb D\times I$ of dimension 3.
\end{itemize}
Let $\Sigma$ be the stratification of laminations on $C$ defined by these strata. We notice that $V$ preserves and $1$-normally expands this stratification.

The persistence  of this stratification, for $C^1$-perturbations of $V$, follows from our main theorem.

In other words, for every endomorphism $V'$ $C^1$-close to $V$, there exists a homeomorphism $i'$ of $C$ onto its image in $\mathbb C\times \mathbb R$, $C^0$-close to the canonical inclusion such that for each stratum $X_k\in \Sigma$:
\begin{itemize}
\item the restriction $i'_{|X_k}$ is an embedding of lamination, $C^1$-close to the canonical inclusion of $X_k$ in $\mathbb C\times \mathbb R$; in particular $i'$ is continuously leafwise differentiable, 
\item the lamination $i'(X_k)$ is preserved by $V'$, and for $x\in X_k$, $V'\circ i'(x)$ belongs to the image by $i'$ of a small plaque of $X_k$ containing $V(x)$.
\end{itemize}

\subsubsection{ Products of hyperbolic rational functions}\label{intro:hyper}
\[\mathrm{Let}\; f\;:\; \hat {\mathbb C}^n\rightarrow \hat{\mathbb C}^n\]
\[(z_i )_i\mapsto (R_i(z_i))_i\]

 We assume that for each $i$, $R_i$ is a hyperbolic rational function of the Riemann sphere $\hat {\mathbb C}$. It follows that its Julia set $K_i$ is expanded and the complement $X_i$ of $K_i$ in $\hat{ \mathbb C}$ is the union of attraction basins of the attracting periodic orbits. 

 Let $\Sigma$ be the stratification of laminations on $\hat{\mathbb  C}^n$ formed by the strata $(Y_J)_{J\subset \{1,\dots,n\}}$, $Y_J$ being of real dimension twice the cardinal of $J$ and with support:
\[ Y_k=\prod_{j\in J}X_j\times \prod_{j\in J^c}K_j.\]
The leaves of $Y_J$ are in the form $\prod_{j\in J} C_j\times \prod_{j\in J^c} \{k_j\}$, with $C_j$ a connected component of $\hat{\mathbb C}\setminus K_j$ and $k_j$ a point of $K_j$.

The $C^r$-persistence  of this stratification of $r$-normally expanded laminations, for all $r\ge 1$, follows from our main theorem.

A similar result exists on $\mathbb R^n$ for products of real hyperbolic polynomial functions.

\subsection{Structure of trellis of laminations and main result}We construct, in section \ref{cexp}, a very simple example of a stratification of normally expanded laminations which is not persistent. Therefore, some new conditions are necessary to imply the persistence of stratifications of laminations.

The hypotheses of our main result on persistence of stratified space $(A,\Sigma)$ require the existence of a {\it tubular neighborhood} $(L_X, \mathcal L_X)$ for each stratum $X\in \Sigma$: this is a lamination structure $\mathcal L_X$ on an open neighborhood $L_X$ of $X$ in $A$, such that each leaf of $X$ is a leaf of $\mathcal L_X$.
 
  Existence of a similar structure was already conjectured in a local way by H. Whitney \cite{W1} in the study of analytic varieties. It was also a key ingredient in the proofs by W. de Melo \cite{dM} and by C. Robinson \cite{Rs} of the structural stability of diffeomorphisms that satisfy axiom $A$ and the strong transversality condition defined in example \ref{intro:axiom A}.

A {\it $C^r$-trellis (of laminations)} on a laminar stratified space $(A,\Sigma)$ is a family of tubular neighborhoods $\mathcal T=(L_X, \mathcal L_X)_{X\in \Sigma}$ such that for all strata $X\le Y$:\begin{itemize}
\item each plaque of $\mathcal L_Y$ included in $L_X$ is $C^r$-foliated by plaques of $\mathcal L_X$, 
\item given two close points $(x,x')\in (\mathcal L_X\cap \mathcal L_Y)^2$, there exist two plaques of $\mathcal L_Y$ containing respectively $x$ and $x'$ for which  such foliations are diffeomorphic and $C^r$-close.\end{itemize} 

\begin{exem}\label{treillis sur compact repu} The stratification in example \ref{intro:K} admits a trellis structure. Let $L_K$ be a neighborhood of $K$ in $M$ endowed with the 0-dimensional lamination structure $\mathcal L_K$. Then $((L_K, \mathcal L_K),X)$ is a trellis structure on $(M,(K,X))$.\end{exem} 
\begin{exem} The canonical stratification $( \partial N, \mathring N)$ of a manifold with boundary $N$ admits a trellis structure: Let $\mathcal L_{\partial N}$ be the lamination structure on a small neighborhood $L_{\partial N}$ of the boundary $\partial N$ whose leaves are the subset of points in $N$ equidistant to the boundary. Then $((L_{\partial N},\mathcal L_{\partial N}), \mathring N)$ is a trellis structure on $(N, ( \partial N, \mathring N))$.\end{exem}

 A  $C^r$-embedding $i$ of $(A,\Sigma)$ into a manifold $M$ is {\it $\mathcal T$-controlled} if $i$ is a homeomorphism onto its image and the restriction of $i$ to $L_X$ is a $C^r$-embedding of the lamination $\mathcal L_X$, for every $X\in \Sigma$.

 We can now formulate a special case of our main theorem:    
\begin{theo}\label{intro:main} Let $r\ge 1$ and let $(A,\Sigma)$ be a compact stratified space supporting a $C^r$-trellis structure $\mathcal T$. Let $i$ be a $\mathcal T$-controlled $C^r$-embedding of $(A,\Sigma)$ into a manifold $M$.  We identify $A$, $\Sigma$ and $\mathcal T$ with their images in $M$. Let $f$ be a $C^r$-endomorphism of $M$ preserving $\Sigma$ and satisfying for each stratum $X$:
\begin{enumerate}[(i)]
\item $f$ $r$-normally expands $X$ and is plaque-expansive at $X$,

there exists a neighborhood $V_X$ of $X$ in  $L_X$ such that 
\item each plaque of $\mathcal L_X$ included in $V_X$ is sent into a leaf of $\mathcal L_X$,

\item there exists $\epsilon>0$, such that every $\epsilon$-pseudo orbit\footnote{An $\epsilon$-pseudo-orbit $(x_n)_n\in V_X^\mathbb N$ respects $\mathcal L_X$, if for all $n\ge 0$, the points $f(x_n)$ and $x_{n+1}$ belong to a same plaque of $\mathcal L_X$ of diameter less than $\epsilon$.}  of $V_X$ which respects $\mathcal L_X$ is included in $X$.\end{enumerate}

  Then for $f'$ $C^r$-close to $f$, there exists a $\mathcal T$-controlled embedding  $i'$, close to $i$, such that for the identification  of $A$, $\Sigma$ and $\mathcal T$ via $i'$, the properties $(i)$, $(ii)$ and $(iii)$ hold with $f'$. Moreover, for each stratum $X\in \Sigma$, each point $i'(x)\in i'(X)$ is sent by $f'$ into the image by $i'$ of a small $X$-plaque of $x$ containing $f(x)$. 
  
   In particular, the stratification of laminations $\Sigma$ is $C^r$-persistent.
\end{theo}

 \begin{rema}
This result has also a version which allows $A$ to be non-compact and/or $i$ to be an immersion. In the immersion case, the plaque-expansivity condition is not required.\end{rema}
\begin{rema}\label{r3} We have also a better conclusion: for every stratum $X$ there exists a neighborhood $V_X'$ of $X$ in $\mathcal L_X$ such that, for every $f'$ $C^r$-close to $f$,  each point $i'(x)\in i'(V_X')$ is sent by $f'$ into the image by $i'$ of a small plaque of $\mathcal L_X$ containing $f(x)$.\end{rema}

 The main difficulty to apply this theorem is to build a trellis structure that satisfies $(ii)$.

Nevertheless, thanks to the formalism, the following proposition provides many trellis structures which imply the persistence of these stratification, via our main result.

\begin{propr}\label{prop:intro} 
Let $r\ge 1$ and let $(A,\Sigma)$ and $(A',\Sigma')$ be compact stratified spaces endowed with $C^r$-trellis structures $\mathcal T$ and $\mathcal T'$ respectively. Let $i$ and $i'$ be $C^r$-embeddings  $\mathcal T$ and $\mathcal T'$-controlled of $(A,\Sigma)$ and  $(A,\Sigma)$ into manifolds $M$ and $M'$ respectively.
 Let $f\in End^r(M)$ and $f'\in End^r(M')$ satisfying properties $(i)$, $(ii)$ and $(iii)$ of theorem \ref{intro:main}.

 Then the partition $\Sigma\times \Sigma'$ on $A\times A'$, whose elements are the product of a stratum of $\Sigma$ with a stratum of $\Sigma'$, is a stratification of laminations which is preserved by the product dynamics $(f,f')$ of $M\times M'$. Moreover, if $(f,f')$ $r$-normally expands this stratification, then properties $(i)$, $(ii)$, and $(iii)$ are satisfied for $(f,f')$ and $\Sigma\times \Sigma'$.
In particular this last stratification is $C^r$-persistent.
\end{propr}

For instance, by using this proposition and example \ref{treillis sur compact repu}, we get the proof of the persistence of examples \ref{intro:Viana} and \ref{intro:hyper}.

On the other hand, remark \ref{r3} has to be considered if one wants to apply our result to prove theorems on structural stability or persistence of laminations which are not hyperbolic. For example, we needed this remark, together with a trellis structure built by de Melo on the stratification of laminations defined in example \ref{intro:axiom A}, to show the following:

\begin{theo} Let $s$ be a $C^1$-submersion of a compact manifold $M$ onto a compact surface $S$. Let $\mathcal L$ be the lamination structure on $M$ whose leaves are the connected components of the fibers of $s$.

Let $f$ be a diffeomorphism of $M$ which preserves the lamination $\mathcal L$. Let $f_b\in Diff^1(S)$ be the dynamics induced by $f$ on the leaves space of $\mathcal L$.
We suppose that:
\begin{itemize}
\item  $f_b$ satisfies axiom $A$ and the strong transversality condition,  
\item the $\mathcal L$-saturated subset generated by the non-wandering set of $f$ in $M$ is $1$-normally hyperbolic.\end{itemize}
Then $\mathcal L$  is $C^1$-persistent.\end{theo}

We hope to publish soon a version of this result closer to our conjecture (see section \ref{axiom A}) whose statement generalizes the persistence of normally hyperbolic laminations and the structural stability of diffeomorphisms that satisfy Axiom A and the strong transversality condition.\\ 

This memoir is the main part of my PhD thesis under the direction of J-C. Yoccoz. I would like to thank him for his guidance. I would like to thank also M. Viana, C. Bonatti, E. Pujal, P. Pansu, F. Paulin, C. Murollo, and D. Trotman for many discussions. Finally, I thank P-Y. Fave who modeled the 3D illustrations of this memoir.

\subsection{Plan}

The first chapter is mostly geometric. We introduce the definitions and the terminologies necessary for all the other chapters. 

In the first section of this chapter, we recall the definitions of laminations, their morphisms and the topologies on these spaces.

In the second section of this chapter, we introduce the stratification of laminations. We present how they are related to other kinds of stratifications (analytic and differentiable) and we show some simple properties of them. Then this section proves that the diffeomorphisms satisfying axiom $A$ and the strong transversality condition defines canonically two stratifications of laminations. At the end of this section, we  define stratified morphisms and endow the space of morphisms with a topology. 

In the third section of this chapter, we introduce the trellis of laminations structure. Then we defines morphisms of this structure and endow the space of morphisms with a topology. Finally, in this chapter, we show how the trellis structure are linked to other works in dynamical system, differentiable geometry or analytic geometry.\\      

The second chapter contains our main result on persistence of stratifications of normally expanded laminations. 

In the first section, we restrict the study to the lamination. First, we discus on the definition of the preservation and persistence of laminations, embedded or immersed. The definitions are motivated by a negative answer to a question of Hirsch-Pugh-Shub. Then we define the $r$-normal expansion of an immersed or embedded lamination and gives some related properties.  Finally, in this section, we state the theorem \ref{th1} on $C^r$-persistence of $r$-normally expanded immersed laminations. After defining the plaque-expansivity, we give the corollary \ref{cor1} on $C^r$-persistence of $r$-normally and plaque-expansive embedded laminations. Both are the restrictions of our main theorem \ref{th2} and its corollary \ref{cor2}, to the case where the stratification consists of a unique stratum.

The second section of this chapter contains the main result. First, we define the preservation and the persistence of stratifications of laminations, embedded or immersed. Then it presents the main theorem \ref{th2} via its corollaries \ref{cor3}, \ref{cor2}, and some easy applications. 

In the third section, we motive our geometrical viewpoint  on the persistence of stratifications of laminations, by giving a counter example of a compact stratification which is not persistent and does not admit a trellis structure.  

The fourth section provides some applications of our main result. We begin by giving the statement of our result on the persistence of normally expanded submanifolds with boundary or corners as stratifications. But, we only give the idea of proof of such applications. Then we give an extension of the Shub's theorem on conjugacy of repulsive compact set. Further, we show proposition \ref{prop:intro} which implies some examples of persistent stratifications of laminations in product dynamics. Finally, we state a conjecture on the persistent of ``normally SA laminations'' and announce a partial result in this direction which was proved by using our main result.\\ 

The five section consists of the proof of our main result. Annex A provides some analysis results needed in this work. Annex B consists of the proof of the existence of an adapted metric to the normal expansion of a lamination by an endomorphism. In annex C, we adapt and develop some results on the plaque-expansivity to the endomorphism context.

\newpage\section{Geometry of stratification of laminations}
In this chapter, we introduce the laminar stratified space. The laminar stratified space is a natural generalization of laminations and stratifications. We know that laminations and stratifications occur in dynamical system as persistent and preserved  structures (as in Hirsch-Pugh-Shub theory or the Morse-Smale theory).

We will state and illustrate our main result on persistence of stratification of laminations in chapter \ref{chap:persist:strati}. In this chapter, we only deal with the geometry of this structure, recalling or introducing some definitions and properties.

Throughout this chapter we denote by $r$ a positive integer or the symbol $\infty$. 
\subsection{Laminations}
\subsubsection{Definitions}
 Let us consider  a locally compact and second-countable metric space $L$ covered by open sets $(U_i)_i$, called
 \emph{distinguished open sets}, endowed with homeomorphisms $h_i$ from $U_i$ onto $V_i\times T_i$, where $V_i$ is an open set of $\mathbb R^d$ and $T_i$ is a metric space.
 
 We say that the \emph{charts} $(U_i, h_i)_i$ define a $C^r$-\emph{atlas} of a lamination structure on $L$ of dimension $d$ if the \emph{coordinate change} $h_{ij}=h_j\circ h_i^{-1}$ can be written in the form
 \[h_{ij}(x,t)=(\phi_{ij}(x,t),\psi_{ij}(x,t)),\]
 where $\phi_{ij}$ takes its values in $\mathbb R^d$, the partial derivatives $(\partial_x^s \phi_{ij})_{s=1}^r$ exist and are continuous on the domain of $\phi_{ij}$, and $\psi_{ij}(\cdot,t)$ is locally constant for any $t$.
 
 Two $C^r$-atlases are said to be \emph{equivalent} if their union is a $C^r$-atlas.
 
 A $(C^r)$-\emph{lamination} is a metric space $L$ endowed with a maximal $C^r$-atlas $\mathcal{L} $.

 A \emph{ plaque } is a subset of $L$ which can be written in the form $h_i^{-1}(V_i^0\times\{t\})$, for a
 chart $h_i$ and a connected component $V_i^0$ of $V_i$. A plaque that contains a point $x\in L$ will be denoted by  $\mathcal{L}_x$; the union of the plaques containing $x$ and of diameter less than $\epsilon>0$ will be denoted by $\mathcal L_x^\epsilon$. As the diameter is given by the metric of $L$, the set $\mathcal L_x^\epsilon$ is, in general, not homeomorphic to a manifold. The \emph{leaves } of $\mathcal L$ are the smallest subsets of $L$ which contain any plaque that intersects them.
 
 We say that a subset $P$ of $L$ is \emph{saturated} if it is a union of leaves.

 If moreover it is a locally compact subset, this subset is \emph{$\mathcal L$-admissible}. Then the charts $(U_i,\phi_i)$ of $\mathcal L$ restricted to $U_i\cap P$ define a lamination structure on $P$. We will call this structure \emph{the restriction of $\mathcal L$ to $P$} and we denote this structure by $\mathcal{L}_{|P}$.
 
 Similarly, if $V$ is an open subset of $L$, the set of the charts $(U,\phi)\in\mathcal{L}$ such that $U\subset V$ constitutes a lamination structure on $V$, which is denoted by $\mathcal{L}_{|V}$.
 
 A subset $P$ of $L$ which is $\mathcal{L}_{|V}$-admissible for a certain open subset $V$ of $L$
will be called \emph{ $\mathcal{L}$-locally admissible}, and we will denote by $\mathcal{L}_{|P}$ its lamination structure $\mathcal{L}_{|V_{|P}}$.

We recall that the locally compact subsets of a locally compact metric space are the intersections of open and closed subsets.

\begin{exems}
\begin{itemize}
    \item A manifold of dimension $d$ is a lamination of the same dimension.
    \item A $C^r$-foliation on a connected manifold induces a $C^r$-lamination structure.
    \item A locally compact and second-countable metric space defines a lamination of dimension zero.
    \item If $K$ is a locally compact subset of $\mathbb S^1$, then the manifold structure of the circle $\mathbb S^1$ induces on $\mathbb S^1\times K$ a $C^\infty$-lamination structure whose leaves are $\mathbb S^1\times \{k\}$, for $k\in K$.
    \item The stable foliation of an Anosov $C^r$-diffeomorphism induces a $C^r$-lamination structure whose leaves are the stable manifold. \end{itemize}
\end{exems}
\begin{propr} If $(L, \mathcal{L})$ and $(L',\mathcal{L}')$ are two laminations, then $L\times L'$ is endowed with the lamination structure whose leaves are the product of the leaves of $(L, \mathcal{L})$ with the leaves of  $(L',\mathcal{L}')$. We denote this structure by $\mathcal{L}\times\mathcal{L}'$.\end{propr}

\begin{propr}\label{cuplam} If $\mathcal L$ and $\mathcal L'$ are two lamination structures on two open subsets $L$ and $L'$ of a metric space $A$ such that the atlases $\mathcal L_{L\cap L'}$ and $\mathcal L_{L\cap L'}$ are equivalent, then the union of the atlases $\mathcal L$ and $\mathcal L' $ is an atlas on $L\cup L'$.\end{propr}

\subsubsection{Morphisms of laminations }
A map $f$ is a \textit{$C^r$-morphism (of laminations)} from $(L,\mathcal{L})$ to $(L',\mathcal{L}')$ if it is a continuous map from $L$ to $L'$ such that, seen via charts $h$ and $h'$, it can be written in the form:
\[h'\circ f\circ h^{-1} (x,t)= (\phi (x,t),\psi(x,t))\]
 where $\phi$ takes its values in $\mathbb R^{d'}$, $\partial_x^s \phi$ exists, is continuous on the domain of $\phi$, for all $s\in\{1,\dots ,r\}$ and $\psi(\cdot , t)$ is locally constant.
 
If, moreover, the linear map $\partial_x \phi(x,t)$ is always one-to-one, we will say that $f$ is an \emph{immersion (of laminations)}.


 An \emph{embedding (of laminations)} is an immersion which is a homeomorphism onto its image.
 
 The \emph{endomorphisms of $(L,\mathcal{L})$} are the morphisms from $(L,\mathcal{L})$ into itself.

 We denote by:
 \begin{itemize}
\item $Mor^r(\mathcal{L},\mathcal{L}')$ the set of the $C^r$-morphisms from $\mathcal{L}$ into $\mathcal{L}'$,
\item $Im^r(\mathcal{L},\mathcal{L}')$ the set of the $C^r$-immersions from $\mathcal{L}$ into $\mathcal{L}'$,
\item $Emb^r(\mathcal{L},\mathcal{L}')$ the set of the $C^r$-embeddings from $\mathcal{L}$ into $\mathcal{L}'$,
\item $End^r(\mathcal{L})$ the set of the $C^r$-endomorphisms of $\mathcal L$.\end{itemize}

  We denote by $T \mathcal{L}$ the vector bundle over $L$, for which the fiber of $x\in L$, denoted by $T_x \mathcal{L}$, is the tangent space at $x$ to its leaf. If $f$ is morphism from $\mathcal{L}$ into $\mathcal{L}'$, we denote by $Tf$ the bundle morphism from
  $T\mathcal{L}$ to $T \mathcal{L}'$ over $f$ induced by the differential of $f$ along the leaves of $\mathcal L$.
\begin{rema} If $M$ is a manifold, we notice that $End^r(M)$ denotes the set of $C^r$-maps from $M$ into itself, possibly non-bijective and possibly with singularities.\end{rema}

\begin{exem} \label{lamhyper}
Let $f$ be a $C^r$-diffeomorphism of a manifold $M$ and let $K$ be hyperbolic compact subset of $M$. Then the union $W^s(K)$ of stable manifold of points in $K$ is the image of a $C^r$-lamination $(L,\mathcal L)$ immersed injectively.

Moreover if every stable manifold does not accumulate on $K$, then $(L,\mathcal L)$ is a $C^r$-embedded lamination.
 \end{exem}

\begin{proof} We endow $M$ with an adapted metric $d_M$ to the hyperbolic compact $K$. For a small $\epsilon >0$, we call \emph{local stable manifold of diameter $\epsilon$ of $x\in K$}, the set of points whose trajectory is $\epsilon$-distant to the trajectory of $x$. Let
$W^s_\epsilon(K)$ be the union of stable manifolds of points in $K$ of diameter $\epsilon$. For $\epsilon$ small enough, the closure of $W^s_\epsilon(K)$ is sent by $f$ into $W^s_\epsilon(K)$ and supports a canonical $C^r$-lamination structure $\mathcal L_0$. Let $C$ be the subset $W^s_\epsilon(K)\setminus f^2\big(cl(W^s_\epsilon(K))\big)$.
For $i>0$, we denote by $C_i$ the set $f^{-i}(C)$ and by $C_0$ the set $W^s_\epsilon(K)$.
The union $\cup_{n\ge 0} C_n $ is consequently equal to $W^s(K)$. Moreover, for $k,l\ge 0$, if $C_k$ intersects $C_l$ then $|k-l|\le 1$.

 Let us now construct a metric on $W^s(K)$ such that $( C_n)_n$ is an open covering and such that the topology induced by this metric on $ C_n$ is the same than the one of $M$. For $(x,y)\in  W^s(K)^2$, we denote by $d(x,y)$:
 \[\inf\Big\{\sum_{i=1}^{n-1} d_M(x_i,x_{i+1}); \;n>0,\;
 (x_i)_i\in W^s(K)^n,\;\mathrm{such\; that}\]
\[  x_1=x,\; x_n=y,\; \forall i\exists j: (x_i,x_{i+1})\in C_j^2\Big\}.\]
We remark that $d$ is a distance with announced properties. The metric space $L$ is therefore $W^s(K)$ endowed with this distance. We remark that if every stable manifold does not accumulate on $K$, then the topology on $L$ induced by this metric and the metric of $M$ are the same. In other words $\mathcal L$ is embedded.

 For $i>0$, the open subset $C_i$ supports the $C^r$-lamination structure $\mathcal L_{i}$ whose  charts are the composition of the charts of $\mathcal L_{0|C}$ with $f^{i}$. As $f$ is a diffeomorphism, for any $i,j$, the restriction of $\mathcal L_i$ and $\mathcal L_j$ to $C_i\cap C_j$ are equivalent.
 By property \ref{cuplam}, the structures $(\mathcal L_i)_{i\ge 0}$ generate a $C^r$-lamination structure $\mathcal L$ on $L$.
\end{proof}

\subsubsection{Riemannian metric on a lamination}
A \emph{Riemannian metric} $g$ on a $C^r$-lamination $(L,\mathcal L)$ is an inner product $g_x$ on each fiber $T_x\mathcal L$ of $T\mathcal L$, which depends $C^{r-1}$-continuously on the base point $x$.
It follows from the existence of partitions of unity (see proposition \ref{part1}) that any lamination $(L,\mathcal{L})$ can be endowed with a certain Riemannian metric.
 \footnote{As the tangent bundle is only continuous when $r=1$, we cannot define the geodesic flow along the leaves. Nevertheless, there exists a $C^\infty$-lamination structure, compatible with the $C^1$-structure. For such a structure, we can define the geodesic flow for another regular metric. To show it we can adapt the proof of theorem 2.9 in \cite{H} with the analysis techniques of annex \ref{partlam}.
}

A Riemannian metric induces -- in a standard way -- a metric on each leaf.
For two points $x$ and $y$ which belong to a same leaf, the distance between $x$ and $y$ is defined by:
\[d_g(x,y)=\inf_{\{\gamma\in Mor([0,1],\mathcal L);\gamma(0)=x,\gamma(1)=y\}}\int_0^1\sqrt{g(\partial_t\gamma(t),\partial_t\gamma(t))},dt\]

\subsubsection{Equivalent Classes of morphisms}\label{Top}
We will say that two elements $f$ and $f'$ in $Mor^r(\mathcal{L},\mathcal{L}')$ (resp. $Im^r(\mathcal{L}, \mathcal{L}')$ and $End^r(\mathcal{L}))$ are equivalent if for every $x\in L$, the points $f'(x)$
 and $f(x)$ belong to a same leaf of $\mathcal L'$. The equivalence class of $f$ will be denoted by $Mor^r_f (\mathcal{L},\mathcal{L}')$
(resp. $Im_f^r (\mathcal{L},\mathcal{L}')$ and $End_f^r(\mathcal{L})$).

Given a Riemannian metric $g$ on $(L',\mathcal L')$, we endow the equivalence class
with the compact-open topology $C^r$. Let us describe elementary open sets which generate the topology.

Let $K$ be a compact subset of $L$ such that $K$ and $f(K)$ are included in distinguished open subsets endowed with charts $(h,U)$ and $(h',U')$. We define $(\phi,\psi)$ by $ h'\circ f\circ h^{-1}= (\phi ,\psi)$ on $h(K)$.

Let $\epsilon>0$. The following subset is an elementary open set of the topology:
\[\Omega:=\Big\{f'\in Mor^r_f (\mathcal{L} ,\mathcal{L}')\;:\;
 f'(K)\subset U', \; \mathrm{and \; s.t.\; if} \; \phi'\;\mathrm{is\; defined\; by}\; \quad\]\[h'\circ f'\circ h^{-1}=(\phi' ,\psi),
 \; \mathrm{we \; have}\; \max_{h(K)}\big(\sum_{s=1}^r\|\partial_x^s\phi-\partial_x^s\phi'\|\big)<\epsilon\Big\}.\]

For any  manifold $M$, each space $Im^r(\mathcal L,M)$, $Emb^r(\mathcal L,M)$ and $End^r(M)$ contains a unique equivalence class. We endow these spaces with the topology of there unique equivalence class.

In particular the topology on $C^r(M,M)=End^r(M)$ is the (classical) $C^r$-compact-open topology.

Given a lamination $(L, \mathcal L)$ $C^r$-immersed by $i$ into a Riemannian manifold $(M,g)$, we define the $C^r$-strong topology on $Im^r(\mathcal L,M)$ by the following (partially defined) distance:
\[\forall (j,j')\in Im^r(\mathcal L,M),\quad d(j,j'):=\sup_{(x,u)\in T\mathcal L,\; \|u\|=1} \sum_{s=1}^r
d( \partial_{T_x\mathcal L}^sj(u^s), \partial_{T_x\mathcal L}^sj'(u^s)),\]
where $(L,\mathcal L)$ is endowed with the Riemannian distance $i^*g$ and $TM$ is endowed with the Riemannian distance induced by $g$.   
\subsection{Stratifications of laminations}\label{stra}
Throughout this section, all laminations are supposed to be of class $C^r$.   
\subsubsection{Stratifications}
The concept of stratification occurs in several mathematical fields. The definition depends on the fields and on the authors. One of the most general definitions was formulated by J. Mather \cite{Ma}\footnote{In this article, the object of  the definition is called a prestratification. This corresponds in fact to the stratifications defined here.}
:
\begin{defi} A \emph{stratification} of  a metric space $A$ is a partition of $A$ into subsets, called \emph{strata}, that satisfy the following conditions:
\begin{enumerate}
\item each stratum  is locally closed, {\it i.e}, it is equal to the intersection of a closed subset with an open subset of $A$,
\item the partition $\Sigma$ is locally finite,
\item (condition of frontier) for any pair of strata $(X,Y)\in \Sigma^2$ satisfying  $Y\cap cl(X)\not =\emptyset$, we have $Y\subset cl(X)$. We write $Y\le X$ and $X$ is said to be \emph{incident} to $Y$.
\end{enumerate}
\end{defi}
 The pair $\chi =(A,\Sigma)$ is called  the \emph{ stratified space of support $A$ and of stratification $\Sigma$}.

\begin{propr}
The stratification equipped with the relation $\le$ is a partially ordered set.\end{propr}
\begin{proof}
The reflexivity and the transitivity are clear. To show the antisymmetricity, we choose two strata $(X,Y)\in \Sigma$ such that $X\le Y$ and $Y\le X$.  This means that $X\subset cl(Y)$ and $Y\subset cl(X)$, hence
$cl(X)$ is equal to $cl(Y)$. As $X$ and $Y$ are locally closed, these two strata cannot be disjoint and as $\Sigma$ is a partition of $A$, these two strata are equal.\end{proof}
 
\subsubsection{Analytic and differentiable stratifications}
Among the fields where the stratifications occur, we can cite analytic geometry and differential geometry.  Let us describe our definitions of stratifications for both of these fields:
\begin{itemize}
\item In analytic geometry, we recall that an analytic variety is the set of zeros of an analytic map from $\mathbb C^n$ to $\mathbb C^m$.

 We define an \emph{analytic stratified space} as a space whose support is an analytic variety and whose strata are analytic manifolds such that:
\[\forall (X,Y)\in \Sigma^2,\quad \mathrm{if}\; X\le Y\; \mathrm{then}\; dim(X)\le dim(Y).\]

A such definition is equivalent to the definition of  H. Whitney \cite{W1} of stratifications in the context of analytic geometry.

\item In differential geometry, following the work of J. Mather \cite{Ma} and R. Thom \cite{Th}, C. Murolo and D. Trotman \cite{MT} define stratifications as stratified spaces whose strata, endowed with the topology induced by the support, are connected manifolds that satisfy:
 \[\forall (X,Y)\in \Sigma^2,\quad \mathrm{if}\; X< Y \; \mathrm{ then}\; dim(X)< dim(Y).\]

We define such a stratified space a \emph{differentiable stratified space}.
 Such spaces occur notably in singularity theory or in the study of zeros of a generic map (\cite{Th}, \cite{Ma}).
\end{itemize}

\subsubsection{Stratifications of laminations}
  Similarly, we introduce the concept of \emph{$(C^r)$-laminar stratified space}:
a stratified space  $(A,\Sigma)$ whose strata, endowed with the topology induced by $A$, are laminations that stratify:
\[\forall(X,Y)\in \Sigma^2,\quad \mathrm{if} \;X\le Y\; \mathrm{ then}\; dim(X)\le dim(Y).\]

 As an analytic or differentiable stratified space have a canonical structure of a laminar space, we will abuse of language by using \emph{stratified space} to refer to \emph{a laminar stratified space}, in the rest of this work.

 In general, differentiable stratified spaces are used with extra regularity conditions: either by supposing them embedded, with a certain regularity, into a manifold, or by endowing them with a supplementary geometric structure.

 Let us begin by introducing the embedding which allows us to introduce some examples of stratified spaces. In section \ref{structures} we shall introduce a supplementary geometric structure, the trellis of laminations which exists on certain stratified spaces.

 An \emph{embedding} $i$ of a stratified space $(A,\Sigma)$ into a manifold is a homeomorphism onto its image, whose  restriction to each stratum is an embedding of laminations.
 
 The embedding $i$ is \emph{$a$-regular} if, for all strata $(X,Y)\in \Sigma^2$ such that $X\le Y$, and for any sequence $(x_n)_n\in Y^\mathbb N$ which converges to a point $x\in X$ and such that the sequence of subspaces $(Ti(T_{x_n} Y))_{n}$ converges to a subspace $E$, then $Ti(T_xX)$ is included in $E$.
 
 When a stratified space is embedded, we often identify the stratified space with its image by the embedding. If the embedding is $a$-regular, we will say that (in this identification) the \emph{stratification (of laminations) is $a$-regular.}

 The $a$-regularity condition is due to H. Whitney, who showed that every analytic variety supports an  $a$-regular (analytic) stratification \cite{W2}. This definition is also standard in the study of differentiable stratified spaces.

\begin{exem} Given a submanifold with boundary, the connected components of the boundary and the interior of the submanifold endowed with their canonical manifold structure define a ($a$-regular) (differentiable) stratification.\end{exem}

\begin{exem} The set $\{0\}\times\mathbb R\cup \mathbb R\times\{0\}$ supports a (differentiable) stratification with two strata: a first stratum being $\{0\}$ and the second being $\{0\}\times\mathbb R^*\cup \mathbb R^*\times\{0\}$. This stratified space is canonically ($a$-regularly) embedded into $\mathbb R^2$.\end{exem}
\begin{exem} Given a manifold $M$ and a compact subset $K$ with empty interior, the set $K$ endowed with its $0$-dimensional lamination structure and the set $M\setminus K$ endowed with its canonical manifold structure define an $a$-regular laminar stratification  on $M$.\end{exem}

\begin{exem}\label{example:for:Sigmap} Let $\mathbb S^1$ be the unit circle and let $\mathbb D$ be the unit open disk of the complex plane $\mathbb C$. Let $A$ be the topological subspace $cl(\mathbb D)\times \{1\}\cup \{1\}\times \mathbb S^1$ of $\mathbb C^2$. We stratify $A$ by four strata: the first is the 2-dimensional lamination supported by $\mathbb D\times \{1\}$,
the others are $0$-dimensional and supported by respectively $\mathbb S^1\times \{1\}\setminus \{(1,1)\}$, $\{1\}\times \mathbb S^1\setminus \{(1,1)\}$ and \{(1,1)\}. This stratified space is canonically ($a$-regularly) embedded in $\mathbb C^2$.\end{exem}

\begin{propr}\label{Strat:AxiomA} Let $f$ be a diffeomorphism that satisfies axiom $A$ and the strong transversality condition\footnote{See section \ref{axiom A} for the definition.}. Let $(\Lambda_i)_i$ be the spectral decomposition of the nonwandering set $\Omega$. Let $W^s(\Lambda_i)$ be the union of the stable manifolds of $\Lambda_i$'s points. Then the family $(W^s(\Lambda_i))_i$ defines a stratification of laminations on $M$, where the leaves of $W^s(\Lambda_i)$ are stable manifolds.
\end{propr}
\begin{proof}
We recall that each basic piece $\Lambda_i$ is a hyperbolic compact subset disjoint from the other basic pieces. As the periodic points are dense in the nonwandering set, there exists a local product structure on $\Lambda_i$. It follows from proposition 9.1 of \cite{Sh} that:
  \[W^s(\Lambda_i):=\cup_{x\in \Lambda_i} W^s(x)=\{x\in M| d(f^n(x),\Lambda_i)\rightarrow 0,\; n\rightarrow \infty\}.\]
 As the subset $L(f)$ of the accumulation points of orbits is included in the nonwandering set, it follows from lemma 2.2 of \cite{Sh} that the manifold $M$ is the disjoint union of the subsets $(W^s(\Lambda_i))_i$.
 
 Let us now show the frontier condition:
 \[cl\big(W^s(\Lambda_i)\big)\cap W^s(\Lambda_j)\not= \emptyset \Rightarrow cl\big(W^s(\Lambda_i)\big)\supset W^s(\Lambda_j).\]
 
 First we recall that if the $W^s(\Lambda_k)$ intersects $W^u(\Lambda_l)$  by the strong transversality condition and the transitivity of $\Lambda_k$, the closure of $W^s(\Lambda_k)$ contains $W^s(\Lambda_l)$. We write then $\Lambda_k\prec \Lambda_l$.  Moreover, the strong transversality condition implies the nocycle condition which states that $\prec$ is an order on $(\Lambda_l)_l$.
 
 We now suppose that the closure of $W^s(\Lambda_i)$ intersects $W^s(\Lambda_j)$. If $i$ is not equal to $j$, it follows from lemmas 1 and 2 P.10 of \cite{Sh} that the closure of $W^s(\Lambda_i)$ intersects $W^u(\Lambda_i)\setminus \Lambda_i$. Let $x$ be a point which belongs to this intersection. As $(W^s(\Lambda_k))_k$ covers $M$, there exists $j_1$ such that $x$ belongs to $W^s(\Lambda_{j_1})$. By the above reminder, the closure of $W^s(\Lambda_{j_1})$ contains $W^s(\Lambda_{j})$. Moreover, the closure of $W^s(\Lambda_{i})$
intersects $W^s(\Lambda_{j_1})$. And so on, we can continue to construct $(\Lambda_{j_k})_k$. As the family $(\Lambda_i)_i$ is finite and there are nocycle, the family $(\Lambda_{j_k})_k$ is finite. Thus, we obtain:
\[cl(W^s(\Lambda_{i}))=cl(W^s(\Lambda_{j_n}))\supset \cdots \supset cl(W^s(\Lambda_{j_1}))\supset W^s(\Lambda_{j}).\]
This proves the frontier condition.

To finish, it only remains to show the existence of the canonical lamination structure on each $W^s(\Lambda_i)$ (which implies that $W^s(\Lambda_i)$ is locally closed). It follows from the nocycle condition and the fact that $L(f)$ is included in $\Omega$ that there exists an adapted filtration to $(\Lambda_i)_i$ (see theorem 2.3 of \cite{Sh}). In other words, there exists an increasing family of compact subsets $(M_i)_i$ such that:
  \[f(M_i)\subset int(M_i)\quad \mathrm{and}\quad \Lambda_i= \cap_{n\in \mathbb Z} f^n(M_i\setminus M_{i-1}).\]
  Let $U:= M_i\setminus M_{i-1}$.  As each point $x\in f^{-n}(M_{i-1})$ has its orbit which will belong eventually to $M_{i-1}$ (which does not contain $\Lambda_{i}$), we have:
  \[W^s(\Lambda_i)\cap U= W^s(\Lambda_i)\cap M_i\setminus \cup_{n\ge 0}f^{-n}(M_{i-1}).\]
  Thus, $W^s(\Lambda_i)\cap U$ is $f$-stable:
  \[f(W^s(\Lambda_i)\cap U)\subset W^s(\Lambda_i)\cap U.\]
  By replacing $(M_i)_i$ by $(f^n(M_i))_i$, the set $W^s(\Lambda_i)\cap U$ may be an arbitrarily small neighborhood of $\Lambda_i$. Hence, for any small $\epsilon$ and small $U$, we way suppose that each point $x$ of $W^s(\Lambda_i)\cap U$ can be $\epsilon$-shadowed by a point $y$ of $\Lambda_i$ ($\Lambda_i$ is endowed with a local product structure). It follows that $x$ belongs to the local stable manifold $W^s_\epsilon(y)$ of $y$. Consequently:
 \[W^s(\Lambda_i)\cap U\subset \cup_{y\in \Lambda_i} W^s_\epsilon(y).\]
 In other words, $W^s(\Lambda_i)$ does not auto-accumulate. Thus, example \ref{lamhyper} shows that $W^s(\Lambda_i)$ is endowed with a lamination structure whose leaves are the stable manifolds of points of $\Lambda_i$.
\end{proof}

\begin{propr}\label{exprod} Let $(A_1,\Sigma_1)$ and $(A_2,\Sigma_2)$ be two stratified spaces. Then the pair $(A_1\times A_2, \Sigma_1\times \Sigma_2)$, where:
 \[\Sigma_1\times \Sigma_2=\{X_1\times X_2\; ;\; X_1\in\Sigma_1\; \mathrm{and}\; X_2\in\Sigma_2\},\]
is a stratified space of support $A_1\times A_2$.

 Moreover, if $p_1$ and $p_2$ are embeddings of $(A_1,\Sigma_1)$ and $(A_2,\Sigma_2)$ into manifolds $M_1$ and $M_2$ respectively, the map $p:=(p_1,p_2)$ is an embedding of $(A_1\times A_2, \Sigma_1\times \Sigma_2)$ into $M_1\times M_2$. This embedding is $a$-regular if and only if $p_1$ and $p_2$ are $a$-regular.\end{propr}

\begin{proof}
 To check that $\Sigma_1\times \Sigma_2$ defines a stratified space is elementary:
for all $(X_1\times X_2,Y_1\times Y_2)\in (\Sigma_1\times \Sigma_2)^2$,  we have
\[X_1\times X_2\cap cl(Y_1\times Y_2)\not =\emptyset\Leftrightarrow  X_1\cap cl(Y_1)\not= \emptyset\; \mathrm{and}\; X_2\cap cl(Y_2)\not =\emptyset\]
Then, for each $i\in\{1,2\}$, we have $X_i\subset cl(Y_i)$ and $dim(X_i)\le dim (Y_i)$.
This implies that $X_1\times X_2\subset cl(Y_1\times Y_2)$ and $dim(X_1\times X_2)\le Y_1\times Y_2$.

The proof of the statement on $a$-regularity is left to reader.\end{proof}

 Let $(A,\Sigma)$ be a stratified space and $U$ an open subset of $A$. The set $\Sigma_{|U}$ of the restrictions of strata $X\in \Sigma$ that intersect $U$, to $U\cap X$, defines a stratification of laminations on $U$.

\subsubsection{Stratified morphisms} \label{Strat:mor} Let $(A,\Sigma)$ and $(A',\Sigma')$ be two $C^r$-stratified spaces.

 A continuous map $f$ from $A$ to $A'$ is a \emph{$(C^r)$-stratified morphism} (resp. \emph{stratified immersion}) if  each stratum $X\in\Sigma$ is sent into a stratum $X'\in \Sigma'$ and the restriction $f_{|X}$ is a $C^r$-morphism (resp. an immersion) from the lamination $X$ to $X'$. We will also say that $f$ is a $(C^r)$-morphism (resp. immersion) from $(A,\Sigma)$ to $(A',\Sigma')$.

 In the particular case of differentiable stratified spaces, we are coherent with the usual definition of stratified morphisms.
 
  A \emph{stratified endomorphism} $(A,\Sigma)$ is a stratified morphism which preserves each stratum.

   We denote respectively by $Mor^r(\Sigma,\Sigma')$, $Im^r(\Sigma,\Sigma')$ and $End^r(\Sigma)$ the set of $C^r$-stratified morphisms, immersions and endomorphisms.

    Two stratified morphisms $f$ and $\hat f$ are said to be {\it equivalent} if they send each stratum $X\in \Sigma$ into a same stratum $X'\in \Sigma'$ and if their restrictions to $X$ are equivalent as morphisms from the lamination  $X$ to $X'$. We denote by $Mor_f^r (\Sigma,\Sigma')$ the equivalence classes of $f$ endowed with the topology induced by the following product:
\[C^0(A,A')\times \prod_{X\in \Sigma,\; f(X)\subset X'\in \Sigma'} Mor_{f|X}^r(X,X')\]

 The aim of this work is to show the persistence of some $a$-regular normally expanded stratifications. However,
 the regularity of these stratifications is not sufficient to guarantee their persistence: there exist compact differentiable stratifications which are normally expanded but not persistent (we will give such a example in part \ref{cexp}).
We are going to introduce a stronger regularity condition: to support a trellis structure. In the differentiable stratified space, other authors have introduced other intrinsic conditions  (\cite{Ma}, \cite{Th}, \cite{MT}).

 \subsection{Structures of trellis of laminations}\label{structures}
Throughout this section, unless stated otherwise, all laminations, laminar stratified spaces and morphisms are supposed to be of class $C^r$, for $r\ge 1$ fixed. 
 
We need some preliminary definitions:

\begin{defi}[Coherence and compatibility of two laminations]
 Let $L_1$ and $L_2$ be two subsets of a metric space $L$, endowed with lamination structures denoted by $\mathcal L_1$ and $\mathcal L_2$ respectively. Let us suppose that, for example, the dimension of $\mathcal L_2$ is at least equal to
the dimension of $\mathcal L_1$.

 The laminations $\mathcal L_1$ and $\mathcal L_2$ are {\it coherent} if, for all $x\in L_1\cap L_2$, there exists a plaque of $\mathcal L_1$ containing $x$ and included in a plaque of $\mathcal L_2$.
 
  The laminations $\mathcal L_1$ and $\mathcal L_2$ are {\it compatible} if, for all  $x\in L_1\cap L_2$, the leaf of $\mathcal L_1$ containing $x$ is included in a leaf of $\mathcal L_2$.\end{defi}

\begin{defi}[Foliated lamination] Let $(L_1,\mathcal L_1)$ and $(L_2,\mathcal L_2)$ be two laminations of dimension $d_1\le d_2$ respectively. We will say that \emph{$\mathcal L_1$ is a foliation of $\mathcal L_2$} if $L_1=L_2$ and if, for every $x\in L_2$, there exists a neighborhood $U$ of  $x$ and a chart $(U,\phi)$ which belongs to $\mathcal L_1$ and to $\mathcal L_2$. This means that there exists open subsets $U_1$ and $U_2$ of $\mathbb R^{d_1}$ and $\mathbb R^{d_2-d_1}$  respectively, such that:
\[\Big(\phi\;:\; U\rightarrow U_1\times \underbrace{U_2\times \overbrace{T_2}^{\mathrm{Transversal\; space\; of}\; \mathcal L_2}}_{\mathrm{Transversal\; space\; of\; \mathcal L_1}}\Big)\in \mathcal L_1\cap \mathcal L_2.\]
We note that the laminations structure $\mathcal L_1$ and  $\mathcal L_2$ are then coherent.
\end{defi}

\begin{rema} If, in this definition, the lamination $\mathcal L_2$ is a manifold, then $\mathcal L_1$ is a (classical) $C^r$-foliation of dimension $d_1$ on this manifold.\end{rema}

\begin{rema} Unlike laminations, a $C^1$-foliation is not diffeomorphic to a $C^\infty$-foliation. For instance, the suspension of a Denjoy $C^1$-diffeomorphism defines a $C^1$-foliation which is not diffeomorphic to a $C^2$-foliation.\end{rema}

\begin{exem}
 Let $K$ be a locally compact set and let $L_1=L_2$ be the product $\mathbb R^{d_2}\times K$. Let $\mathcal L_2$ be the canonical lamination structure of dimension $d_2$ on $L_2$. Let $d_1$ be an integer less than $d_2$, let $\phi$ be a continuous map from $K$ to $Diff^r(\mathbb R^{d_2},\mathbb R^{d_2}))$ and let $\mathcal L_1$ be the lamination structure on $L_1$ whose leaves are:
\[\{\phi(k)(\mathbb R^{d_1}\times\{t\})\times \{k\},\; (k,t)\in K\times \mathbb R^{d_2-d_1}\}.\]
Then $\mathcal L_1$ is a $C^r$-foliation of $\mathcal L_2$.
\end{exem}

\begin{propr} Let $(L,\mathcal L)$ be a lamination immersed into a manifold $M$.
We identify $(L,\mathcal L)$ with its image in $M$.
Let $\mathcal F$ be a $C^r$-foliation on an open neighborhood of $L$, whose leaves are transverse to the leaves of $\mathcal L$.
Then the lamination on $L$, whose  plaques are the transverse intersection of plaques of $\mathcal L$ with plaques of $\mathcal F$, is a $C^r$-foliation of $\mathcal L$. Let us denote by $\mathcal L\pitchfork \mathcal F$ this lamination structure.
\end{propr}
\begin{proof}
As the statement is a locale property, it is sufficient to prove it in a neighborhood of every point $x\in L$.
Via a local chart of $\mathcal F$, we identify a neighborhood $U$ of $x$ to $\mathbb R^n$ and $\mathcal F$ to the
foliation  associated to the splitting $\mathbb R^{n-d}\times\mathbb R^d$, whose leaves are of dimension $d$. We may suppose that, in this identification, $T_x\mathcal L$ is the vectorial subspace $\mathbb R^{d'}\times \{0\}$ of $\mathbb R^n$, with $d'\ge n-d$. 

We may suppose $U$ small enough such that the intersection of the leaves of $\mathcal L$ with $U$ can be identified to a continuous family of $C^r$-graphs from $\mathbb R^{d'}$ to $\mathbb R^{n-d'}$. Let $(\rho_t)_{t\in T}$ be such a family of $C^r$-maps. We note that the following application:

\[\phi_0\;:\;\mathbb R^n\cap L\rightarrow \mathbb R^d\times T\]
\[(u,\rho_t(u))\mapsto (u,t)\]
is a chart of $\mathcal L$.

Thus, for all $t\in T$ and $v\in \mathbb R^{n-d}$, the intersection of the plaque of $\mathcal L$:
\[\big\{(u_1,u_2,\rho_t(u_1,u_2)):\; (u_1,u_2)\in \mathbb R^{n-d}\times \mathbb R^{d+d'-n}\big\}\]
 with the plaque $\{v\}\times \mathbb R^{d+d'-n}\times \mathbb R^{n-d'}$ of $\mathcal F$, is:
\[\big\{(v,u_2,\rho_t(u_1,u_2)):\; u_2\in \mathbb R^{d+d'-n}\big\}.\]

The chart $\phi$ sends this intersection onto $\{v\}\times \mathbb R^{d+d'-n}\times \{t\}$.

\[\mathrm{Let}\; \psi\; (v,u)\in \mathbb R^{n-d'}\times \mathbb R^{d+d'-n}\mapsto (u,v) \in \mathbb R^{d+d'-n}\times \mathbb R^{n-d'}.\]

 Finally, we define
\[\phi \;:\; \mathbb R^n\cap L\rightarrow \mathbb R^{d'+d-n}\times \mathbb R^{n-d}\times T\]
\[(u,\rho_t(u))\mapsto (\psi(u),t)\]
which is a chart of $\mathcal L$ and of $ \mathcal L\pitchfork\mathcal F$. We conclude that the lamination $\mathcal L\pitchfork \mathcal F$ is a foliation of the lamination $\mathcal L$.\end{proof}
 
\begin{defi}[Tubular neighborhood]
Let $(A,\Sigma)$ be a stratified space and let $X$ be a stratum of $\Sigma$. A {\it tubular neighborhood} of $X$ is a lamination $(L_X,\mathcal L_X)$ such that:\begin{itemize}
\item the support $L_X$ is a neighborhood of $X$ included in strata incident to $X$,
\item the leaves of the stratum $X$ are leaves of $\mathcal L_X$,
\item the lamination $(L_X,\mathcal L_X)$ is coherent with the other strata of  $\Sigma$.
\end{itemize}
\end{defi}
\begin{defi}[Trellis structure]
A {\it $(C^r)$-trellis (of laminations) structure} on a stratified space $(A,\Sigma)$ is a family of tubular neighborhoods
$\mathcal T=(L_X,\mathcal L_X)_{X\in \Sigma}$ satisfying, for all $X\le Y$, that the lamination $\mathcal L_{X|L_X\cap L_Y}$ defines a $C^r$-foliation of the lamination $\mathcal L_{Y|L_X\cap L_Y}$.\end{defi}

Throughout the rest of this chapter, all the trellis structures are supposed to be of class $C^r$. We will not mention their regularity.

\begin{rema} If $(A,\Sigma)$ is a lamination $(L,\mathcal L)$, then $(L,\mathcal L)$ is also the unique trellis structure on the stratified space $(A,\Sigma)$.\end{rema}

\begin{exem}\label{carre} Let $(X_0,X_1,X_2)$ be the canonical stratification on the filled square: the lamination $X_0$ is the subset of vertexes, the lamination $X_1$ is one-dimensional and supported by the edges and finally the lamination $X_3$ is two-dimensional and supported by the interior.

Let $\mathcal L_{X_0}$ be the $0$-dimensional lamination structure on a neighborhood of vertexes. Let $\mathcal L_{X_1}$ be the $1$-dimensional lamination structure on four disjoints neighborhoods of each edge, whose leaves are parallel to the associated edge. Then the family $((L_{X_0},\mathcal L_{X_0}), (L_1,\mathcal L_1),X_2)$ forms a trellis structure on $(A,\Sigma)$. We illustrate this structure in figure \ref{coupe treillis}.

\begin{figure}[h]
    \centering
\includegraphics[width=7cm]{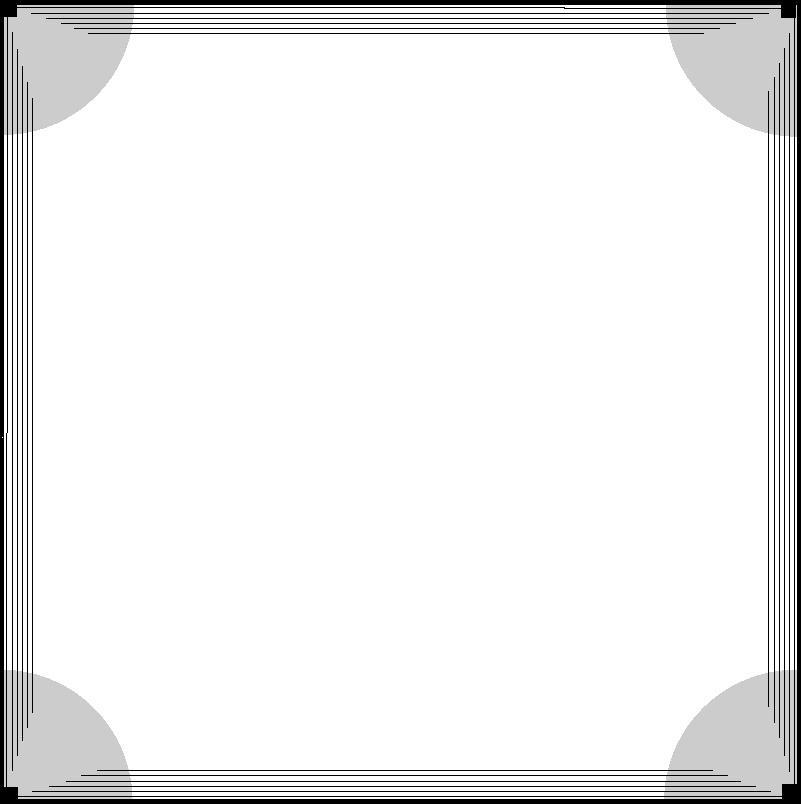}
    \caption{A trellis structure on the filled square}
    \label{coupe treillis}
\end{figure}
\end{exem}

\begin{exem}\label{previana} Let $I$ be an open interval of $\mathbb R$. We denote $\partial I$ its boundary. Let $C$ be the closed and filled cylinder in $\mathbb R^3$ defined by
\[C:=\{(x,y,z)\in\mathbb R^3; x^2+y^2\le 1\; \mathrm{and}\; z\in cl(I)\}.\]
The cylinder $C$ supports the stratification of laminations consisting of the following strata:\begin{itemize}
 \item the $0$-dimensional lamination $X_0$ supported by $\mathbb S^1\times \partial I$,
 \item the $1$-dimensional lamination $X_1$ supported by $\mathbb S^1\times I$ whose leaves are vertical,
 \item the $2$-dimensional lamination $X_2$ supported by $\mathbb D\times \partial I$,
 \item the $3$-dimensional lamination $X_3$ supported by the interior of $C$.
 \end{itemize}

  This stratification is canonically  $a$-regularly embedded into $\mathbb R^3$.

 Let us construct a trellis structure on this stratified space.
 Let $L_{X_0}$ be an open neighborhood of $X_0$ in $C$. We endow $L_{X_0}$ with the 0-dimensional lamination structure $\mathcal L_{X_0}$. Let $L_{X_1}$ and $L_{X_2}$ be two disjoint open subsets of $C\setminus X_0$ containing respectively  $X_1$ and $X_2$. We endow $L_{X_1}$ with the 1-dimensional lamination structure  $\mathcal L_{X_1}$ whose leaves are vertical. We endow $L_{X_2}$ with the 2-dimensional lamination structure $\mathcal L_{X_2}$ whose leaves are horizontal. Finally, we define $(L_{X_3},\mathcal L_{X_3})$ as
equal to the lamination $X_3$. We notice that $\mathcal T:=(L_{X_i},\mathcal L_{X_i})_{i=0}^3$ is a trellis structure on the stratified space $(C,\Sigma)$.

The figure \ref{coupe treillis} also illustrates a section of such a trellis structure by a plane containing the axis $(O_z)$.
\end{exem}

\begin{exem} With conventions of figure \ref{coupe treillis}, figure \ref{treilliscube} gives a trellis structure  on the canonical stratification of a cube, that is, the simplicial splitting into vertex, edges and faces.

\begin{figure}
    \centering
        \includegraphics{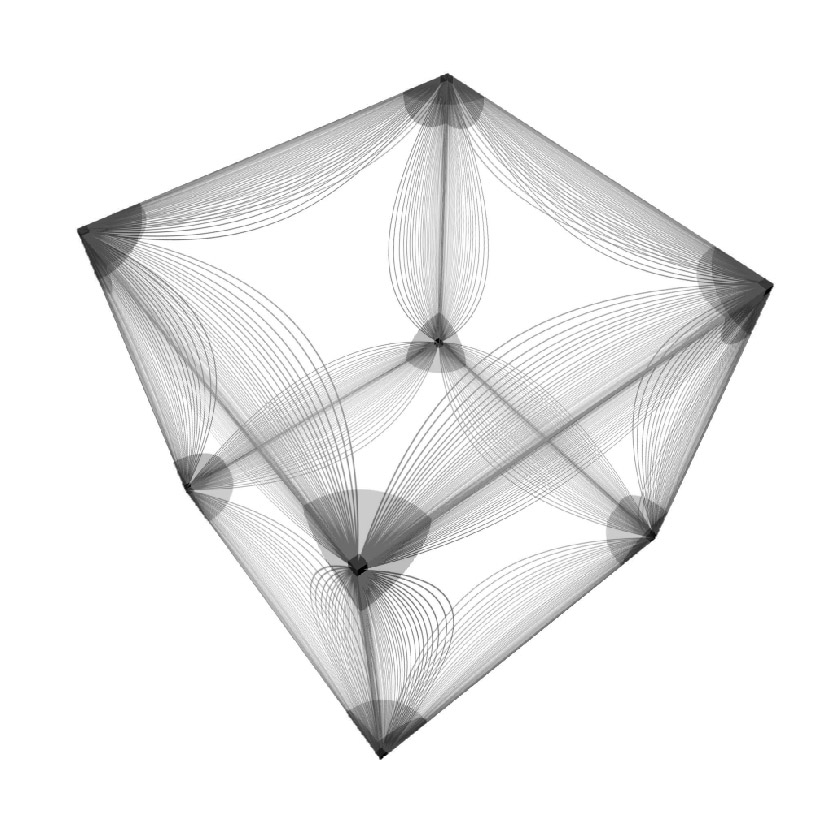}
    \caption{A trellis structure on a cube.}
    \label{treilliscube}
\end{figure}
\end{exem}

 Let $\mathcal T$ be a trellis structure on a stratified space $(A,\Sigma)$ and let $U$ be an open subset of $A$. Then the family of restrictions of the laminations $(L,\mathcal L)\in \mathcal T$ to $U$ makes up a trellis structure on $(U,\Sigma_{|U})$. We denote by $\mathcal T_{|U}$ this trellis structure.

\begin{rema} Given a trellis structure on a stratified space $(A,\Sigma)$, the foliation condition implies the coherence between the tubular neighborhoods.
\end{rema}

The following property implies in particular that every tubular neighborhood is compatible with every stratum.
\begin{propr}\label{coher}
Let $(A,\Sigma)$ be a stratified space and let $(L,\mathcal L)$ be a lamination such that $L$ is included in the union of strata of dimension at least equal to the dimension of $\mathcal L$. If the lamination $\mathcal L$ is coherent with each stratum of $\Sigma$, then for each stratum $X\in \Sigma$, the set $X\cap L$ is $\mathcal L$-admissible.
\end{propr}
\begin{proof}
The coherence implies that the leaves of $\mathcal L$ intersect the leaves of each stratum of $\Sigma$ in an open set.
This is why the partition $\Sigma$ of $A$ induces a partition of each leaves into open sets. By connectivity, each  leaf is contained in a unique leaf of a stratum of $\Sigma$. This implies that the subset $X\cap L$ is $\mathcal L$-saturated for every stratum $X\in \Sigma$. As the intersection of two locally compact subsets is locally compact, the subset  $X\cap L$ is $\mathcal L$-admissible.\end{proof}

\begin{propr}\label{locpt}
Let $(A,\Sigma)$ be a stratified space which admits a trellis structure. Then $A$ is locally compact.\end{propr}
\begin{proof} The family of the supports of the tubular neighborhoods is an open covering of $A$.
As each of these supports is locally compact, it follows that $A$ is locally compact.\end{proof}

 Let $M$ be a manifold and $(A,\Sigma)$ be a stratified space endowed with a $C^r$-trellis structure $\mathcal T$.
A stratified embedding $i$ from $(A,\Sigma)$ into $M$ is \emph{$(C^r)-\mathcal T$-controlled} if the restriction of $i$ to each tubular neighborhood $(L_X,\mathcal L_X)$ of $\mathcal T$ is a $C^r$-embedding of the lamination $\mathcal L_X$ into $M$. In other words, a $\mathcal T$-controlled $C^r$-embedding of $(A,\Sigma)$ in $M$ is a homeomorphism onto its image such that the $r$-first partial derivatives of $p$ along the leaves of each tubular neighborhood $(L_X,\mathcal L_X)$ exist, are continuous on $L_X$ and the first one is injective.

The figure \ref{cylindreconv} represents a  $\mathcal T$-controlled embedding into $\mathbb R^3$ of the stratified space defined in example \ref{previana} (not equal to the canonical inclusion).

\begin{figure}[h]
    \centering
\includegraphics[width=7cm]{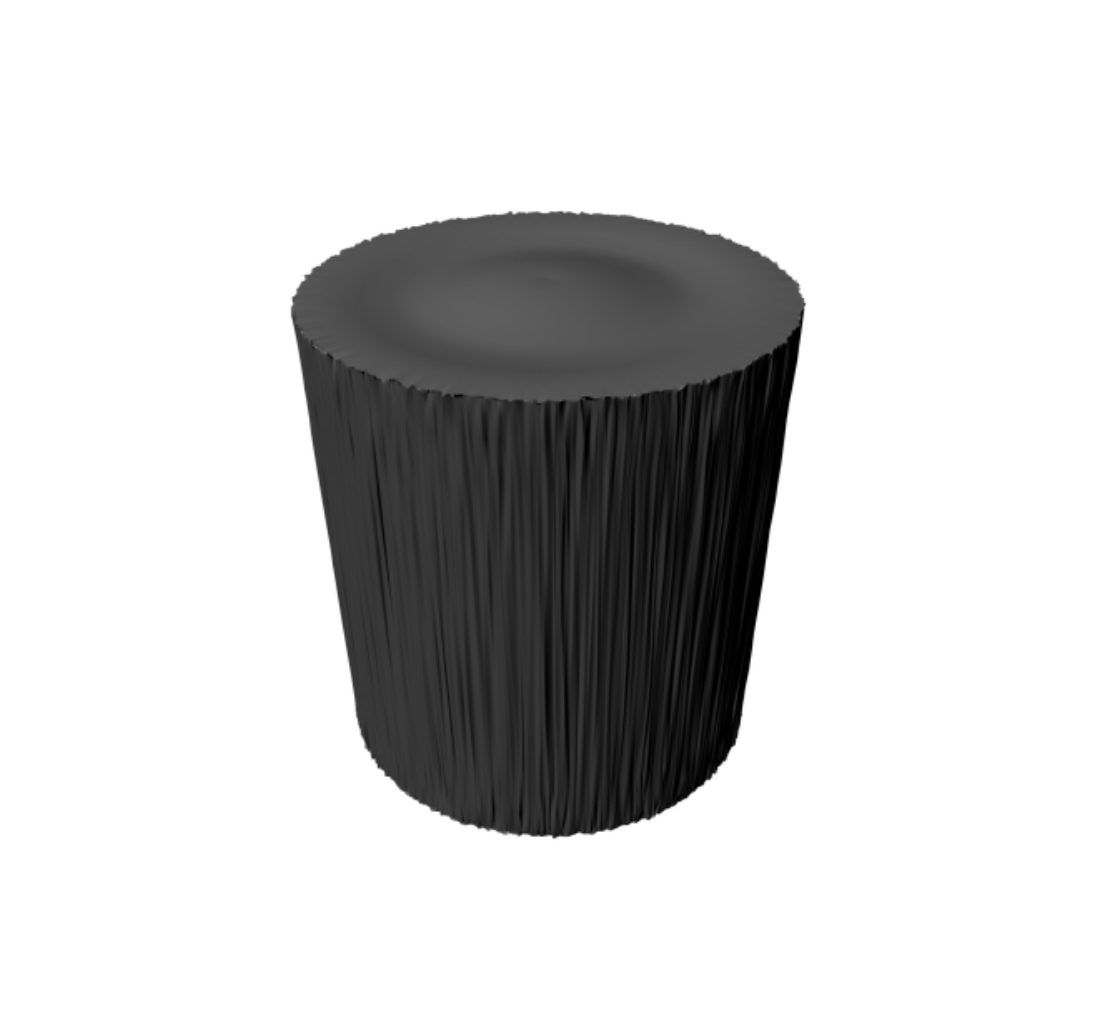}
    \caption{Controlled embedding of an exotic stratification on the cylinder.}
    \label{cylindreconv}
\end{figure}

 A controlled embedding into the space $\mathbb R^3$ of the cube, endowed with its trellis structure defined in figure \ref{treilliscube} is represented in figure \ref{cube::plonge}.
\begin{figure}[h]
    \centering
        \includegraphics[width=7cm]{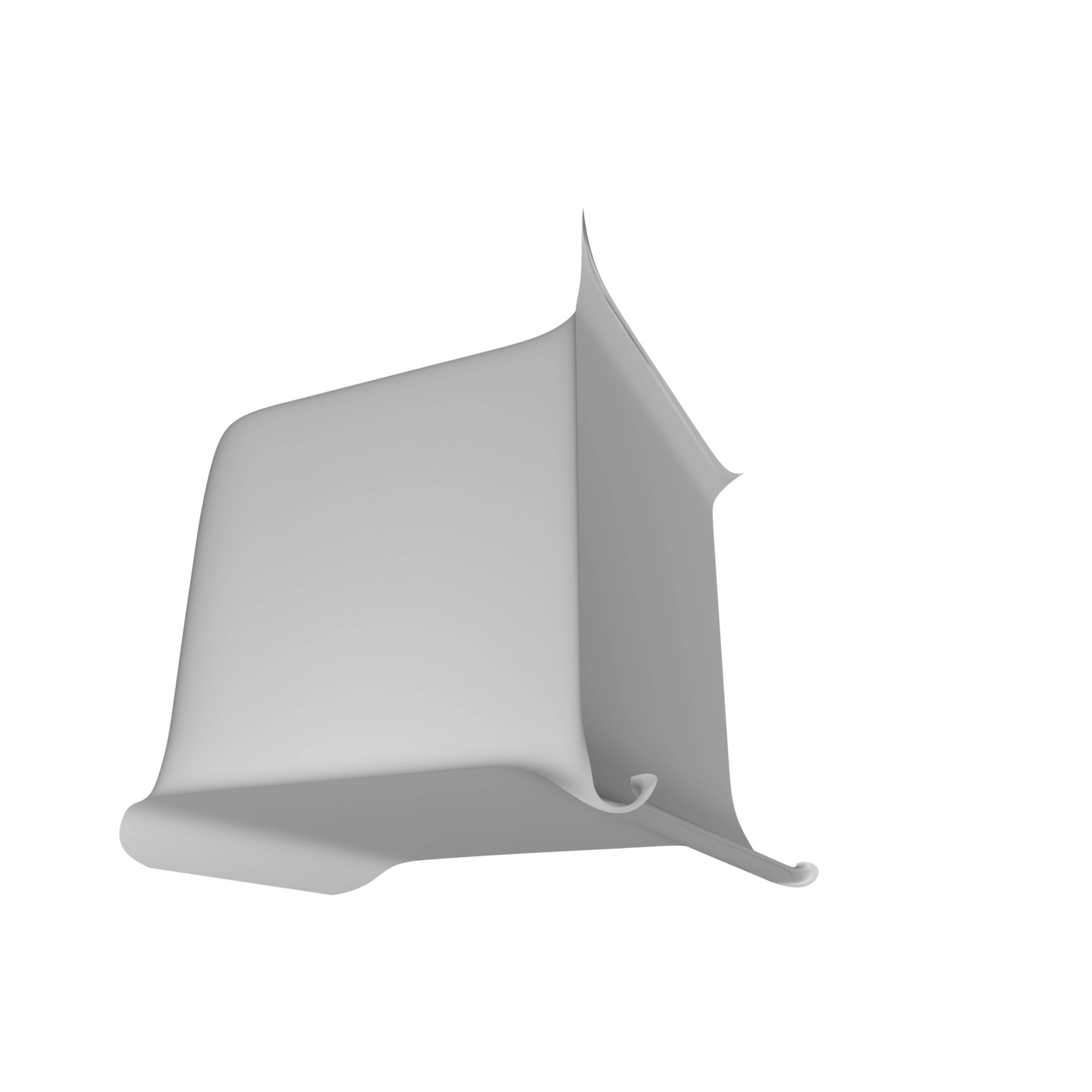}
    \caption[width=7cm]{Controlled embedding of a cube.}
    \label{cube::plonge}
\end{figure}

\begin{propr}\label{mora} Let $M$ be a manifold and let $(A,\Sigma)$ be a stratified space endowed with a trellis structure $\mathcal T$. Then any $\mathcal T$-controlled embedding  is $a$-regular.
\end{propr}
\begin{proof}
Let $X\le Y$ be two strata of $\Sigma$. Let $(x_n)_n\in Y^\mathbb N$ be a sequence which converges to $x\in X$. Then, for $n$ large enough, the point $x_n$ belongs to $L_X$ and $Ti(T_{x_n}Y)$ contains $Ti(T_{x_n}\mathcal L_{X})$ which converges to $Ti(T_{x}X)$.\end{proof}

We will remark in sections \ref{desstruc} and \ref{cexp}, that there exist stratified spaces which cannot support any trellis structure. However, the following proposition gives sufficient conditions for stratified space to admit a  trellis structure.

\begin{prop}\label{treillis:sur:produit}
 Let $(A,\Sigma)$ and $(A',\Sigma')$ be two stratified spaces. If each of these stratified spaces admits a trellis structure, then there exists a trellis structure on the product stratified space $(A\times A',\Sigma\times \Sigma')$.\end{prop}
 \begin{proof} We have already seen in property \ref{exprod} that $(A\times A',\Sigma\times \Sigma')$ is a stratified space whose partial order $\le $ on $\Sigma\times \Sigma'$ satisfies 
\[\forall X\times X' \in \Sigma\times \Sigma',\;\forall Y\times Y' \in \Sigma\times \Sigma',\quad (X\times X'\le Y\times Y')\Leftrightarrow (X\le Y\;\mathrm{ and}\; X'\le Y').\]

 We now apply the following lemma for the stratified space $(A\times A',\Sigma\times \Sigma)$ equipped with this partial order.
\label{part<}
\begin{lemm}\label{<} For every stratified space $(A,\Sigma)$, there exists a family of subsets  $(W_X)_{X\in \Sigma}$ such that $W_X$ is an open neighborhood of $X$ and intersects $W_{Y}$ if and only if $X$ and $Y$ are incident.
\end{lemm}
\begin{proof}  Let $X$ be a stratum of $\Sigma$ and let $\chi$ be the subset of $\Sigma$ consisting of strata which are not comparable to $X$. We note that 
\[X\cap cl\big(\cup_{Y\in \chi} Y\big)=X\cap \big(\cup_{Y\in \chi} cl(Y)\big)=\cup_{Y\in \chi}\big(X\cap  cl(Y)\big)=\emptyset.\]
Moreover, for any point $x\in X$, the distance between $x$ and $\cup_{Y\in \chi} Y$ is positive.
We define 
 \[W_X:=\bigcup_{x\in X} B\left(x,\frac{d(x,\cup_{Y\in \chi} Y)}{2}\right).\]
We remark that the open set $W_X$ is a neighborhood of $X$. Let $Y$ be a stratum of $\chi$; we denote by $\Upsilon$ the subset of $\Sigma$ which is not comparable to $Y$. Let $x\in X$ and $y\in Y$; we have then $x\in \cup_{Z\in \Upsilon} Z$ and $y\in \cup_{Z\in \chi} Z$;
this implies 
\[B\left(x,\frac{d(x,\cup_{Z\in \chi} Z)}{2}\right)\cap B\left(y,\frac{d(y,\cup_{Z\in \Upsilon} Z)}{2}\right)=\emptyset.\]
Consequently $W_X$ and $W_Y$ are disjoint.
\end{proof}

By using this lemma, we can define an open family $(W_{X\times X'})_{(X,X')\in \Sigma\times \Sigma'}$ satisfying 
\[W_{X\times X'}\supset X\times X',\]
\[W_{X\times X'}\cap W_{Y\times Y'}\not= \emptyset\Rightarrow X\times X'\le Y\times Y'\;\mathrm{or}\; X\times X'\ge Y\times Y'.\]

 We denote by $(L_X,\mathcal L_X)_{X\in \Sigma}$ and $(L_{X'},\mathcal L_{X'})_{X'\in \Sigma'}$ the trellis structures on the stratified spaces $(A,\Sigma)$ and $(A',\Sigma')$.

 Let $L_{X\times X'}$ be the open neighborhood $(L_{X}\times L_{X'})\cap W_{X\times X'}$ of $X\times X'$, that we endow with the lamination structure $\mathcal L_{X\times X'}:=(\mathcal L_{X}\times \mathcal L_{X'})_{|L_{X\times X'}}$.
 
 As $X$ and $X'$ are restrictions of respectively $(L_X, \mathcal L_{X})$ and $(L_{X'}, \mathcal L_{X'})$ to  admissible subsets, the stratum $X\times X'$ is the restriction of $\mathcal L_{X\times X'}$ to an admissible subset.
 
 Moreover, if $L_{X\times X'}$ and $L_{Y\times Y'}$ have a non-empty intersection, then
$W_{X\times X'}$ and $W_{Y\times Y'}$ have also a non-empty intersection. Therefore,  $X\times X'$ and $Y\times Y'$ are  comparable. We suppose for instance that $X\times X'\le Y\times Y'$. This is equivalent to suppose that $X\le Y$ and $X'\le Y'$. The lamination $(L_X\cap L_{Y},\mathcal L_{X|L_X\cap L_Y})$ is a $C^r$-foliation of the lamination $(L_X\cap L_Y,\mathcal L_{Y|L_X\cap L_Y})$ and
$(L_{X'}\cap L_{Y'},\mathcal L_{X'|L_{X'}\cap L_{Y'}})$ is a $C^r$-foliation of the lamination $(L_{X'}\cap L_{Y'},\mathcal L_{Y'|L_{X'}\cap L_{Y'}})$. As the product of  $C^r$-foliated laminations is a $C^r$-foliated lamination, the restriction $(L_{X\times X'}\cap L_{Y\times Y'},\mathcal L_{X\times X'|L_{X\times X'}\cap L_{Y\times Y'}})$ is a foliation of the lamination of $(L_{X\times X'}\cap L_{Y\times Y'},\mathcal L_{Y\times Y'|L_{X\times X'}\cap L_{Y\times Y'}})$.

 Therefore $\mathcal T_{prod}:=(L_{X\times X'},\mathcal L_{X\times X'})_{X\times X'\in \Sigma\times \Sigma'}$ is a  trellis structure on the product stratified space.\end{proof}
 
\label{Tprod}

\subsubsection{Structure of the union of strata of the same dimension}

The following property sheds light on the geometry of the union of strata of the same dimension.
\begin{propr}\label{arbo}
Let $(A,\Sigma)$ be a stratified space endowed with a trellis structure $\mathcal T$. Let $(d_p)_{p\ge 0}$ be the strictly increasing sequence of the different dimensions of strata of $\Sigma$.

Then, for each $p\ge 0$, we have:
\begin{enumerate}
\item The union of strata of dimension $d_p$ constitutes a lamination $X_p$. Every stratum of dimension $d_p$ is the restriction of $X_p$ to an admissible subset.
\item The union of tubular neighborhoods of strata of dimension $d_p$ forms a lamination $(L_p,\mathcal L_p)$.
\item $X_p$ is the restriction of $\mathcal L_p$ to an admissible subset.
\item For all $q\le p$, the subset $cl(X_p)\cap X_q$ is an $X_q$-admissible subset.
\item The subset $cl(X_p)$ is included in $\cup_{q\le p} X_q$.
\end{enumerate}
\end{propr}
\begin{proof}

 2) Let $\Sigma_p\subset \Sigma$ be the subset of the $d_p$-dimensional strata. For all strata $X$, $Y\in \Sigma_p$, the lamination $\mathcal L_{X|L_X\cap L_Y}$ is a $C^r$-foliation of $\mathcal L_{Y|L_X\cap L_Y}$ of codimension 0. Then the restrictions of the laminations $\mathcal L_X$ and $\mathcal L_Y$ to $L_X\cap L_Y$ are generated by a same atlas and so are equal. It follows from property \ref{cuplam} that the set of the charts of $\mathcal L_X$ for all $X\in \Sigma_p$ generates a lamination
structure $\mathcal L_p$ on  $L_{p}:= \cup_{X\in \Sigma_p} L_X$.

 1)-3) First of all, by local finiteness of $\Sigma$, the set $X_p$ is locally compact. Moreover, all the tubular neighborhoods of strata that belong to $\Sigma_p$ is coherent with the strata of $\Sigma$ and their union constitutes an open covering of $L_p$. So, the lamination $\mathcal L_p$ is coherent with $\Sigma$.
The property \ref{coher} implies that for $X\in \Sigma_p$ the support of $X=X\cap L_p$ is $\mathcal L_p$-admissible.
Thus, a leaf of $\mathcal L_p$ which intersects $X$ is contained in $X$, and so equal to a leaf of $X$. Since $L_p$ contains $X_p$, $X_p$ is canonically endowed with the lamination structure $\mathcal L_{p|X_p}$ and 1)- 3) are satisfied.

 4) The frontier condition implies that 
\[cl(X_p)\cap X_q=\bigcup_{X\in \Sigma_p,\; Y\in \Sigma_q}cl(X)\cap Y=\bigcup_{X\in \Sigma_p,\; Y\in \Sigma_q,\; Y\le X} Y\]
As each stratum of $\Sigma_q$ is $X_q$-admissible and as the stratification is locally finite,
$cl(X_p)\cap X_q$ is $X_q$-admissible.

5) The frontier condition implies that 
 \[cl(X_p)=\bigcup_{X\in \Sigma_p}cl(X)=\bigcup_{X\in \Sigma_p}\bigcup_{Y\le X}Y\subset \bigcup_{q\le p}\bigcup_{X\in \Sigma_q}X=\bigcup_{q\le p}X_k\]
 \end{proof}

\begin{rema}
In example \ref{example:for:Sigmap}, the pair $(A,(X_p)_p)$ is not a stratified space.\footnote{However, if the frontier condition is replaced by the following more general condition:
\begin{center}
"for every pair of strata $(X,Y)$ such that $cl(X)$ intersects $Y$, the subset $cl(X)\cap Y$ is $Y$-admissible and $\dim Y$ is at least equal to $\dim X$",
\end{center}
it appears that all that is proved in this work remains true and that $(A,(X_p)_p)$ is still a stratified space. Moreover  $(L_p,\mathcal L_p)_p$ is also a trellis structure on $(A,(X_p)_p)$.}
\end{rema}
\begin{ques}
Given a compact subset $C$ of $\mathbb R^n$ which is the union of two disjoint locally compact subsets $A$ and $B$, does there exists an (abstract) stratification on $C$ such that $A$ and $B$ are unions of strata?\end{ques}

\subsubsection{Morphisms $\big(\mathcal T_A,\mathcal T_{A'}\big)$-controlled}

Let $(A,\Sigma)$ and $(A',\Sigma')$ be two stratified spaces admitting a trellis structure $\mathcal T$ and $\mathcal T'$ respectively.

A \emph{$(C^r)$-morphism} (resp. \emph{immersion}) \emph{$\big(\mathcal T,\mathcal T'\big)$-controlled} is a stratified morphism $f$ from $(A,\Sigma)$ to $(A',\Sigma')$ such that, for every stratum $X\in \Sigma$, there exists a neighborhood $V_X$ of $X$ in $L_X$ such that the restriction of $f$ to $V_X$ is a morphism (resp. immersion) from the lamination $\mathcal L_{X|V_X}$ into the lamination $\mathcal L_{X'}$, where $X'$ is the stratum of $\Sigma'$ which contains the image by $f$ of $X$.
In other words, every plaque of $\mathcal L_X$ contained in $V_X$ is sent into a leaf of $\mathcal L_{X'}$, the $r$-first derivatives of $f$ along such a plaque exist, are continuous on $V_X$ (resp. and the first one is moreover injective) and $f$ is continuous.

The neighborhood $V_X$ of $X$ is said to be \emph{adapted} to $f$, and the family $\mathcal V:=(V_X)_{X\in \Sigma}$ is called a \emph{family of neighborhoods adapted to $f$ (and to ($\mathcal T,\mathcal T'$))}.
 
If $(A',\Sigma')$ is a manifold $M$, then $M$ is also the only trellis structure on the stratified space $M$. In this case, we will say that this $C^r$-morphism is \emph{$\mathcal T$-controlled}.
 
We remark that $\mathcal T$-controlled embeddings from $(A,\Sigma)$ to $M$ are the $\mathcal T$-controlled are immersions which are homeomorphisms on their images.

A \emph{$\mathcal T$-controlled endomorphism } is a stratified endomorphism $f$ of $(A,\Sigma)$ which is $(\mathcal T,\mathcal T)$-controlled. This means that each stratum $X$ is sent by $f$ into itself and that there exists a neighborhood $V_X$ of $X$ in $L_X$ such that the restriction $f_{|V_X}$ is a morphism from the lamination $\mathcal L_{|V_X}$ to $\mathcal L_X$.

We denote by $Mor^r(\mathcal T, \mathcal T')$, $Im^r(\mathcal T,\mathcal T')$ and $Emb^r(\mathcal T,\mathcal T')$ the sets of  $(\mathcal T,\mathcal T')$-controlled $C^r$-morphisms, immersions and embeddings respectively.
We denote by $End^r(\mathcal T)$ the set of $\mathcal T$-controlled $C^r$-endomorphisms.

Throughout the rest of this chapter, all controlled morphisms are supposed to be of class $C^r$. 
\begin{propr}\label{mor}
Let $(A,\Sigma)$, $(A',\Sigma')$, and $(A'',\Sigma'')$ be stratified spaces admitting a trellis structure $\mathcal T$, $\mathcal T'$ and $\mathcal T''$ respectively.
\begin{itemize}
\item The identity of $A$ is a $\mathcal T$-controlled endomorphism.
\item The composition of a $(\mathcal T',\mathcal T'')$-controlled morphism with a $(\mathcal T,\mathcal T')$-controlled morphism is a $(\mathcal T,\mathcal T'')$-controlled morphism.
\end{itemize}
\end{propr}
\begin{proof}
The identity of $A$ is clearly a controlled endomorphism. Let us prove that the composition of two controlled morphisms is a controlled morphism.

Let $f\in Mor^r(\mathcal T,\mathcal T')$ and $f'\in Mor^r(\mathcal T',\mathcal T'')$.  Each stratum $X\in \Sigma$ is sent by $f'$ into a stratum $X'\in \Sigma'$ which is sent by $f'$ into a stratum $X''$. Let $V_X$ and $V_{X'}$ be two neighborhoods of $X$ and $X'$ adapted to respectively $f$ and $f'$. Then the neighborhood $V_X\cap f^{-1}(V_{X'})$ of $X$ is adapted to $f'\circ f$.\end{proof}

 \subsubsection{Equivalent controlled morphisms}\label{eq:co:mo}
Let $\mathcal T$ and $\mathcal T'$ be two trellis structures on the stratified spaces $(A,\Sigma)$ and $(A,\Sigma')$ respectively.
Two morphisms $f$ and $\hat f$ of  $Mor^r(\mathcal T , \mathcal T')$ (resp.
 $Im^r(\mathcal T,\mathcal T')$, resp. $End^r(\mathcal T'))$ are said to be \emph{equivalent} if, for each stratum $X\in \Sigma$, there exists a stratum $X'\in \Sigma'$ such that $f$ and $\hat f$ send $X$ into an $X'$ and there exists a neighborhood $V_{X}$ of $X$ in $L_X$ such that the restrictions $ f_{|V_X}$ and $\hat f_{|V_X}$ are equivalent as morphisms of laminations from $\mathcal L_{X|V_X}$ to $\mathcal L_{X'}$.
 This means that every point $x\in V_X$ is sent by $f$ and $\hat f$ into a same leaf of $\mathcal L_{X'}$.

Given a family of tubular neighborhoods $\mathcal V=(V_X)_X$ adapted to $f$, let $Mor_{f\mathcal V}^r(\mathcal T,\mathcal T')$ (resp.  $Im_{f\mathcal V}^r(\mathcal T,\mathcal T')$, resp. $End_{f\mathcal V}^r(\mathcal T'))$ be the set of $\mathcal T$-controlled morphisms $\hat f$ such that $\mathcal V$ is also adapted to $\hat f$ and such that, for every $X\in \Sigma$ sent by $f$ into a certain stratum $X'$,
the restrictions of $f$ and $\hat f$ to $V_X$ are equivalent as morphisms from the lamination $\mathcal L_{X|V_X}$ into $\mathcal L_{X'}$.

We endow $Mor_{f\mathcal V}^r(\mathcal T,\mathcal T')$ with the topology induced by the product topology on 
\[\prod_{X\in \Sigma\; f(X)\subset X'\in \Sigma'} Mor_{f|V_X}^r (\mathcal L_{X|V_X},\mathcal L_{X'}').\]
We equip the set of families of neighborhoods adapted to $f$ with the following partial order:
\[\mathcal V\le \mathcal V'\Longleftrightarrow \forall X\in \Sigma,\quad V_X\subset V_X'.\]
\begin{propr}\label{topmor}
Let $\mathcal T$ and $\mathcal T'$ be two trellis structures on two stratified spaces $(A,\Sigma)$ and $(A',\Sigma')$ respectively. Let $f$ be a $(\mathcal T,\mathcal T')$-controlled morphism
and let $\mathcal V\le \mathcal V'$ be two families adapted to $f$. Then the topology of $Mor_{f\mathcal V'}^r(\mathcal T,\mathcal T')$ is equal to the topology induced by $Mor_{f\mathcal V}^r(\mathcal T,\mathcal T')$.\end{propr}
\begin{proof}
By definition of these topologies it is clear that the topology of $Mor_{f\mathcal V'}^r(\mathcal T,\mathcal T')$ is finer
than the topology induced by $Mor_{f\mathcal V}^r(\mathcal T,\mathcal T')$. So it is sufficient to prove that the topology of $Mor_{f|V_X'}^r(\mathcal L_{X|V_X'}, \mathcal L_{X'}')$ is coarser  than the topology induced by $Mor_{f\mathcal V}^r(\mathcal T,\mathcal T')$, for every $X\in \Sigma$.
As a compact subset of  $V_{X}'$ is a finite union of compact subsets in $(V_Y)_{Y\ge X}$, the topology of $Mor_{f|V_X'}^r(\mathcal L_{X|V_X'}, \mathcal L_{X'}')$ is coarser than the topology induced by the product 
\[\prod_{Y\ge X\; f(Y)\subset Y'\in \Sigma'} Mor_{f|V_{Y}}^r (\mathcal L_{Y|V_Y},\mathcal L_{Y'}')\]
which is also coarser than the topology of $Mor_{f\mathcal V}^r(\mathcal T,\mathcal T')$.\end{proof}

This last property implies that the spaces $(Mor_{f\mathcal V}^r(\mathcal T,\mathcal T'))_\mathcal V$
are only different by plaque-preserving conditions. If $(A',\Sigma')$ is a manifold $M$, then the trellis structure $\mathcal T'$ consists of the only manifold $M$. Plaque-preserving conditions are then obviously always satisfied. Consequently, the space  $Mor_{f\mathcal V}^r (\mathcal T, \mathcal T')$ depends neither on $f$ nor on $\mathcal V$. That is why we abuse notation by denoting by $Mor^r(\mathcal T,M)$ this topological space.

\begin{rema} The topology of $Mor_{f\mathcal V}^r(\mathcal T,\mathcal T')$ is finer than the one induced by $Mor_f^r(\Sigma,\Sigma')$. Actually, the topology induced by the compact-open topology of $C^0(A,A')$ is coarser than the topology of $Mor_{f\mathcal V}^r(\mathcal T,\mathcal T')$. This is because, on the one hand, $\mathcal V$ is an open cover of $A$ and the canonical inclusion of $C^0(A,A')$ in $\prod_{X\in \Sigma} C^0(V_X,A')$ is a homeomorphism onto its image. And on the other, for any stratum $X\in \Sigma$ sent by $f$ into a stratum $X'\in \Sigma$, the topology of $Mor_{f|V_X}^r(\mathcal L_{X|V_X},\mathcal L_{X'}')$ is finer than the topology $Mor_{f|X}^r(X,X')$.

 Generally, the topology of $Mor_{f\mathcal V}^r(\mathcal T,\mathcal T')$ is strictly finer than the topology induced by $Mor_{f\mathcal V}^r(\Sigma,\Sigma')$. For example, let us consider the unit closed disk $cl(\mathbb D)$ of the complex plan $\mathbb C$, which supports the canonical (differentiable) stratification $\Sigma$ formed by the unit disk $\mathbb D$ and the unit circle $\mathbb S^1$. This stratified space admits the trellis structure $\mathcal T$ consisting of the tubular neighborhood  $\mathbb D$ of $\mathbb D$ and the lamination $(L_{\mathbb S^1},\mathcal L_{\mathbb S^1})$ whose leaves are the circles centered in 0 and of radius $\rho\in]1/2,1]$.
Let $(\theta,\rho)$ be the polar coordinates on $cl( \mathbb D)$.
The set of $(\mathcal T, \mathbb R)$-controlled functions $f$ on $cl(\mathbb D)$ such that $sup_{x\in L_{\mathbb S^1}}\|\partial_\theta f(x)\|< 1$ is an open subset of $Mor^r(\mathcal T,\mathbb R)$. But, in every open subset $O$ of $Mor^r(\Sigma,\mathbb R)$, there exists a sequence of functions $(f_n)_n\in (Mor(\mathcal T,\mathbb R)\cap O)^\mathbb N$ such that $\sup_{n\ge 0,\; x\in L_{\mathbb S^1}}
 \|\partial_\theta f_n(x)\|=\infty$.

\end{rema}
\subsubsection{Geometric structures on stratified spaces}\label{desstruc}
In this section, we recall other works defining similar structure than trellis. These structures  are almost always weaker than the trellis structure, because they were used in a topological context.

The study of such structures come back as far as the work of H. Whitney on the study of the singularities of analytic varieties \cite{W1}.
\begin{conj}[H. Whitney 1965]
Every analytic variety $V$ of $\mathbb C^n$ supports an analytic stratification such that, for every point $p$ of a stratum $X$, there exists a neighborhood $U$ of $p\in V$, a metric space $T$ and a homeomorphism
\[\phi\;:\; (X\cap U)\times T\rightarrow U\]
such that, for every $t\in T$, the restriction $\phi_{|X\cap U\times \{t\}}$ is biholomorphic onto its image, the differential on these restrictions are continuous on $(X\cap U)\times T$ and the lamination generated by the chart $\phi^{-1}$ (and of the same dimension as $X$) is coherent with all the strata.
\end{conj}
In other words, H. Whitney conjectured the existence of a stratification that admits locally a tubular neighborhood, with holomorphic leaves, for each stratum.

 Later, R. Thom and J. N. Mather were also interested in these structures for the study of the singularities  of differentiable maps. They introduced differentiable stratified spaces $(A,\Sigma)$ with extra intrinsic regularity conditions which in some cases allow to proof that their stratifications are locally trivial: for every point $x\in A$ which belongs to a stratum $X\in \Sigma$, there exists a neighborhood $U$ of $x$ in $A$, a neighborhood $V$ of $x$ in $X$, a differentiable stratified space
$(A',\Sigma')$ and a homeomorphism $h\;:V\times A'\rightarrow U$ such that the strata of $\Sigma_{|U}$ are the images by $h$ of strata of the product stratified space $(V\times A',X_{|V}\times \Sigma')$.

 In 1993, D. Trotman adapted the Whitney conjecture to the differentiable stratified spaces\footnote{This an adaptation because H. Whitney shows that every analytic variety supports a $b$-regular analytic stratification \cite{W2}.}. To formulate his conjecture, let us recall that an embedding $p$ of a differentiable stratified space $(A,\Sigma)$ into $\mathbb R^n$ is \emph{$b$-regular} if 
 
 For all strata $(X,Y)\in \Sigma^2$ with $Y< X$, for all sequences $(x_i)_i\in X^\mathbb N$ and $(y_i)_i\in Y^\mathbb N$ which converge to $y\in Y$, if $(T_{x_i}X)_i$ converges to $\tau$ and the unitary vector in the direction of $
\overrightarrow{x_i y_i}\in\mathbb R^n$ converges to $\lambda$, then $\lambda$ is included in $\tau$.

 It is well known and easy to show that the $b$-regularity implies the $a$-regularity.

The Trotman's conjecture is the following:
\begin{conj}[D. Trotman 1993]\label{trot}
 Let $(A,\Sigma)$ be a differentiable stratified space $b$-regularly embedded by $p$ into $\mathbb R^n$; then for every stratum $X\in \Sigma$ and point $x\in X$,
 there exists a neighborhood $U$ of $x$ and a tubular neighborhood $(L,\mathcal L)$ of $X_{|U}$ in the stratified space  $(U,\Sigma_{|U})$ such that the restriction  $p_{|L}$ is an embedding of $(L,\mathcal L)$.
\end{conj}

Unfortunately, we will see in part \ref{part varbor} that there exist $b$-regular stratifications that are persistent but not $b$-regularly persistent.

In the spirit of this conjecture, in his PhD thesis C. Murolo \cite{MT} defines the "Syst\`eme de contr\^ole feuillet\'e, totalement compatible et $a$-r\'egulier" which is the data of compatible tubular neighborhoods $(L_X,\mathcal L_X)_{X\in \Sigma}$, on a differentiable stratified space.

In another context, for surface diffeomorphisms that satisfy Axiom A and the strong transversality condition,
 W. De Melo \cite{dM} built a real  trellis structure on the stratification of laminations $(W^s(\Lambda_i))_i$ defined  in property \ref{Strat:AxiomA}
\footnote{Actually, he requires the existence of a "system of unstable tubular families". This structure is a family of compatible tubular neighborhoods. Even if the foliation's condition is not required, he proves it. Actually, he is just interested in finding topological properties. The present work can be used to show smoother property of the conjugacy homeomorphism: to be an stratified endomorphism.}.
 This construction allowed him to prove the structural persistence of surface $C^1$-diffeomorphism that satisfies Axiom A and the strong transversality condition. In a local way, this idea was improved by  C. Robinson  \cite{Rs} to achieve the proof of the Palis and Smale's conjecture \cite{Sm}:  every $C^1$-diffeomorphism of compact manifold that satisfies Axiom A and the strong transversality condition is $C^1$-structurally stable.
 \footnote{Actually, he requires the existence of "Compatible families of unstable disks". Even if the algorithm builds, locally, a trellis structure, for the same reason as W. de Melo, he only requires to have locally a "system of unstable tubular families".}.

 We will see that there exist $a$-regular stratifications that do not admit trellis structures and are not even locally trivial.
 
 But beyond the local obstructions, there exist also global topological constrains, that prevent stratified spaces to admit trellis structures.
 
For instance, let us consider the stratification on the tangent bundle of the sphere, consisting of two strata: the first being the graph of the zero section endowed with the structure of $2$-manifold (that we identify with the sphere $\mathbb S^2$) and the second being the complement of this graph, endowed with its structure of $4$-manifold. For the sake of contradiction, let us suppose that this stratification admits a trellis structure.
Then there exists a lamination on a neighborhood of the sphere in the tangent bundle, such that the sphere is a leaf.
As the sphere is simply connected, the holonomy along the leaves is trivial. Therefore we can transport a small non-zero vector of the tangent bundle, by holonomy to define a vector field on the tangent bundle without zero. But this is well known that such a vector field does not exit on the 2-sphere.

\newpage\section{Persistence of stratification of laminations}\label{chap:persist:strati}
Throughout this chapter, all manifold considered are of class $C^\infty$. The regularity classes of laminations, stratifications or trellis structures will be not mentioned if they are not relevant.

In section \ref{main:result}, we state our main result on persistence of stratification of normally expanded laminations. In section \ref{applications}, we give various applications of this result. Before dealing with stratifications of laminations, we study the restricted case of laminations (which is a stratification consisting of only one stratum), hoping that our definitions and main result will be better understood.

\subsection{Persistence of lamination}
\subsubsection{Preserved laminations}
 A lamination $(L,\mathcal L)$ embedded by $i$ into a manifold $M$ is \emph{preserved} by an endomorphism $f$ of $M$ if
 the embedding of each leaf of $\mathcal L$ is sent by $f$ into a leaf of $\mathcal L$.
 
This is equivalent to suppose the existence of an endomorphism $f^*$ of $(L,\mathcal L)$ such that the following diagram commutes:
\[\begin{array}{rcccl}
&&f&&\\
&M&\rightarrow&M&\\
i&\uparrow&&\uparrow&i\\
&L&\rightarrow&L&\\
&&f^*&&\end{array}\]

The endomorphism $f^*$ is the \emph{pullback of $f$ via $i$}.

When the lamination is only immersed by $i$, these two definitions are not equivalent.

A lamination $(L,\mathcal L)$ immersed by $i$ into a manifold $M$ is \emph{preserved} by an endomorphism $f$ of $M$ if there exists a \emph{pull back of $f$ in $(L,\mathcal L)$ via $i$}. That is an endomorphism $f^*$ of $(L,\mathcal L)$ such that the following diagram commutes:

\[\begin{array}{rcccl}
&&f&&\\
&M&\rightarrow&M&\\
i&\uparrow&&\uparrow&i\\
&L&\rightarrow&L&\\
&&f^*&&\end{array}\]

The leaves of a lamination $(L,\mathcal L)$ immersed by $i$ into a manifold $M$ is \emph{preserved} by an endomorphism $f$ of $M$ if
 the immersion of each leaf of $\mathcal L$ is sent by $f$ into an immersion of a leaf of $\mathcal L$.

Clearly, if $f$ preserves an immersed laminations, then it preserves its leaves. We now gives two examples of diffeomorphisms preserving the leaves of an immersed lamination but not the lamination.
 
 \begin{exem}
 Let $(L,\mathcal L)$ be the circle $\mathbb S^1$ and let $i$ be the immersion from $\mathbb S^1$ into $\mathbb R^2$ represented below:
\begin{figure}[h]
    \centering
\includegraphics[width=7cm]{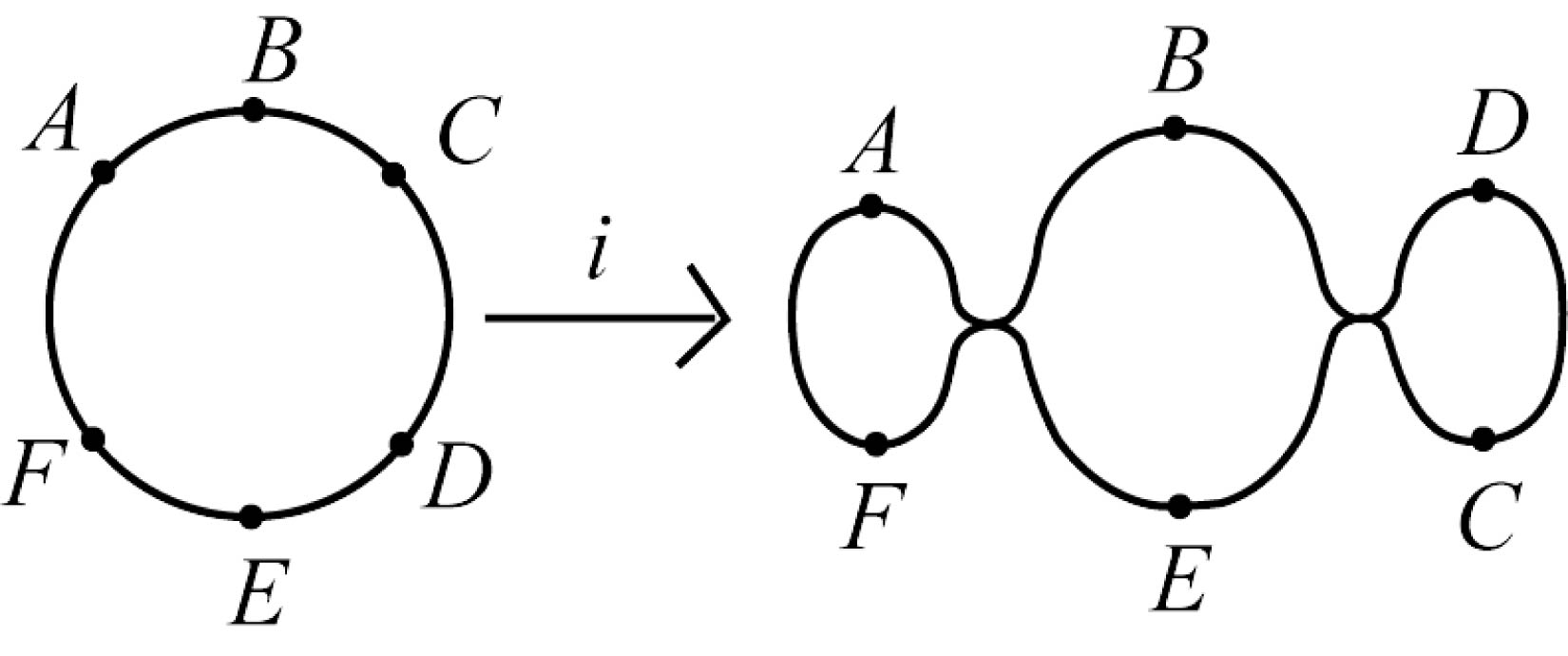}
    \caption{}
\end{figure}

 Let $f$ be the diffeomorphism of $\mathbb R^2$, preserving $i(\mathbb S^1)$ and such that $f_{|i(\mathbb S^1)}$ is homotopic to the symmetry of axis $(BE)$. One notes therefore that $f$ does not pull back to $\mathbb S^1$.\footnote{Actually this example is similar to the example P70 of \cite{HPS}, but we allow $f$ to be a diffeomorphism.}
\end{exem}
\begin{exem}\label{contrempletore} Let $\mathbb T^2$ be the torus which is the quotient $\mathbb R^2/\mathbb Z^2$. Let $f$ be the diffeomorphism of $\mathbb T^2$, whose lift in $\mathbb R^2$ is the linear map, with matrix 
\[ \left[\begin{array}{cc}2&1\\1&1\end{array}\right]\]
Let $i\;:\; \mathbb T^2\rightarrow \mathbb T^2$ be a $2$-covering map of $\mathbb T^2$. Hence $i$ is an immersion, but there does not exist any pullback of $f$ into $\mathbb T^2$ via $i$.\end{exem}
    \begin{proof}
     We suppose, for the sake of contradiction, that there exists an endomorphism $f^*$ of $f$. As $f$ fixes the point $0$,        $f^*$ preserves the fiber $i^{-1}(\{0\})$ that we denote by $\mathbb Z/2\mathbb Z$.

    Given two integers $(a,b)$, we denote by $hol_{ ( a,b)}$ the automorphism of the fiber $\mathbb Z/2\mathbb Z$                   obtained by holonomy along a closed path of $\mathbb T^2$ pointed in $0$ and tangent to the vector $(a,b)$.
    By commutation of the diagram, we have for any integers $a$, $b$

    \begin{equation}\label{hol*} hol_{ A( a,b)}\circ f^*=f^*\circ hol_{ ( a,b)}.\end{equation}

We remark that $f^*_{\mathbb Z/2\mathbb Z}$ is either an automorphism or a non-bijective map. If $f^*_{|\mathbb Z/2\mathbb Z}$  is non-bijective, we may suppose that $f^*$ sends $\mathbb Z/2\mathbb Z$ onto $\{0\}$. The above                 equation implies that 
    \[hol_{A(a,b)}=0,\quad \forall (a,b)\in \mathbb Z^2.\]
But this is not possible because the covering is connected.

    If $f^*_{\mathbb Z/2\mathbb Z}$ is an automorphism, as $Aut(\mathbb Z/2\mathbb Z)\cong\mathbb Z/2\mathbb Z$ is                          commutative, equation (\ref{hol*}) implies 
    \[ hol_{ A( a,b)}=hol_{ ( a,b)}, \quad \forall (a,b)\in \mathbb  N^2.\]
    As $\mathbb T^2$ is connected, we have three possibilities for the morphism $hol$:
        \begin{itemize}
        \item $hol_{(1,0)}=+1$ and $hol_{(0,1)}=+1$. Then $hol_{ A( 0,1)}=1+1=0$  is not equal to $hol_{(0,1)}=1$.
        \item $hol_{(1,0)}=+1$ and $hol_{(0,1)}=+0$. Then $hol_{ A( 1,0)}=2+0=0$  is not equal to $hol_{(1,0)}=1$.
        \item $hol_{(1,0)}=+0$ and $hol_{(0,1)}=+1$. Then $hol_{ A( 1,0)}=0+1=1$  is not equal to $hol_{(1,0)}=0$.
        \end{itemize}
    \end{proof}

Both above examples standardize the following question of Hirsch-Pugh-Shub ( \cite{HPS}, P. 70):

\begin{ques}\label{question:HPS} Let $N$ be a manifold immersed by $i$ into a manifold $M$ and let $f$ be a diffeomorphism of $M$ which preserves the leaves of $N$.
Does there exists any immersion $i'$ from $N$ into $M$, whose image is the same as the image of $i$ and such that $f$ pullback to $N$ via $i'$?\end{ques}


\begin{proof}[Negative answer to this question \ref{question:HPS}] The idea of the proof is to use example \ref{contrempletore}, by obliging $i$ to be a $2$-covering. To do it, as an algebraic geometer, we blow up $\mathbb T^2$ at the fixed point $0$ of the diffeomorphism $f$ of $\mathbb T^2$. Hence, we obtain a diffeomorphism $f^\#$ of the connected sum $M^\#:=\mathbb T^2\#\mathbb P^2(\mathbb R)$. As there exists a $2$-covering of the torus by the torus, there exist a $2$-covering of $M^\#$ by the manifold
$\hat{M}^\#:= \mathbb T^2\#\mathbb P^2(\mathbb R)\#\mathbb P^2(\mathbb R)$. Such a covering is an immersion from $\hat{M}^\#$ to ${M}^\#$, whose image is obviously $f^\#$-invariant.
 
Let us show that any immersion $j$ from $\hat{M}^\#$ onto ${M}^\#$ is a $2$-covering.

 As the preimage by $j$ of any point is a compact discrete subset, its cardinality is finite. This cardinality depends lower semi-continuously on the point. The upper semi-continuity follows from the compactness of ${M}^\#$ and $\hat{M}^\#$. By connectedness, the cardinality $k$ of $j$-fibers is constant. Thus, the map $j$ is a $k$-covering.
 
As the Euler constant of $\hat{M}^\#$ is equal to $-2$ which is twice the Euler constant of $M^\#$, the map $j$ is a $2$-covering (since a triangulation of $M^\#$ small enough has all vertexes, edges or faces which are $k$-times lifted in $\hat{M}^\#$).

To reply to the question, it is sufficient to prove that there does not exist any $2$-covering map from $\hat{M}^\#$ onto ${M}^\#$, such that $f^\#$ pullback to $\hat {M}^\#$.

We suppose, for the sake of contradiction, that there exists such a $2$-covering map $j$.

The point $0\in \mathbb T^2$ was blowing up to a circle $\mathbb S^1$ and a small neighborhood of $0$ was blowing up to a M\"obius strip. The preimage by $j$ of this strip is either two disjoint M\"obius strips or one cylinder.
In the first case, the restriction of $j$ to each of these strips is a homeomorphism. In the second case, the restriction of $j$ to the cylinder is a $2$-covering.

 In the first case, we shall blow down $M^\#$ and $\hat{M}^\#$ at the circle $S^1$ and its preimages by $j$. Hence, we make up a pull back of $f$ in a $2$-covering of the torus $\mathbb T^2$. Therefore, example \ref{contrempletore} shows a contradiction. But in the second case, we cannot blown down at the preimage of $\mathbb S^1$. Let us use a little trick in the second case.
 
 Let $\hat{ \mathbb S}^1$ be the circle which is the unique preimage of $\mathbb {S}^1$. We cut along $\hat {\mathbb {S}}^1$ the surface $\hat{M}^\#$. This makes a surface $\hat{M}'^\#$ with two boundaries $B_1$ and $B_2$. Each of these boundaries is a $2$-covering of $\mathbb S^1$ via the immersion $j'$ from $\hat{M}'^\#$ onto ${M^\#}$, canonically made up from $j$. We know identify the points of $B_1$ (resp. $B_2$) which have the same image into $\mathbb S^1$ via $j'$. This constructs a new surface which
is now a $2$-covering of ${M^\#}$, for which $f^\#$ pullbacks to a certain map $\hat {f}'^\#$, and such that the preimage of a small neighborhood of $\mathbb S^1$ consists of two M\"obius strips.
 
  We will only use these last properties shared with the first case, in order to find a contradiction.

Let us come back to the lighter notations of the first case. We blow down the surface $M^\#$ to $M$ along a small neighborhood of the circle $\mathbb S^1$. As $f^\#$ was obtained by bowing up at the fixed point 0, via this blowing down, the dynamics induced by $f^\#$ on $M$ is $f$. As $j$ is a homeomorphism from both preimages of the small neighborhood of the circle $\mathbb S^1$, we can blow down $\hat{M}^\#$ along both M\"obius strips in order to construct a $2$-covering of the torus in which we can
pullback $f'$. A contradiction follows from example \ref{contrempletore}.
 \end{proof}

\subsubsection{Persistence of laminations}
Let $r$ be a fixed positive integer.

Let $(L,\mathcal L)$ be a lamination $C^r$-embedded by $i$ into a manifold $M$.  Let $f$ be a $C^r$-endomorphism of $M$ which preserves $\mathcal L$. Then the embedded lamination $(L,\mathcal L)$ is \emph{$C^r$-persistent} if for any endomorphism $f'$ $C^r$-close to $f$, there exists an embedding $i'$ $C^r$-close to $i$ such that $f'$ preserves the lamination $(L,\mathcal L)$ embedded by $i'$ and such that 
each point of $i'(L)$ is sent by $f'$ into the image by $i'$ of a small plaque containing $f(x)$. 
This implies that the pullback $f'^*$ of $f'$ is equivalent and $C^r$-close to the pullback $f^*$ of $f$.

Let $(L,\mathcal L)$ be a lamination immersed by $i$ into a manifold $M$.  Let $f$ be a $C^r$-endomorphism of $M$ which preserves $\mathcal L$. Let $f^*$ be a pull back of $f$ in $(L,\mathcal L)$.
Then the immersed lamination $(L,\mathcal L)$ is \emph{persistent} if for any endomorphism $f'$ $C^r$-close to $f$, there exists an immersion $i'$ $C^r$-close to $i$, such that $f'$ preserves the lamination $(L,\mathcal L)$ immersed by $i'$ and pullback to $(L,\mathcal L)$ to an endomorphism $f'^*$ equivalent and $C^r$-close to $f^*$. In other words, for  every $f'\in End^r(M)$ close to $f$ there exists $i'\in Im^r(\mathcal L,M)$ and $f'^*\in End_{f^*}^r(\mathcal L)$ close to $i$ and $f^*$ such that the following diagram commutes:
\[\begin{array}{rcccl}
&&f'&&\\
&M&\rightarrow&M&\\
i'&\uparrow&&\uparrow&i'\\
&L&\rightarrow&L&\\
&&f'^*&&\end{array}\]

In the above definitions the topology of $End^r(M)$, $Im^r(\mathcal L,M)$, $Emb^r(\mathcal L,M)$ and $End_{f^*}^r(\mathcal L)$ are described in section \ref{Top}.

\subsubsection{Lamination persistence theorems}
Up to now the following result was the most general theorem showing that hyperbolicity implies persistence of laminations.
\begin{theo}[Hirsch-Pugh-Shub \cite{HPS}]
Let $r\ge 1$ and let $(L,\mathcal L)$ be a compact lamination $C^r$-immersed by $i$ into a manifold $M$. Let $f$ be a $C^r$-diffeomorphism of $M$ which preserves $\mathcal L$, is $r$-normally hyperbolic at it and such that a pull back $f^*$ of $f$ let invariant $L$ ($f^*(L)=L$).  Then  the immersed lamination is $C^r$-persistent.

 If moreover $i$ is an embedding and $f$ is plaque-expansive at $\mathcal (L,\mathcal L)$ then the embedded lamination is $C^r$-persistent.\end{theo}

We recall the definition of the normal hyperbolicity and the plaque-expansiveness in section \ref{axiom A} and \ref{plaque-expansiveness} respectively.

A first consequence of our main result is an analogous theorem of the above:

We allow $f$ to be an endomorphism, that is to be possibly non-bijective and with singularities, but we suppose $f$ to be normally expanding instead of normally hyperbolic.

\begin{theo}
Let $(L,\mathcal L)$ be a compact lamination $C^r$-immersed by $i$ into a manifold $M$. Let $f$ be a $C^r$-endomorphism of $M$ which preserves and $r$-normally expands $\mathcal L$.  Then  the immersed lamination is $C^r$-persistent.

 If moreover $i$ is an embedding and $f$ is plaque-expansive at $\mathcal (L,\mathcal L)$ then the embedded lamination is $C^r$-persistent.\end{theo}

 Let us describe this theorem.

\subsubsection{Normal expansion}
Let $(L,\mathcal L)$ be a lamination and let $(M,g)$ be a Riemannian manifold. Let $i\in Im(\mathcal L,M)$ and let $f\in End(M)$ which preserves the immersion $i$ of $(L,\mathcal L)$. Let $f^*$ be a pullback of $f$ in $(L,\mathcal L)$.
 
We identify, via the injection given by $i$, the bundle $T\mathcal L\rightarrow L$ to a subbundle of $\pi\;:\; i^*TM\rightarrow L$. Thus, $\mathcal L$ is endowed with the Riemannian metric $i^*g$. By commutativity of the diagram above, the endomorphism $i^*Tf$ of $i^*TM\rightarrow L$, over $f^*$, preserves the subbundle $T\mathcal L$.
The action of the endomorphism $i^*Tf$ to the quotient $i^*TM/T\mathcal L$ is denoted by 
\[[i^*Tf]\;:\; i^*TM/T\mathcal L\rightarrow i^*TM/T\mathcal L\]
We notice that the quotient $i^*TM/T\mathcal L$ is the normal bundle of $\mathcal L$. We endow this bundle with the norm induced by the Riemannian metric of $M$: the norm of a vector $u\in i^*TM/T\mathcal L$ is the norm of the vector $i^*TM$ which is orthogonal to $T\mathcal L$ and represents $u$.

\begin{defi} For every $r\ge 1$, we say that $f$ \emph{$r$-normally expands the lamination $(L,\mathcal L)$ (immersed by $i$ over $f^*$)}, if there exist a function $C$ on $L$ and $\lambda<1$ such that for all $v\in i^*TM/T\mathcal L\setminus \{0\}$ and $n\ge 0$, we have 
\[\max\big( 1,\|T_{\pi(v)}f^{*^n}\|^r\big)\cdot \|v\|<C(x)\cdot \lambda^n\cdot \|[i^*Tf]^n(v)\|\]\end{defi}
\begin{rema} Usually, one supposes $L$ to be compact and $C$ to be constant. This is coherent with this definition, by replacing $C$ by its maximum on $L$.\end{rema}
\begin{defi} If the function $C$ is bounded, we say that $f$ \emph{uniformly $r$-normally expands the lamination} $(L,\mathcal L)$.\end{defi}

\begin{propr}\label{propertyA}
Let $(L,\mathcal L)$ be a lamination immersed by $i$ into a manifold $M$ and $1$-normally expanded.
Then, for all $x$, $y\in L$ with the same images by $i$, the spaces $T_x \mathcal L$ and $T_y\mathcal L$ are sent by $Ti$ to the same subspace of $T_{i(x)}M$. Thus, we can abuse of notation by denoting $T_{i(x)}\mathcal L$ the subspace $Ti(T_x\mathcal L)$.
 
Moreover, for every compact subset $K$ of $L$, the section of the Grassmannian $z\in i(K)\rightarrow T_z\mathcal L$ is continuous.
\end{propr}
\begin{proof}
By normal expansion at $x$, the vectors of $Ti(T_y\mathcal L)\setminus Ti(T_x\mathcal L)$ grow exponentially faster than those of $Ti(T_x\mathcal L)$ and by normal expansion at $y$, the vectors of $Ti(T_x\mathcal L)\setminus Ti(T_y\mathcal L)$ grow exponentially faster than those of $Ti(T_y\mathcal L)$. Then the vectors of $Ti(T_y\mathcal
L)\setminus Ti(T_x\mathcal L)$ grow and decrease exponentially faster than those of  $Ti(T_x\mathcal L)\setminus 
Ti(T_y\mathcal L)$. Therefore the spaces $Ti(T_x\mathcal L)$ and $Ti(T_y\mathcal L)$ are equal.

For any compact subset $K$ of $L$, the continuity of the map $z\in i(K)\rightarrow T_z\mathcal L$ follows from the compactness of $K$: given a sequence $(z_n)\in i(K)^\mathbb N$ which converge to some $z\in i(K)$, there exists a sequence in $(x)_n \in K^\mathbb N$, sent by $i$ to $(z_n)_n$. By compactness of $K$, we may suppose that $(x_n)_n$ converges to some $x\in K$. Therefore, by continuity, $x$ is sent by $i$ to $z$ and we have 
\[\lim_{n\rightarrow \infty} T_{z_n}\mathcal L =\lim_{n\rightarrow \infty} Ti(T_{x_n}\mathcal L)=  Ti(T_{x}\mathcal L)=T_z\mathcal L\]
\end{proof}

\begin{rema} The definition of the $r$-normal expansion above is equivalent to the following:

There exists $\lambda>1$ and a continuous positive function $C$ on $L$ such that for every $x\in L$, for all unitary vectors $v_0\in T_{i(x)}\mathcal L$ and $v_1\in (T_{i(x)}\mathcal L)^\bot$, for any $n\ge 0$, we have
\[\|p\circ Tf^{n}(v_1)\|\ge C(x)\cdot \lambda^n\cdot (1+\|Tf^n(v_0)\|^r),\]
with $p$ equal to the orthogonal projection of $TM_{|i(L)}$ onto $T\mathcal L^\bot$.\end{rema}

 \begin{prop}\label{P1}
Let $r\ge 1$. Let $(L,\mathcal{L})$  be a lamination immersed by $i$ into a Riemannian manifold $(M,g)$.
Let $f\in End(M)$, $i\in Im(\mathcal{L},M)$, and $f^*\in End(\mathcal{L})$.

 If $f$ $r$-normally expands the immersed lamination $\mathcal L$ over $f^*$, for every compact subset $K$ of $L$ stable by $f^*$ ($f^*(K)\subset K$), there exists a Riemannian metric $g'$ on $M$ and $\lambda'<1$ such that, for the norm induced by $g'$ on $i^*TM$ and every $v\in (i^*TM/T\mathcal L)_{|K}\setminus \{0\}$, we have 
\[\max\big( 1,\|T_{\pi(v)}f^{*}\|^r\big)\cdot \|v\|<\lambda'\cdot \|[i^*Tf](v)\|.\]

We say that $g'$ is an adapted metric to the normal expansion of $f$ on $K$.
\end{prop}
\label{section:P1}
We will show this proposition in annex \ref{proof P1}.

The following property give a geometrical equivalent interpretation of the $1$-normal expansion, which is useful.
\begin{propr}\label{property B} Let $(L,\mathcal L)$ be a lamination immersed by $i$ into a manifold $M$. Let $f\in End^1(M)$ which $r$-normally expands $(L,\mathcal L)$ over $f^*\in End^1(\mathcal L)$, for some $r\ge 1$. 
Let $K$ be a compact of $L$ sent into itself by $f^*$. Then there exist $\lambda>1$, a Riemannian metric on $M$ adapted to the $r$-normal expansion of $\mathcal L$ over $K$ and a (open) cone field $C$ on $i(K)$ such that, for every $x\in i(K)$:
\begin{enumerate}
\item $T_x\mathcal L^\bot$ is a maximal subspace of $T_xM$ contained in $C_x$,
\item $Tf\Big(cl\big(C(x)\big)\Big)$ is included in $C(f(x))\cup\{0\}$,
\item $\|Tf(u)\|>\lambda\cdot \|u\|$ for every $u\in C(x)$.
\end{enumerate}
\end{propr}
\begin{proof}
We endow $M$ with an adapted Riemannian metric to the normal expansion of $f$ on $K$. Let $x\in i(K)$, let $u\in T_y\mathcal L$ be a unitary vector and let $v\in T_y\mathcal L^\bot$ be small.
Then we have 
\[\tan \angle(T_xf(u+v),T_{f(x)}\mathcal L)= \frac{\|p_{\bot}\circ T_xf(v)\|}{p_T\circ T_xf(v)+T_xf(u)\|},\]
where $p_T$ and $p_\bot$ are respectively the orthogonal projection of $T_xM$ onto $T_x\mathcal L$ and $T_x\mathcal L^\bot$ respectively.

We have,
\[\frac{\|p_\bot\circ T_xf(v)\|}{p_\bot\circ T_xf(v)+T_xf(u)\|}\ge \frac{\|p_\bot \circ T_xf(v)\|}{\|T_xf(u)\|} +o(v)\]
Thus, by $1$-normal expansion 
\[\tan \angle\Big(T_xf(u+v),T_{f(x)}\mathcal L\Big) \ge \lambda\cdot \frac{\|v\|}{\|u\|}=\tan \angle\big(u+v,T_x\mathcal L\big)\]
Thus, for $\epsilon>0$ small enough, the properties 1-2 are satisfied by the following  cone field:
\[C_x:= \big\{(u+v)\in TM:\; u\in T\mathcal L,\; v\in T\mathcal L^\bot\; \mathrm{and}\; \|v\|>\epsilon \|u\|\big\}.\]
Let $\eta>0$ and $g'$ be the following inner product on $TM_{|i(K)}$:
\[g':= \eta\cdot g_{|T\mathcal L}+ g_{|T\mathcal L^\bot}.\]
For every $w\in C(x)$, we have 
\[g'(w)=\eta\dot g(u)+ g(v),\]
where $u\in T_x\mathcal L$  and $v\in T_x\mathcal L^\bot$ satisfy  $w=u+v$.

 By normal expansion, we have the existence of $\lambda'>1$ which does not depends on $x$, such that 
\[g'(Tf(w))=\eta\cdot g(p_T(Tf(v))+Tf(u))+g(p_\bot (Tf(v)))\ge \lambda'^2 g(v).\]
As $\epsilon\cdot \|u\|<\|v\|$, we have 
\[g'(Tf(w))\ge (\lambda'^2-\eta/\epsilon^2) g(v)+\eta\cdot  g(u).\]
Thus, for $\eta>0$ sufficiently small, $(\lambda'^2-\eta/\epsilon^2)$ is greater than $1$ and we get conclusion 3.
\end{proof}    

\subsubsection{Persistence of immersed laminations}\label{persistence:lam}
 
The following result is a particular case of our main theorem (theorem \ref{th2}).
\begin{theo}\label{th1}
Let $r\ge 1$ and $(L,\mathcal L)$ be lamination $C^r$-immersed by $i$ into a manifold $M$. Let $f$ be a $C^r$-endomorphism of $M$ preserving this immersed lamination. Let $f^*$ be a pullback of $f$ and let $L'$ be a precompact open subset of $L$ such that 
\[f^*(cl(L'))\subset L'.\]
If $f$ $r$-normally expands $(L,\mathcal L)$, then the immersed lamination $(L',\mathcal L_{|L'})$ is persistent.
Moreover, there exists a continuous map 
\[f'\mapsto (i(f'),f'^*)\in Im^r(\mathcal L,M)\times End^r_{f^*}(\mathcal L)\]
defined on a neighborhood $V_f$ of $f$, such that $i(f)$ is equal to $i$ and such that the following diagram commutes:
\[\begin{array}{rcccl}
&&f'&&\\
&M&\rightarrow&M&\\
i(f')&\uparrow&&\uparrow&i(f')\\
&L'&\rightarrow & L'&\\
&&f'^*&&\end{array}.\]
Furthermore, there exists a compact neighborhood $W$ of $L'$ such that, for every $f'\in V_f$, the maps $i(f')$ and $f'^*$ are equal to $i$ and $f^*$ respectively, on the complement of the compact set $W$.
\end{theo}
\begin{rema}We remark that, every map $f'\in V_f$, close enough to $f$ normally expand the lamination $(L',\mathcal L_{|L'})$ immersed by $i(f')$ over $f'^*$. Hence the hypotheses of the theorem are open. \end{rema}

\begin{rema} In the above theorem, the continuity of
\[f'\mapsto (f'^*,i(f'))\in End^r_{f^*}(\mathcal L)\times Im^r(\mathcal L,M)\]
and the existence of $W$ imply that, for any $\epsilon>0$, for any $f'$ close enough to $f$ and for every $x\in L$, the points $i(x)$ and $i(f')(x)$ are $\epsilon$-distant and the points $f^*(x)$ and $f'^*(x)$ belong to a same plaque of $\mathcal L$ with diameter less than $\epsilon$. Similarly, the $r$-first derivatives of $i(f')$ and $f'^*$ along the leaves of $\mathcal L$ are uniformly close to those of $i$ and $f^*$, for $f'$ close to $f$.
\end{rema}

\begin{exem} Let $f_1$ be a $C^r$ diffeomorphism of a manifold $N_1$. Let $K$ be a hyperbolic compact subset. Then, by example \ref{lamhyper}, $W^s(K)$ is the image of a lamination $(L_1,\mathcal L_1)$ $C^r$-immersed injectively, whose leaves are the stable manifolds. 
Let $E_u$ be the unstable direction of $K$ and 
\[m:=\min_{u\in E_u\setminus \{0\}} \frac{\|Tf(u)\|}{\|u\|}.\]
We may suppose $m>1$.

Let $M_2$ be a compact Riemannian manifold and let $f_2$ be a $C^r$-endomorphism of $M_2$ whose differential is less than $\sqrt[r]{m}$ (hence $f_2$ has possibly many singularities and is not necessarily bijective). 

Thus, the product dynamics $f:= (f_1,f_2)$ on $M:=M_1\times M_2$ $r$-normally expands the $C^r$-immersed lamination $(L,\mathcal L):=(L_1\times M_2, \mathcal L_1\times M_2)$ over an endomorphism $f^*$. Thus, for any precompact subset $L'$ of $L$, whose closure is sent into itself by $f^*$ (there exists arbitrarily big such a subset), the lamination $(L',\mathcal L_{|L'})$ is $C^r$-persistent.
\end{exem} 
 
\subsubsection{Plaque-expansiveness}\label{plaque-expansiveness}
\begin{defi}[pseudo-orbit] Let $(L,\mathcal L)$ be a lamination and let $f$ be an endomorphism of $(L,\mathcal L)$. Let $\epsilon$ be a positive continuous function on $L$. An \emph{$\epsilon$-pseudo-orbit which respects $\mathcal L$} is a sequence\footnote{In the diffeomorphism context, as in Hirsch-Pugh-Shub's theorem, sequences are indexed by $\mathbb Z$.} $(x_n)_{n\ge 0}\in L^{\mathbb N}$ such that, for any $n\ge 0$, the point
$f(x_n)$ belongs to a plaque of $\mathcal L$ containing $x_{n+1}$ whose diameter is less than $\epsilon(x_{n+1})$.\end{defi}

\begin{defi}[Plaque-expansiveness] Let $\epsilon$ be a positive continuous function on $L$. The endomorphism $f$ is \emph{$\epsilon$-plaque-expansive at $(L,\mathcal L)$} if for any positive continuous function $\eta$ on $L$ less than $\epsilon$, for all $\eta$-pseudo-orbits $(x_n)_{n}$ and $(y_n)_{n}$ which respect $\mathcal L$, if for any $n$
the distance between $x_n$ and $y_n$ is less than $\eta(x_n)$, then $x_0$ and $y_0$ belong to a same small plaque of $\mathcal L$. \end{defi}

\begin{rema} Usually, one supposes $L$ to be compact and $\epsilon$ to be constant. This is coherent with this definition by replacing $\epsilon$ by its minimum.\end{rema}
\begin{rema} We do not know if the normal expansion implies the plaque-expansiveness, even when $L$ is compact. But in many case this is true (see annex \ref{pppppplaque}).\end{rema}

\subsubsection{Persistence of embedded laminations}
The following result is a particular case of the corollary \ref{cor2} of our main theorem.
\begin{coro}\label{cor1}
Let $r\ge 1$ and let $(L,\mathcal L)$ be lamination embedded by $i$ into a manifold $M$. Let $f$ be a $C^r$-endomorphism of $M$ preserving this embedded lamination. Let $f^*$ be a pullback of $f$ and let $L'$ be a precompact open subset of $L$ such that 
\[f^*(cl(L'))\subset L'.\]
If $f$ $r$-normally expands $(L,\mathcal L)$ and if $f^*$ is plaque-expansive,
then the embedded lamination $(L',\mathcal L_{|L'})$ is persistent.

Moreover there exists a continuous map 
\[f'\mapsto (i(f'),f'^*)\in Em^r(\mathcal L,M)\times End^r(\mathcal L)\]
defined on a neighborhood $V_f$ of $f$, such that $i(f)$ is equal to $i$ and such that the following diagram commutes:
\[\begin{array}{rcccl}
&&f'&&\\
&M&\rightarrow&M&\\
i'&\uparrow&&\uparrow&i'\\
&L'&\rightarrow&L'&\\
&&f'^*&&\end{array}.\]
Furthermore, there exists a compact neighborhood $W$ of $L'$ such that, for every $f'\in V_f$, the maps $i(f')$ and $f'^*$ are equal to $i$ and $f^*$ respectively, on the complement of $W$.
\end{coro}
\begin{rema} We remark that, every map $f'\in V_f$, close enough to $f$, normally expands the lamination $(L',\mathcal L')$ embedded by $i(f')$ and is plaque-expansive. Thus, the hypotheses of this corollary are open.\end{rema}
\begin{exem}
Let $P:= x\mapsto x^2+c$ which has some repulsive compact subset $K$. For instance, $c$ can be a Collet-Eckmann parameter or a hyperbolic parameter. Let $f:= (x,y)\in \mathbb R^2\mapsto (x^2+c,0)$. The one dimensional embedded lamination $K\times \mathbb R$ is $r$-normally expanded by $f$ for any $r\ge1$. As this lamination is a bundle, $f$ is obviously plaque-expansive at this lamination. Let $R>0$ and let $L':= K\times ]-R,R[$. Then this lamination is $C^r$-persistent. This means that, for any $C^r$-perturbation $f'$ of $f$, there exists a $C^r$-embedding $i'$ of $K\times ]-R,R[$ into $\mathbb R^2$ such that $f'$ sends $i'(\{k\}\times ]-R,R[)$ into $i'(\{P(k)\}\times ]-R,R[)$, and $i'$ is $C^r$ close to $i$: $i'$ is uniformly $C^0$ close to $i$, the $r$-first partial derivatives of $i$ with respect to its second coordinate exists, are continuous and are uniformly close to those of $i$. \end{exem}       
\subsection{Main result on persistence of stratifications laminations}
\subsubsection{Problematics}
\subsubsection{Embedded stratifications of laminations}
Throughout this section, we denote by $r$ a fixed positive integer. 

Let $(A,\Sigma)$ be a stratified (laminar) space $C^r$-embedded by $i$ into a manifold $M$. We can identify the space $(A,\Sigma)$ with its image in $M$.

  A $C^r$-endomorphism $f$ of $M$ \emph{preserves} the stratification (of lamination) $\Sigma$ if
$f$ preserves each stratum of $\Sigma$ as an embedded lamination.

  The stratification $\Sigma$ is \emph{$C^r$-persistent} if every endomorphism $f'$ $C^r$-close to $f$ preserves a stratified embedding $i'$ of $(A,\Sigma)$, $C^r$-close to $i$ and such that the image by $f'$ of every point $i'(x)\in i'(A)$ belongs to the image by $i'$ of a small plaque containing $f(x)$ (of the stratum containing $x$).
This is equivalent to require the existence of an endomorphism $f'^*\in End_{f^*}^r(\Sigma)$ close to the restriction $f_{|A}\in End^r(\Sigma)$ such that the following diagram commutes:
\[\begin{array}{lll}
 & f'& \\
  M& \rightarrow &  M \\
  \uparrow i'& & \uparrow i'\\
  A&\rightarrow  &  A \\
 & f'^*& \\
\end{array}\]
In particular, this implies also that the dynamic induced by $f'$ and $f$ on the space of the leaves of each stratum $X$ are the same.

We notice that by ``close'' we mean close for the topologies of $Em^r(\Sigma ,M)$ or $End_{f^*_{|A}}^r(\Sigma)$ described in \ref{Strat:mor}.

If, moreover, $i$ is an embedding $a$-regular and the embedding $i'$ is also $a$-regular, for every $f'$ close to $f$, we say that {\it the $a$-regular stratification $\Sigma$ is $C^r$-persistent}

Our problematics are to found sufficient conditions implying the $C^r$-persistence of stratifications of laminations.
 
As we are dealing with endomorphisms, the normal expansion together with the plaque-expansiveness of each stratum appear to be good hypotheses.

 If we reject the hypothesis of plaque-expansiveness, we shall consider the immersions of stratifications of laminations.

\subsubsection{Immersed stratifications of laminations}

  A stratified (laminar) space $(A,\Sigma)$ $C^r$-immersed by $i$ into a manifold $M$ is \emph{$C^r$-preserved} by an endomorphism $f$ of $M$, if there exists an endomorphism $f^*\in End^r(\Sigma)$ such that the following diagram commutes:

\[\begin{array}{lll}
 & f& \\
  M& \rightarrow &  M \\
  \uparrow i& & \uparrow i\\
  A&\rightarrow  &  A \\
 & f^*& \\
\end{array}\]

Such endomorphism $f^*$ is the \emph{pullback} of $f$ (into $(A,\Sigma)$ via $i$).

Such an immersed stratified space is $C^r$-\emph{persistent} if, for any $C^r$-endomorphism $f'$ close to $f$, there exist an immersion $i'$ close to $i$ and an endomorphism $f'^*\in End_{f^*}^r(\Sigma)$ close to $f^*$ such that the following diagram commutes:

\[\begin{array}{lll}
 & f'& \\
  M& \rightarrow &  M \\
  \uparrow i'& & \uparrow i'\\
  A&\rightarrow  &  A \\
 & f'^*& \\
\end{array}\]

\subsubsection{Main result}\label{main:result}
Unfortunately, we will see in section \ref{cexp}, an example of an embedded (differentiable) compact stratified space, which is normally expanded but not persistent. Therefore some extra hypotheses are required.

We suspect the topology of the stratified space to play a main role. That is why throughout this section $(A,\Sigma)$ is a stratified space endowed with a $C^r$-trellis structure $\mathcal T$ (such structure does not exist in the example cited above), for a $r\ge 1$ now fixed. Moreover $M$ will refer to a Riemannian manifold.

Hypotheses of the main result need to generalize the notion of pseudo-orbits.

\begin{defi} Let $(L,\mathcal L)$ be a lamination, let $V$ be an open set of $L$, let $f$ be a continuous map from $V$ to $L$, and let $\epsilon$ be a continuous positive function on $V$.
A sequence $(p_n)_n\in V^\mathbb N$ is an \emph{ $\eta$-pseudo-orbit of $V$ which respects $\mathcal L$} if  $p_{n+1}$ and $f(p_n)$ belong to a same plaque of $\mathcal L$ of diameter less than $\epsilon(p_{n+1})$, for every $n\ge 0$.\end{defi}

We now state an useful corollary of the main theorem \ref{th2}.
\begin{coro}\label{cor3}
Let $i$ be a $\mathcal T$-controlled $C^r$-embedding of $(A,\Sigma)$ into $M$. We suppose that $A$ is compact.
We identify $(A,\Sigma)$ with its image by $i$ in $M$.
 Let $f$ be a $C^r$-endomorphism of $M$ preserving $(A,\Sigma)$. Moreover, we suppose that, for every stratum $X\in \Sigma$, there exists a neighborhood $V_X$ of $X$ in $A$ such that:

\begin{enumerate}[(i)]
\item each plaque of $\mathcal L_X$ included in $V_X$ is sent by $f$ into a leaf of $\mathcal L_X$, for every stratum $X$,
\item $f$ $r$-normally expands each stratum $X$,
\item for every stratum $X$, there exits a positive continuous function $\epsilon$ such that every  $\eta$-pseudo-orbit of $V_X$ which respects the plaques of $\mathcal L_{X}$ is contained in $X$,
\item $f$ is plaque-expansive at each stratum $X$.
\end{enumerate}

Then the stratification of laminations $(A,\Sigma)$ is persistent.

 Moreover, there exists a family of neighborhoods $(V_X')_{X\in\Sigma}$ adapted to $f$
such that, for every $f'$ $C^r$-close to $f$ (for the $C^r$-compact-open topology), there exists a $\mathcal T$-controlled $C^r$-embedding $i'$ of $(A,\Sigma)$ into $M$, close to $i$, which is preserved by $f'$ and satisfies:

\begin{enumerate}[(i)]
\item for every $x\in V_X'$, the endomorphism $f'$ sends $i'(x)$ into the image by $i'$ of the leaf of $f(x)$ in $\mathcal L_X$,
\item $f'$ $r$-normally expands the lamination $X$ embedded by $i'$,
\item there exits a positive continuous function $\epsilon'$ such that every $\eta'-f'$-pseudo-orbits of $V'_X$, which respects the plaques of the lamination $\mathcal L_{X}$ embedded by $i'$, is contained in $X$,
\item $f'$ is plaque-expansive at the lamination $X$ embedded by $i'$.
\end{enumerate}

\end{coro}
\begin{rema} Hypothesis $(i)$ says that the restriction $f_{|A}$ is a $\mathcal T$-controlled endomorphism and that the family $(V_X)_X$ is adapted to $f_{|A}$.\end{rema}

\begin{rema} 
The fact that $i'$ is a $\mathcal T$-controlled $C^r$-embedding $i'$ of $(A,\Sigma)$ into $M$, close to $i$, means that: 
\begin{itemize}
\item $i'$ is an homeomorphism of $A$ onto its image in $M$, $C^0$-close to the embedding $i$,
\item for each stratum $X$, the restriction of $i'$ to $L_X$ is an immersion of the lamination $(L_X,\mathcal L_X)$ and the $r$-first partial derivatives of $i'$ along the plaques of $\mathcal L_X$  are close to those of $i$ for the compact-open topology.\end{itemize}\end{rema}
\begin{rema}
Conclusion $(i)$ says that: 
\begin{itemize}
\item the pullback $f'^*$ of $f'$ via $i'$ is $\mathcal T$-controlled,
\item the family of neighborhoods $\mathcal V'=(V_{X}')_{X\in \Sigma}$ is adapted to this pull back and to the pull back $f_{|A}$ of $f$,
\item the endomorphism $f'^*$ belongs to the equivalence class $End_{f_{|A}\;\mathcal V'}^r(\mathcal T)$ of $f_{|A}$.\end{itemize}
Moreover, as $f'$ is close to $f$ and $i'$ close to $i$, the pull back of $f'$ via $i'$ is close to $f_{|A}$ for the topology of $End_{f_{|A}\;\mathcal V'}^r(\mathcal T)$.\end{rema}
\begin{rema} Conclusions $(i)$, $(ii)$, $(iii)$ and $(iv)$ imply that the hypotheses of this theorem are open.

 Moreover, as $(V_{X}')_{X\in \Sigma}$ does not depend on $f'$ $C^r$-close to $f$, conclusion $(i)$ appears to be very useful for the proof of the structural stability of non-hyperbolic compact subsets or the persistence of non-normally hyperbolic laminations (SA diffeomorphisms, SA bundles, extension of Shub's theorem on conjugacy of repulsive compact subset...).  
\end{rema}
Let us now give some easy applications of this corollary.

\begin{exem} \[\mathrm{Let}\;  f\;:\; \mathbb R^2\rightarrow \mathbb R^2\]
\[(x,y)\mapsto (x^2,y^2)\]
For any $r\ge 1$, the endomorphism $f$ $r$-normally expands the canonical stratification $\Sigma$ on the square $[-1,1]^2$ formed by the vertexes $X_0$, the edges $X_1$ and the interior $X_2$. Moreover $f$ is obviously plaque-expansive at each stratum of this stratification. Let $\mathcal T=(L_{X_i},\mathcal L_{X_i})_{i=0}^3$ be the trellis structure built in example \ref{carre}. We suppose that $L_{X_1}$ is disjoint from the diagonals of the square. Let $V_0$, $V_1$ and $V_2$ be equal to respectively $L_0\cap f^{-1}(L_0)\setminus [-1/2,1/2]^2$, $V_1\cap f^{-1}(V_1)\setminus [-1/2,1/2]^2$ and $L_2=X_2$. For these settings, hypotheses $(i)$, $(ii)$, $(iii)$ and $(iv)$ are satisfied. 

Thus, by corollary \ref{cor3}, this stratification is $C^r$-persistent: 

 For $f'\in End^r(\mathbb R^2)$ close enough to $f$, there exists a homeomorphism $i'$ from $[-1,1]^2$ onto its image in $\mathbb R^2$, whose restriction to each stratum $X_0$, $X_1$, $X_2$ is a $C^r$-embedding and such that $f'$ preserves the stratification $(i'(X_0), i'(X_1), i'(X_2))$ of $i([-1,1]^2)$. 

But, as we will see in section \ref{part varbor}, generally $i([-1,1]^2)$ is not diffeomorphic to $[-1,1]^2$. 

In fact, in section \ref{applications} we will systematize this example by showing, on the one hand,  the persistence of the canonical stratification of submanifolds with corners normally expanded and, on the other, the persistence of some product stratifications.\end{exem}

\begin{exem}{Viana Map}
 \[\mathrm{Let}\; V\;:\; \mathbb C\times \mathbb R\rightarrow \mathbb C\times \mathbb R\]
\[(z,h)\mapsto (z^2,h^2+c)\]

We fix $c\in ]-2,1/4[$. Therefore the map $h\mapsto h^2+c$ preserves an open interval $I$ and expands its boundary $\partial I$.

Thus, the endomorphism $V$ preserves the stratification $\Sigma$ of $C:= \{(z,h)\in \mathbb C\times \mathbb R;\; |z|\le 1\; \mathrm{and}\; h\in I\}$ formed by the strata $X_0:= \mathbb S^1$, $X_1:= \mathbb S^1\times I$, $X_2:= \mathbb D\times \partial I$ and $X_3:= \mathbb D\times I$ of dimension respectively $0$, $1$, $2$ and $3$.

We endow $(C,\Sigma)$ with the trellis structure $\mathcal T= (L_{X_i}, \mathcal L_{X_i})_{i=0}^3$ described in example \ref{previana}. 

We notice that $V_{|C}$ is $\mathcal T$-controlled, hence hypothesis $(i)$ of corollary \ref{cor3} is satisfied.

The endomorphism $1$-normally expands each stratum of $\Sigma$ (see \cite{BST} for estimation which implies the $1$-normal expansion of $X_1$). Thus, hypothesis $(ii)$ of the corollary is satisfied.    

Moreover, for an adapted family of tubular neighborhoods small enough, hypothesis $(iii)$ is also satisfied (for any functions $\eta$).  

Finally, as all the strata are bundles, $V$ is plaque-expansive at each of these laminations (see annex \ref{plaque-expansiveness}). Thus, hypothesis $(iv)$ is also satisfied. 

Therefore, by corollary  \ref{cor3}, the  $a$-regular stratification $\Sigma$ is $C^1$-persistent.

In other words, for every endomorphism $V'$ $C^1$-close to $V$, there exists a homeomorphism $i'$ of $cl(\mathbb D\times I)$ onto its image in $\mathbb C\times \mathbb R$, $C^0$-close to the canonical inclusion, such that for each stratum $X_k\in \Sigma$:
\begin{itemize}
\item the restriction $i'_{|X_k}$ is an embedding of the lamination, close to the canonical inclusion of $X_k$ in $\mathbb C\times \mathbb R$.
\item the lamination $i'(X_k)$ is preserved by $V'$ and, for $x\in X_k$, the point $V'\circ i'(x)$ belongs to the image by $i'$ of a small plaque of $X_k$ containing $V(x)$.
\end{itemize}

An artistic view of such a perturbation of this stratification is represented  figure \ref{cylindreconv}.


\end{exem}

More sophisticated applications of this corollary will be given in section \ref{applications}. 

The above corollary is a consequence of the following corollary \ref{cor2} of theorem \ref{th2} (for $A$ compact and $A'=A$). Now $A$ is no more supposed to be compact. We now use the notations explain in \ref{eq:co:mo}.

\begin{coro}\label{cor2}
Let $i$ be a $\mathcal T$-controlled $C^r$-embedding of $(A,\Sigma)$ into $M$. Let $f$ be a $C^r$-endomorphism of $M$ preserving $(A,\Sigma)$. We suppose that:
\begin{enumerate}[(i)]
\item $f_{|A}$ is $\mathcal T$-controlled,
\item $f$ $r$-normally expands each stratum $X$,
\item for every stratum $X$, there exits a positive continuous function $\epsilon$ on a neighborhood $V_X$ of $X$ in $A$ such that every  $\eta$-pseudo-orbit of $V_X$, which respects the plaques of $\mathcal L_{X}$, is contained in $X$,
\item $f$ is plaque-expansive at each stratum $X$.
\end{enumerate}

 Let $A'$ be a precompact open subset of $A$ such that $f^*(cl(A'))$ is included in $A'$. Then there exist a neighborhood $V_f$ of $f$ in $End^r(M)$, a family of neighborhoods $\mathcal V'$ adapted to $f_{|A'}$ and a continuous map 
\[V_f\rightarrow  Em^r (\mathcal T_{|A'},M )\]
\[f'\mapsto i(f')\]
with $i(f)= i$ and such that $(f',i(f'))$ satisfies the above properties $(i)$, $(ii)$, $(iii)$ and $(iv)$ for the stratified space $(A',\Sigma_{|A'})$ endowed with the trellis structure $\mathcal T_{|A'}$.
Moreover, for every $x\in V_X'$, the endomorphism $f'$ sends $i'(x)$ into the image by $i'$ of a small $\mathcal L_X$-plaque containing $f(x)$. In particular $(V_X')_X$ is adapted to $f'_{|i'(A)}$.

  In particular, the stratification $(A',\Sigma_{|A'})$ is persistent.
\end{coro}

\begin{rema}
The continuity of the map $f'\mapsto i(f')$ means that for $f'$ close to $f''$ in $V_f$, for any stratum $X\in \Sigma_{|A'}$, any compact subset $K\subset A'\cap L_{X}$, the elements $i(f')(x)$ and $\partial^s_{T_x\mathcal L_X}i(f')$  are uniformly close, for $x\in K$ and $s\in \{1,\dots, r\}$, to respectively $i(f'')(x)$ and $\partial_{T_x\mathcal L_X}^s i(f'')$.\end{rema}

The following theorem is the main result of this memory.

\begin{theo}\label{th2}

 Let $f$ be a $C^r$-endomorphism of $M$, $i$ be a $\mathcal T$-controlled $C^r$-immersion of $(A,\Sigma)$ into $M$ and $f^*$ be a $\mathcal T$-controlled $C^r$-endomorphism such that:
\begin{enumerate}[(i)]
    \item the following diagram commutes
$\begin{array}{lll}
 & f& \\
  M& \rightarrow &  M \\
  \uparrow i& & \uparrow i\\
  A&\rightarrow  &  A \\
 & f^*& \\
\end{array}$,
  \item $f$ normally expands each stratum $X$ immersed by $i_{|X}$  and over $f^*_{|X}$,

 \item for every stratum $X\in \Sigma$, there exist an adapted neighborhood $V_X$ of $X$ and a continuous positive function $\eta$ on $V_X$,  such that every $\eta$-pseudo-orbit of $V_X$ which respects $\mathcal L_{X}$ is contained in $X$.
\end{enumerate}

 Let $A'$ be a precompact open subset of $A$ such that $f^*(cl(A'))\subset A'$. Then $f^*_{|A'}$ is a $\mathcal T_{|A'}$-controlled $C^r$-endomorphism of class $C^r$,  there exist a neighborhood $V_f$ of $f$ in $End^r(M)$, a family of neighborhoods $\mathcal V'$ adapted to $f_{|A'}^*$ and a continuous map 
\[V_f\rightarrow End_{f_{|A'}^*\mathcal V'}^r(\mathcal T_{|A'})\times Im^r (\mathcal T_{|A'},M )\]
\[f'\mapsto (f'^*,i(f'))\]
with $i(f)= i$ and such that $(f',i(f'),f'^*)$ satisfies the above properties $(i)$, $(ii)$ and $(iii)$ for the stratified space $(A',\Sigma_{|A'})$ endowed with the trellis structure $\mathcal T_{|A'}$.

 In particular, $f'$ preserves the stratification of laminations $\Sigma_{|A'}$ immersed by $i(f')$  and,
 for every $X\in \Sigma_{|A'}$, each point $x\in V_X'$ is sent by $f'^*$ into a small plaque of $\mathcal L_X$ containing $f^*(x)$.
 
 In other words, the immersed stratifications $\Sigma_{|A'}$ is persistent. 
\end{theo}

\begin{rema}
 The continuity of the map $f'\mapsto f'^*$ means that for $f'$ $C^r$-close to $f''$ in $V_f$, for any stratum $X\in \Sigma_{|A'}$, any compact subset $K\subset V'_{X}$, the elements $f'^*(x)$ and $\partial_{T_x\mathcal L_X}^s f'^*$  are uniformly close, for $x\in K$ and $s\in \{1,\dots ,r\}$, to respectively $f''^*(x)$ and $\partial_{T_x\mathcal L_X}^sf''^*$.\end{rema}

\begin{ques}
\begin{itemize}
\item The counterexample in \ref{cexp} shows that the existence of a trellis structure, at least locally, seems to be important.
However, is it necessary that such a structure controls $f^*$ (or $i$), to imply the persistence of an embedded stratification? In proof of theorem \ref{cor2}, this hypothesis is used only in the lemma of \ref{lem6}. Without this hypothesis this theorem would be much more easier to apply to several cases.
\item When $i$ is an embedding, is hypothesis $(iii)$ always satisfied?  Under the hypotheses of theorem \ref{th2}, is this hypothesis necessary?

The first question could be a first step to the construction of a counterexample of a normally expanded embedded lamination, but not plaque-expansive.
\item Given an $a$-regular stratification of normally expanded laminations, does the existence of a (local) trellis structure is linked to extra dynamic conditions?

For example, given a diffeomorphism that satisfies axiom $A$ and the strong transversality condition, the stratification 
$(W^s(\Lambda_i))_i$ (see \ref{desstruc}) locally admits a trellis structure.

 \end{itemize}
\end{ques}
\subsection{A normally expanded but not persistent stratification}
\label{cexp}

Let us present, for all $r\ge 1$,  an example of an $a$-regular compact differentiable stratification with $r$-normally expanded strata, but not topologically persistent. This means that there exists $f'$ $C^\infty$-close to $f$, which does not preserve the image of each stratum by any homeomorphism $C^0$-close to the canonical inclusion of the support. We will notice that this stratified space cannot support a trellis structure (even locally).

Let $\mathbb S^1$ be a circle embedded into $\mathbb R^3$ and $r$-normally hyperbolic for a diffeomorphism $f$ of $\mathbb R^3$. We suppose that the strong stable dimension is $1$. According to \cite{HPS}, the union of the strong stable manifolds of the circle is an immersed manifold, that we denote by $W^s$. We suppose that the point $0\in \mathbb R^3$ is fixed by $f$ and that the restriction of $f$ to $]-1,1[^3$ is equal to 
\[f_{|]-1,1[^3} \;:\; ]-1,1[^3\rightarrow \mathbb R^3\]
\[(x,y,z)\mapsto(x+x^3,2y,2z)\]
 
Thus, the point $0$ is topologically repulsive for $f$. We suppose that $W^s$ without the circle $\mathbb S^1$ is contained in the repulsive basin of $0$. Therefore, $W^s$ is a manifold embedded into $\mathbb R^3$.

Let us suppose that the restriction of $f$ to the circle $\mathbb S^1$ has a repulsive fixed point and that the stable manifold of this point (in $\mathbb R^3$) intersects $]-1,1[^3$ at  $]-1,1[\times \{0\}^2\setminus \{0\}$.

  Then the union of $0$ with this stable manifold is a circle $X$ differentially $C^r$-embedded into $\mathbb R^3$. Let $Y$ be the submanifold $W^s\setminus X$.

  We may suppose that the intersection of $W^s$ with \[\Big[-\frac{1}{2}-\frac{1}{8},-\frac{1}{2}\Big]\cup\Big[\frac{1}{2},\frac{1}{8}+\frac{1}{2}\Big]\times \Big]-1,1\Big[^2\] is an union of graph of maps from $[-\frac{1}{2}-\frac{1}{8},-\frac{1}{2}]\cup[\frac{1}{2},\frac{1}{8}+\frac{1}{2}]$ into $]-1,1[^2$.

   Thus, the partition $\Sigma:=(X,Y)$ on $A:=X\cup Y$ is an $a$-regular stratification of $\mathbb R^3$, $r$-normally expanded by $f$. We draw in figure \ref{le cexple} how this stratification looks-like.

\begin{figure}
    \centering
        \includegraphics[width=7cm]{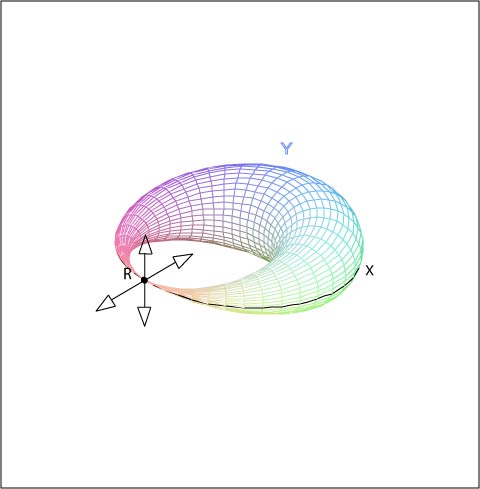}
    \caption{Stratification of normally expanded laminations, which is not persistent}
    \label{le cexple}
\end{figure}

If this stratified space (restricted to a neighborhood of $0$) could admit a trellis structure, then a small neighborhood of $0$ in $A$ would be homeomorphic to the product of a neighborhood of $0$ in $X$ with the intersection of $A$ with a  plan transverse to $X$. This last product can be a segment, which is not homeomorphic to any neighborhood of $0$ in $A$, because a segment does not contain any surface.

  We suppose, for the sake of contradiction, that the stratification  $(X,Y)$ is topologically $C^r$-persistent. This means that for every diffeomorphism $f'$ $C^r$-close to $f$ there exists a homeomorphism $p$, $C^0$-close to the canonical inclusion, such that $h(X)$ and $h(Y)$ are $f'$-stable.

  We build now a family of $C^\infty$-perturbations  of $f$ which contradicts this persistence hypothesis. Let $\rho$ a $C^\infty$-function with support in $]-1,1[$ and such that its restriction to $]-\frac{1}{2},\frac{1}{2}[$ is equal to $1$. For any small $t\ge 0$, let $f_t$ be the diffeomorphism of $\mathbb R^3$ equal to $f$ on the complement of $]-1,1[^3$ and such that its restriction  to $]-1,1[^3$ is equal to 
\[f_{t|]-1,1[^3}\;:\; ]-1,1[^3\rightarrow \mathbb R^3\]
\[(x,y,z)\mapsto f(x,y,z)+(-t\cdot \rho(x)\cdot x,0,0)\]
We notice that $f_0$ is equal to $f$.

 For $t$ small enough, let $(X(t),Y(t))$ be the $f_t$-stable stratification given by the persistence hypothesis.

Let us prove that the stratum $X(t)$ is equal to $X$ for $t$ small enough. We note that each diffeomorphism $(f_t)_t$ preserves and $0$-normally expands $X$. Thus, there is a neighborhood $V$ of $X$ such that, for $t$ small enough, the intersection $\cap_{n\ge 0} f_t^{-n}(V)$ is equal to $X$. But, for $t$ small enough, the stratum $X(t)$ is included in $V$. By $f_t$-stability of the stratum $X(t)$, we have
\[X(t)\subset \cap_{n\ge 0} f_t^{-n}(U)=X\]
since the stratum $X(t)$ is compact, this stratum is a close subset of $X$. Since this stratum has the same dimension as $X$, this stratum is an open subset of $X(t)$. Therefore, by connectivity, $X(t)$ is equal to $X$.

For every $r\in]0,\frac{1}{2}[$ small enough, the set $Y$ only intersects the faces $\{-r\}\times ]-r,r[^2$ and $\{r\}\times ]-r,r[^2$ of the cube $[-r,r]^3$. The same is satisfied by $Y(t)$, for $t$ small enough.

We can also suppose that $t$ is less than $r^2<\frac{1}{4}$. This implies that the interval $]-\sqrt{t},\sqrt{t}[$ is sent into itself by the map 
\[\phi_t\;:\;x\mapsto x+x^3-t\cdot \rho(x)\cdot x.\]
We notice that the restriction $f_{t|]-1,1[^3}$ is equal to $(x,y,z)\mapsto(\phi_t(x),2y,2z)$.

The compact set $X\cup Y$ is locally connected and the closure of $Y$ contains $X$. Since these properties are invariant by homeomorphism, they are also satisfied by  $X(t)$ and $Y(t)$.

 Thus, there exists a path $\gamma$ included in $]-\sqrt{t},\sqrt{t}[\times [-r,r]^2\cap Y(t)$ and containing $0\in X=X(t)$ in its closure.
We are going to show that the $f_t$-orbit of $\gamma$ intersects another face of the cube  $[-r,r]^3$ than $\{-r\}\times ]-r,r[^2$ and $\{r\}\times ]-r,r[^2$. As $Y(t)$ is $f_t$-stable, this would imply that $Y(t)$ intersects another face of  $[-r,r]^3$ than $\{-r\}\times ]-r,r[^2$ or $\{r\}\times ]-r,r[^2$. This is a contradiction.

As $\gamma$ is included in the repulsive basin of $X$, there exists a first integer $n$ such that $f_t^n(\gamma)$ intersects the complement of $]-r,r[^3$. Since the set $]-\sqrt{t},\sqrt{t}[$ is $\phi_t$-stable and  $r$ is less than $\frac{1}{2}$, it follows that $f^{n}_t(\gamma)$ is included in $]-\sqrt{t},\sqrt{t}[\times ]-1,1[^2$ and intersects $]-\sqrt{t},\sqrt{t}[\times (]-1,1[^2\setminus ]-r,r[^2)$.
As $0$ is a fixed point of $f_t$, it belongs to the closure of $f_t^n(\gamma)$. By connectivity, there exists a point of $f_t^n(\gamma)$ whose second or third coordinate are equal to $-r$ or $r$.  But $\sqrt{t}$ is less than $r$, so the path $f_t^n(\gamma)$ intersects the boundary of $[-r,r]^3$ in other faces than $\{-r\}\times ]-r,r[^2$ and $\{r\}\times ]-r,r[^2$.

\subsection{Consequences of the main result (theorem $\ref{th2}$)}\label{applications}
\subsubsection{Submanifolds with boundary}
\label{part varbor}
\begin{theo}
Let $(M,g)$ be a Riemannian manifold and let $N$ be a compact submanifold with boundary of $M$. Let $f$ be a  $C^1$-endomorphism of $M$ which preserves and $1$-normally expands the boundary $\partial N$ and the interior $\mathring{N}$ of $N$.
Then the stratification $(\mathring N,\partial N)$ on $N$ is persistent.
\end{theo}
\begin{rema}
 In other words, the above theorem concludes that, for any map $f'$ $C^1$-close to $f$, there exist two submanifolds $\partial N'$ and $\mathring N'$ such that:\begin{itemize}
\item  $\mathring N'$ (resp. $\partial N'$) is preserved by $f'$, diffeomorphic and close to $\mathring N$ (resp. $\partial N$) for the compact-open $C^1$-topology,
\item the pair $(\mathring N',\partial N')$ is a stratification (of laminations) on $N':= \mathring{N'} \cup \partial N'$,
\item the set $N'$ is the image of $N$ by an embedding $C^0$-close to the canonical inclusion of $N$ into $M$.
\end{itemize}
\end{rema}
\begin{proof}[Idea of proof] We build a trellis structure on $(N,(\partial N,\mathring N))$ which satisfies properties $(i)$ and $(iii)$ of corollary \ref{cor3}. As other properties $(ii)$ and $(iv)$ are obviously checked, the corollary implies the $C^1$-persistence of the stratification. Details of this proof are in \cite{PB}; we hope to publish them later.\end{proof}

\begin{rema}
Usually, $N'$ is not a submanifold with boundary.
\end{rema}

\subsubsection{Submanifolds with corners}
The above theorem can be generalized to submanifold with corners.

 We recall that a compact manifold with corners $N$ is a differentiable manifold modeled on  $\mathbb R_+^d$. A subset $N$ of manifold $M$ (without corner) is a submanifold with corners if there exists charts $(\phi_\alpha)_\alpha$ of $M$ whose restrictions to respectively $(\phi_\alpha^{-1}(\mathbb R^d\times \{0\}))_\alpha$ form an atlas of manifold with corners (for a fixed integer $d$). For example,  a cube, a product of manifolds
with boundary or a generic intersection of submanifolds are endowed with a canonical structure of manifold with corners.

 We denote by $\partial^{0_k}N$ the set of points in $N$ which, seen in a chart, have exactly $k$ coordinates equal to zero. The pair $(N,\Sigma:=\{\partial^{0_k} N\})$ is a stratified space.
\begin{theo}
 Let $N$ be a compact submanifold with corners of a manifold $M$. Let $f$ be a $C^1$-endomorphism of $M$, which preserves and $1$-normally expands each stratum  $\partial^{0_k} N$.
Then the stratification $\Sigma$ on $N$ is $C^1$-persistent.
\end{theo}
\begin{rema}
 In other words, the above theorem concludes that, for every endomorphism $f'$ $C^1$-close to $f$, there exists submanifolds $(\partial^{0_k} N')_k$ such that:\begin{itemize}
\item for each $k$, $\partial^{0_k} N'$ is preserved by $f'$, is diffeomorphic, and is $C^1$-close to $\partial^{0_k} N$ for the compact-open topology,
\item the family  $(\partial^{0_k} N')_k$ is a stratification (of laminations) on $N':=\cup_k \partial^{0_k} N'$,
\item the set  $N'$ is the image of $N$ by an embedding $C^0$-close to the canonical inclusion of $N$ into $M$.
\end{itemize}
\end{rema}
\begin{proof}[Idea of proof] We build a trellis structure on $(N,\Sigma)$ which satisfies  properties $(i)$ and $(iii)$ of corollary \ref{cor3} (this is far to be obvious). As the other properties $(ii)$ and $(iv)$ are obviously satisfied, the corollary implies the $C^1$-persistence of the stratification. The complete proof is in \cite{PB} ; we hope to publish it later.\end{proof}

\begin{rema}
Usually, $N'$ is not an embedded submanifold with corner.
\end{rema}

\subsubsection{Extension of the Shub's theorem on conjugacy of repulsive compact set}
The same corollary \ref{cor3} implies a complement of M. Shub's celebrated result \cite{Shubthese} on the structural stability of expanded compact sets for endomorphisms:
the conjugacy between the compact set and its continuation can be extended to a homeomorphism of the ambient space, which is still a conjugacy in the neighborhood of the compact set. C. Robinson already proved a similar result for locally maximal hyperbolic sets of diffeomorphisms.
\begin{coro}\label{K}
Let $r\ge 1$, let $M$ be a compact Riemannian manifold, let $f$ be a $C^r$-endomorphism of $M$, and let $K$ be a compact subset of $M$ that satisfies 
\[f^{-1}(K)= K.\]

Moreover, we suppose that $f$ expands $K$, that is, for every $x\in K$, the differential $T_xf$ is invertible and with inverse contracting.

Then there exists a neighborhood $V_K$ of $K$  such that, for every $f'$ $C^r$-close to $f$, there exists a homeomorphism $i'$ of $M$ close to the identity such that
\[\forall x\in V_K, \quad f'\circ i'(x)=i'\circ f(x).\]
Moreover the restriction $i(f')_{|K^c}$ belongs to $Diff^r(M\setminus K, M\setminus i(f')(K))$ and is $C^r$-close to the identity (for the compact-open topology). \end{coro}

This corollary is, in a way, the (regular) analogous theorem for endomorphisms of the following (\cite{Rs}, Theorem 4.1):
\begin{theo}[Robinson 1975']
Let $f\;:\; M\rightarrow M$ be a $C^1$-diffeomorphism of a compact manifold $M$. Let $K$ be a compact subset of $M$ which is $f$-invariant and which has a local product structure. Then there exist a neighborhood $V_K$ of $K$ and a $C^1$-neighborhood $V_f$ of $f$ such that, if $f'$ belongs to $V_f$, then there exists a homeomorphism $h$ from $V_K$ onto its image, satisfying $h\circ f=f'\circ h$. Moreover, when $f'$ is $C^1$-close to $f$, then $h$ is $C^0$-close to the canonical inclusion.
\end{theo}

To show corollary \ref{K}, we will use the following lemma, which will be useful in other context.
\begin{lemm}\label{cpctrep}
Let $M$ be a Riemannian $n$-manifold and $f$ be a $C^1$-endomorphism of $M$. Let $A$ be a compact subset of $M$, which is the union of a compact subset $K$ with an open subset $X$ disjoint from $K$, and such that 
\[K\subset cl(X),\; f(K)\subset K,\; f(X)\subset X\]

Let us suppose that  $f$ is expanding on $K$.

We endow $K$ with the $0$-dimensional lamination structure and $X$ with the $n$-dimensional lamination structure.

Then $(A,(K,X))$ is a stratified space whose canonical embedding is preserved and normally expanded by $f$, and there exists a trellis structure on it such that hypotheses $(i)$, $(ii)$, $(iii)$ and $(iv)$ of corollary \ref{cor3} are satisfied.
\end{lemm}
\begin{rema} As $K$ is a $0$-dimensional lamination structure (and $M$ a $C^\infty$-manifold), the trellis structure is of class $C^r$ and $f$ is $r$-normally expands $(A,(K,X))$, for every $r\ge 1$.\end{rema}

\begin{proof}[Proof of lemma \ref{cpctrep}]
As $f_{|A}^{-1}(K)=K$ and as $f$ is expanding on $K$, there exists an open neighborhood $L_K$ of $K$ in $A$, which satisfies
$\cap_{n\ge 0}f_{|A}^{-1}(L_K)=K$.

 The subset $L_K$ endowed with the $0$-dimensional lamination structure $\mathcal L_K$ and $X$ endowed with the $n$-dimensional lamination structure form a trellis structure on the stratified space $(A,(K,X))$, which obviously controls $f_{|A}$. That is why hypothesis $(i)$ is satisfied.

Let $V_K$ be the  open neighborhood of $K$ in $A$, equal to $f_{|A}^{-1}(L_K)\cap L_K$. As the pseudo-orbits of $V_K$  which respect $\mathcal L_K$ are orbits of $f$ in $V_K$
and, as $\cap_{n\ge 0}f^{-1}(V_K)$ is equal to $K$, hypothesis $(iii)$ of corollary \ref{cor3} is satisfied. Moreover, an endomorphism
which is expanding on a compact set is necessarily expansive, thus $f$ is plaque-expansive at each stratum $K$ and $X$. That is why hypothesis $(iv)$ is satisfied.\end{proof}

\begin{proof}[Proof of lemma \ref{K}]
First of all, let us show that the compact set $K$ split into a compact set $K_1$ nowhere dense and  a union $K_0$ of connected components of $M$.

We suppose the contrary, for the sake of contradiction. Thus, the boundary $\partial K$ of $K$ is non-empty. As $K$ is $f$-invariant ($f^{-1}(K)=K$) and as $f$ is open on a neighborhood of $K$, we have the $f$-invariance of $\partial K$. We endow $M$ with an adapted metric to the expansion of $K$. We denote by $B(\partial K,\epsilon)$ the $\epsilon$-neighborhood of $\partial K$.
For $\epsilon>0$ small enough, by expansion of $\partial K$, the closure of $f^{-1}(B(\partial K,\epsilon))$ is included in $B(\partial K,\epsilon)$.
 Let $U:= K\setminus f^{-1}(B(\partial K,\epsilon))$ which, for $\epsilon>0$ small enough, is non-empty by the contradictory hypothesis. Moreover, the map $f$ sends $U$ into the interior of $U$. Let $U_n:=f^{n}(U)$. The sequence of compact subsets $(U_n)_n$ is decreasing and so converges, for the Hausdorff distance, to $U_\infty:= \cap_{n\ge 0}U_n$. As the restriction of $f$ to $K$ is open, for every $n\ge 0$, the compact subset $U_{n+1}$ is included in the interior of $U_n$. For every $n\ge 0$, let
$\epsilon_n>0$ be the maximal radius of a ball centered on the boundary of $U_\infty$ which is included in $U_n$. It follows from the convergence of $(U_n)_n$, that the sequence $(\epsilon_n)_n$ converges to $0$ but, by expansion of $f$, for $\epsilon_n$ small enough, the real $\epsilon_{n+1}$ is greater than $\epsilon_n$. This is a contradiction.

 We endow the compact subsets $K_0$ and $K_1$ with the $0$-dimensional lamination structure. Let $X$ be the open subset $M\setminus K$ of $M$, endowed with the lamination structure of the same dimension as $M$.

Therefore,  $(M,\Sigma:=\{K_0,K_1,X\})$ is a stratified space whose strata are normally expanded.
It follows from lemma \ref{cpctrep}, that hypotheses $(i)$, $(ii)$, $(iii)$ and $(iv)$ of corollary  \ref{cor3} are satisfied, for a trellis structure $\mathcal T$ on $(M,\Sigma$). This corollary provides a $C^r$-neighborhood $V_f$ of $f$ and a neighborhood $V_{K_1}'$ of $K_1$ such that, for $f'$ $C^r$-close to $f$, there exists $C^r$-embedding $i'$ from $(A,\Sigma)$ into (and hence onto) $M$, close to the identity, such that conclusion $(i)$ holds:

For any $x\in V_{K_1}\cup K_0$, the map $f'$ sends $i'(x)$ into the image by $i'$ of the leaf of $f(x)\in {K_0}\cup L_{K_1}$.

 As this leaf is 0-dimensional, we have obtained:
\[\forall x\in V_{K_1}\cup K_0,\quad f'\circ i'(x)=i'\circ f(x)\]
We note that $V_K=V_{K_1}\cup K_0$ is a neighborhood of $K$.

As $i'$ is a stratified embedding $C^r$-close to the identity, $i'$ is $C^0$ close to the identity and the restriction $i'_{|K^c}$ is $C^r$-close to the identity of $K^c$ for the compact-open topology.
 \end{proof}

\subsubsection{Product of stratifications of laminations}
The following proposition provides several examples of persistent stratifications in product dynamics.
\begin{prop}\label{prod}
Let $M$, $(A,\Sigma)$, $\mathcal T$, $f$, $i$ and $f^*$ be respectively a manifold, a stratified space, a trellis structure on $(A,\Sigma)$, an endomorphism of $M$, a $\mathcal T$-controlled immersion and a $\mathcal T$-controlled endomorphism that satisfy hypotheses $(i)$ and $(iii)$ of theorem \ref{th2}.

Let $M'$, $(A',\Sigma')$, $\mathcal T'$,  $f'$, $i'$, and $f'^*$ be respectively a manifold, a stratified space, and a trellis structure on $(A',\Sigma')$, an endomorphism of $M'$, a $\mathcal T'$-controlled immersion and a $\mathcal T'$-controlled endomorphism that satisfy hypotheses $(i)$ and $(iii)$ of theorem \ref{th2}.

We denote by $(f,f')$ and $(f^*,f'^*)$ the product dynamics on $M\times M'$ and on $A\times A'$.
We denote by $(i,i')$ the immersion of the product stratified space $(A\times A',\Sigma\times \Sigma')$ into $M\times M'$. 

Then there exists a trellis structure $\mathcal T_{prod}$ on the product stratified space such that  $(f^*,f'^*)$ and $(i,i')$ are $\mathcal T_{prod}$-controlled and satisfy hypotheses $(i)$ and $(iii)$ with $(f,f')$.

Moreover if $i$ and $i'$ are embedding and if $f$ and $f'$ are plaque-expansive at each stratum of respectively $\Sigma$ and $\Sigma'$, then $(f,f')$ is plaque-expansive at each stratum of $\Sigma\times \Sigma'$.
\end{prop}
\begin{proof} Let $\Sigma\times \Sigma'$ be the product stratification on $A\times A'$ defined in \ref{exprod} and $\mathcal T_{prod}$ be the trellis structure defined in  \ref{Tprod}. Let us show that this structure controls $(f,f')$.
For each strata $(X,X')\in \Sigma\times \Sigma'$, there exist neighborhoods $V_X$ and $V_{X'}$ of $X$ and $X'$ adapted to respectively  $f^*$ and $f'^*$. This means that the restrictions $f_{|V_X}^*$ and $f'^*_{|V_{X'}}$ are morphism from  $\mathcal L_{X|V_X}$ and $\mathcal L_{X'|V_{X'}}$ to respectively $\mathcal L_{X}$ and $\mathcal L_{X'}$. Then the products dynamics $(f^*,f'^*)$ of $A\times A'$ restricted to $V_X\times V_{X'}$ is a morphism from the product lamination  $\mathcal L_{X|V_X}\times \mathcal L_{X'|V_{X'}}$ to $\mathcal L_{X}\times \mathcal L_{X'}$. Let $V_{X\times X'}:= (V_X\times V_{X'})\cap L_{X\times X'}\cap (f^*,f'^*)^{-1}(  L_{X\times X'})$ be the adapted neighborhood of $X\times X'$.
The products dynamics $(f^*,f'^*)$ restricted to $V_{X\times X'}$ is a morphism from the product lamination  $\mathcal L_{X\times X'|V_{X\times X'}}$ to $\mathcal L_{X\times X'}$, since the lamination
$\mathcal L_{X\times X'}$ is a restriction of $\mathcal L_{X}\times \mathcal L_{X'}$.
 
Thus, the endomorphism $(f^*,f'^*)$ is $\mathcal T_{prod}$-controlled.

Let us check hypothesis $(iii)$ of theorem \ref{th2}. Let $X\times X'$ be a stratum of $\Sigma\times \Sigma'$. Let $\eta$ and $\eta'$ be the functions on respectively $V_X$ and $V_{X'}$ provided by hypothesis $(iii)$. Let $\eta_{prod}$ be the function on $V_{X\times X'}$ defined by 
\[\eta_{prod}\;:\; (x,x')\in V_{X\times X'}\mapsto \min(\eta(x),\eta'(x')).\]
Let $(x_n)_n$ be an $\eta_{prod}$-pseudo-orbit of $V_{X\times X'}$ which respects $\mathcal L_{X\times X'}$. By projecting canonically to $A$ and $A'$, we obtain an $\eta$-pseudo-orbit of $V_X$ which respects $\mathcal L_X$ and an $\eta'$-pseudo-orbit of $V_{X'}$ which respects $\mathcal L_{X'}$.  Hypothesis $(iii)$ implies that these two last pseudo-orbits belong to respectively $X$ and $X'$. Therefore, the pseudo-orbit $(x_n)_n$ belongs to $X\times X'$. Thus, hypothesis $(iii)$ is satisfied.

We show similarly the plaque-expansiveness.
\end{proof}

\subsubsection{Example of persistent stratifications of laminations in product dynamics}\label{expprod}

\subsubsection{Product of quadratic hyperbolic polynomials}
\[\mathrm{Let}\; f\;:\; \mathbb R^n\rightarrow \mathbb R^n\]
\[(x_i )_i\mapsto (x_i^2+c_i)_i\]

 We choose $(c_i)_i\in [-2,1/4]^n$, such that, for each $i$, the endomorphism $f_i\;:\;x\mapsto x^2+c_i$ has an attractive periodic orbit. Therefore, the trace of the (non filled) Julia set is an expanding compact set $K_i$.

According to Graczyk-{\'S}wiatek \cite{grac} and Lyubich \cite{Lyu}, this is the case for an open and dense set of parameters $(c_i)_i\in [-2,1/4]^n$.

The map $x\mapsto x^2+c_i$ normally expands the stratification of laminations $\Sigma_i$ formed by the $0$-dimensional lamination supported by $K_i$ and the $1$-manifold $X_i$ supported by $\mathbb R\setminus  K_i$ without its unbounded connected components.

It follows from lemma \ref{cpctrep} that $f$ is a product of maps $f_i$ which satisfy hypotheses $(i)$, $(ii)$, $(iii)$ and $(iv)$ of corollary \ref{cor3} with the stratification $\Sigma_i$.

We note that the product stratification $\prod \Sigma_i$ consists of the strata $(Y_J)_{J\subset \{1,\dots ,n\}}$, with $Y_J$ the lamination of dimension the cardinal of $J$ and  of support 
\[ \prod_{j\in J}X_j\times \prod_{j\in J^c}K_j.\]
The leaves of $Y_J$ are in the form $\prod_{j\in J} C_j\times \prod_{j\in J^c} \{k_j\}$, with $C_j$ a connected component of $X_j$ and $k_j$ a point of $K_j$.

Since $f$ $r$-normally expands the product stratification  $\prod \Sigma_i$, by applying $n-1$-times the proposition  \ref{prod} and finally the corollary \ref{cor3}, we show the $C^r$-persistence of this $a$-regular stratification of laminations, for every $r\ge 1$.

Figure \ref{fig:cantore} is a numerical experimentation which sheds light the persistent stratification of a $C^\infty$-perturbation of $f$, for $n=2$ and $c_1=c_2=-1$. The curves form the one-dimensional strata which spiral at each intersection at a exponential speed, that is why this spiraling is imperceptible.

\begin{figure}[h]
    \centering
        \includegraphics[width=7cm]{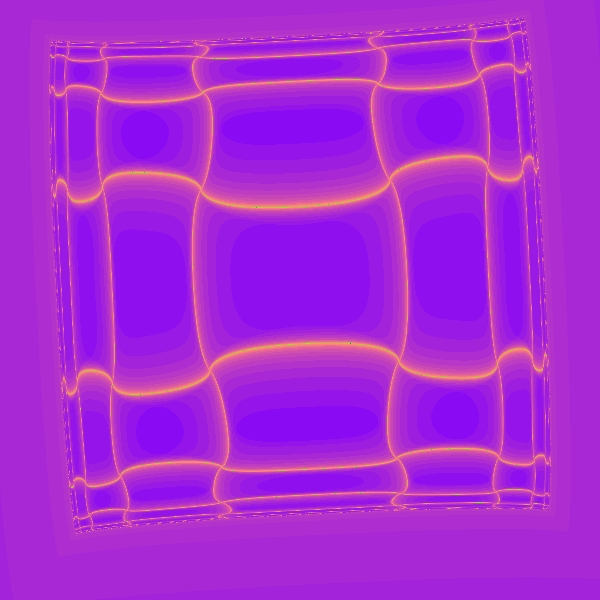}
    \caption{Numerical experimentation of the example \ref{expprod}.2}
    \label{fig:cantore}
\end{figure}

In dimension two, it is actually  J.-C. Yoccoz who remarked the persistence of such a family of curves which spiral at the crossing points. It is this example who motives the presented theory.

\subsubsection{Products of hyperbolic rational functions}\label{ex3}
\[\mathrm{Let}\; f\;:\; \hat {\mathbb C}^n\rightarrow \hat{\mathbb C}^n\]
\[\qquad (z_i )_i\mapsto (R_i(z_i))_i\]

 We assume that, for each $i$, $R_i$ is a hyperbolic rational function of the Riemann sphere $\hat{\mathbb C}$. It follows that its Julia set $K_i$ is an expanded compact subset. The complement $X_i$ of $K_i$ in $\hat{ \mathbb C}$ is the union of attraction basins of the attracting periodic orbits. The map $R_i$ normally expands the stratification of laminations $\Sigma_i$ consisting of the $0$-dimensional lamination support by $K_i$ and of the $2$-dimensional lamination supported by $X_i$.

We notice that the product stratification $\prod_i \Sigma_i$ consists of the strata $(Y_J)_{J\subset \{1,\dots,n\}}$, where the stratum $Y_J$ is of dimension $2$ times the cardinal of $J$ and with support $\prod_{j\in J}X_j\times \prod_{j\in J^c}K_j$. The leaves of $Y_J$ are in the form $\prod_{j\in J} C_j\times \prod_{j\in J^c} \{k_j\}$, with $C_j$ a connected component of $\hat{ \mathbb C}\setminus K_j$ and $k_j$ an element of $K_j$.

For the same reason as above, the $a$-regular stratification $\Sigma$ is $C^r$-persistent, for every $r\ge 1$.

\subsubsection{Lamination normally axiom $A$}\label{axiom A}

In the diffeomorphism context, we would love to unify two remarkable theorems, that we are going to recall.

The first is the following:
\begin{theo}[Hirsch-Pugh-Shub \cite{HPS}]
Let $(L,\mathcal L)$ be a compact lamination embedded into a manifold $M$. Let $f$ be a diffeomorphism of $M$ which lets invariant $L$ ($f(L)=L$) and preserves $\mathcal L$.
If $f$ is $r$-normally hyperbolic and plaque-expansive at $\mathcal L$ then the embedded lamination is $C^r$-persistent.\end{theo}

\label{diffexp} We recalled the definition of plaque-expansiveness in the diffeomorphism context in section \ref{plaque-expansiveness}.
Let us be more precise about normal hyperbolicity.

\begin{defi} Let $(L,\mathcal L)$ be a compact lamination embedded into a manifold $M$. Let $f$ be a diffeomorphism of $M$ preserving $(L,\mathcal L)$. The diffeomorphism $f$ is \emph{$r$-normally hyperbolic} to $(L,\mathcal L)$ if the support $L$ is $f$-invariant $(f(L)=L)$ and if there exist two subbundles $E^s$ and $E^u$ of the restriction of tangent
bundle of $M$ to $L$, such that:
\begin{itemize}
\item $E^s$ and $E^u$ are $Tf-$invariant,
\item $E^s\oplus T\mathcal L\oplus E^u=TM_{|L}$,
\item there exists $\lambda < 1$ satisfying for all $x\in L$, $u\in T_x\mathcal L\setminus \{0\}$ and $v\in E^u(x)$:\end{itemize}
\[\left\{\begin{array}{c}
   \|T_xf_{|E^s}\|\le \lambda\cdot \min\left(1,\frac{\|Tf_{|T_x\mathcal L}(u)\|^r}{\|u\|^r}\right)\\
   \\
   \lambda\cdot \|T_xf(v)\|\ge  \max(1,\|Tf_{|T_x\mathcal L}\|^r)\|v\|
  \end{array}\right.\]

\end{defi}

Hence for a zero dimensional lamination $K$, the normal hyperbolicity of $K$ means that $K$ is hyperbolic. Moreover the persistence of $K$ means that $K$ is structurally stable. But there exist structurally stable diffeomorphisms, that is diffeomorphisms such that
their $C^1$-perturbation are $C^0$-conjugated to them, which are not Anosov. Thus, the Hirsch-Pugh-Shub so called theorem is not optimal. Fortunately, the identification of $C^1$-structurally stable diffeomorphism is done, and lead up to the following definition:

\begin{defi} A \emph{diffeomorphism satisfies axiom $A$  and the strong transversality condition (SA)} if:
\begin{itemize}
\item  the nonwandering set $\Omega$ is hyperbolic,
\item  the periodic points are dense in $\Omega$,
\item the stable and unstable manifolds of points of $\Omega$ intersects their self transversally.\end{itemize}\end{defi}
 
The work of Smale \cite{Sm}, Palis  \cite{PS}, de Melo  \cite{dM}, Ma\~ne \cite{Mane}, Robbin  \cite{Ri} and Robinson  \cite{Rs} have concluded to the following theorem:
\begin{theo} The diffeomorphisms $C^1$-structurally stable of a compact manifold are exactly the SA diffeomorphisms.\end{theo}

In order to generalize the above two theorems in only one conjecture, let us introduce a last definition:

\begin{defi} Let $(L,\mathcal L)$ be a compact lamination, preserved by a diffeomorphism $f$ of a manifold $M$. We denote by $\Omega(\mathcal L)$ the smallest $\mathcal L$-saturated compact subset, which contains the nonwandering set of $f_{|L}$. The lamination $\mathcal L$ is \emph{$r$-normally SA} if:
\begin{itemize}
 \item there exist $\epsilon>0$ and a neighborhood $U$ of $\Omega(\mathcal L)$, such that every $\epsilon$-pseudo-orbit of $U$ which respects $\mathcal L$ is included in $\Omega(\mathcal L)$,
 \item the lamination $\Omega(\mathcal L)$ is $r$-normally hyperbolic and plaque-expansive,
\item the stable set of a leaf of $\Omega(\mathcal L)$ (which is an immersed manifold) intersect transversally the unstable set of every leaf of $\Omega(\mathcal L)$.
\end{itemize}\end{defi}

This is our conjecture:
\begin{conj}\label{3}
Compact $r$-normally SA laminations are $C^r$-persistent.
\end{conj}
\begin{exem} Let $f$ be a SA diffeomorphism of a manifold $M$. Let $N$ be a compact manifold. Let $\mathcal L$ be the lamination on $M\times N$ whose leaves are written in the form $\{m\}\times N$, for $m\in M$. Let $F$ be the dynamics on $M\times N$ equal to the product of $f$ with the identity of $N$. Then the lamination $\mathcal L$ is normally SA. This conjecture would imply that this lamination is persistent.\end{exem}

Kipping in mind the example \ref{Strat:AxiomA}, we can hope to show the conjecture \ref{3}, by using our main persistence theorem \ref{th2}. With this tool, we have proved in \cite{PB} the following theorem:
\begin{theo}\label{normallysa}
A compact $1$-normally SA lamination, whose leaves are the connected components of a $C^1$-bundle over a surface is $C^1$-persistent.\end{theo}

\begin{rema} This theorem provides non-trivial lamination which are persistent but not normally hyperbolic.\end{rema}
\begin{rema} The proof of this theorem can be found in \cite{PB}. We hope to publish soon a proof where the base of the bundle is not necessarily a surface.
\end{rema}
The above theorem can be written in the following equivalent form:
\begin{theo} Let $s$ be a $C^1$-submersion of a compact manifold $M$ onto a compact surface $S$. Let $\mathcal L$ be the lamination structure on $M$ whose leaves are the connected components of the fibers of $s$.

Let $f$ be a diffeomorphism of $M$ which preserves the lamination $\mathcal L$. Let $f_b\in Diff^1(S)$ be the dynamics induced by $f$ on the leaves spaces of $\mathcal L$.
We suppose that:
\begin{itemize}
\item  $f_b$ satisfies axiom $A$ and the strong transversality condition,
\item the $\mathcal L$-saturated subset generated by the nonwandering set of $f$ in $M$ is $1$-normally hyperbolic.\end{itemize}
Then $\mathcal L$  is $C^1$-persistent.\end{theo}

\newpage\section{Proof of the persistence of stratifications}

\subsection {Preliminary}\label{prelim}
\subsubsection{Statements and notations}

Along all this chapter, we work -- at least -- under the hypotheses of theorem \ref{th2} with the following notations:
\begin{itemize}
\item we denote by $n$ the dimension of $M$,
\item we denote by $\mathcal V=(V_X)_{X\in \Sigma}$ the neighborhood and the family of neighborhoods adapted to  $f^*$,
\item we denote by $\Sigma':=\{X_1,\dots,X_N\}$ the set of strata of $\Sigma$ which intersect $cl(A')$, indexed such that, for any integers $i\le  j$ of $\{1,\dots,N\}$, if $X_i$ and  $X_j$ are comparable then $X_i\le X_j$,
\item we denote by $d_j$ the dimension of $X_j$. To make lighter the notations, we denote by $(L_j,\mathcal L_j)$ the tubular neighborhood $(L_{X_j},\mathcal L_{X_j})$ and by $V_j$ the neighborhood $V_{X_j}$.
\item Given a compact subset $C$, we denote by $V_C$ an open and precompact neighborhood of $C$, and  
by $(V_C',\hat V_C)$ a pair of open subsets that satisfies 
\[C\subset V_C'\subset cl(V_C')\subset V_C\subset cl(V_C)\subset \hat V_C\]
\end{itemize}

We recall that if $L_j$ intersects $L_k$, then $X_j$ and $X_k$ are comparable and, if $j\le k$, then we have $d_j\le d_k$.

\subsubsection{Construction of vector bundles}

As in theorem \ref{th2}, the trellis structure occurs in a ``germinal'' way, we can restrict each tubular neighborhood of the structure (and so the adapted neighborhood associated) without loss of generality.
Therefore, we allow ourself to restrict the laminations $(L_k,\mathcal L_k)_k$, to open neighborhoods of strata $(X_k)_k$.

For $d\in \{ 0,\dots,n \}$, let $Gr(d,TM)$ be the Grassmannian bundle over $M$ of $d$-plans of $TM$. The Riemannian metric $g$ over $M$ induces -- in a standard way -- a Riemannian metric on this bundle.

Let $j\in \{1,\dots,N\}$ and let
\[N_j': L_j \rightarrow Gr(n-d_j,TM)\]
 \[x\mapsto \big(T_x i(T_x\mathcal{L}_j)\big)^\bot.\]

This map is a continuous lifting of $i_{|L_j}$:
\[\begin{array}{rcl}
 & N_j' & Gr(n-d_j,TM)\\
&\nearrow &\downarrow \\
L_j &\rightarrow  &M\\
&i_{|L_j}&\\
\end{array}\]

The family $(N_j')_j$ satisfies two properties that we like:
\begin{itemize}
\item for all $j\ge k$ and $x\in L_k\cap L_j$, the space $N_j'(x)$ is included in $N_k'(x)$,
\item for all $j\in \{1,\dots ,N\}$ and $x\in L_j$, $N_j'(x)\oplus Ti(T_x\mathcal{L}_j)=T_{i(x)}M$.
\end{itemize}

Nevertheless, these liftings are -- in general -- only continuous. And we would like to use these liftings in order to control the map $f'\mapsto i(f')$ in the following way:
for an open covering $(V_{C_k})_k$ of $A'$, for all $x\in U_k$ and $f'\in V_f$, the point $\exp^{-1}_{i(x)} \circ i(f')(x)$ is well defined and belongs to $N_k'(x)$, with $\exp$ the exponential map associated to the Riemannian metric on $M$.

As we want $i(f')$ to be $C^r$-$\mathcal T$-controlled, we have to change $(N_k')_k$ to a family of $\mathcal T$-controlled $C^r$-liftings $(N_k)_k$ which still satisfy the above properties, by using the following lemma:

\begin{lemm}\label{Ni}
If we reduce the tubular neighborhoods, there exists a family of maps $(N_k)_k$, where $N_k\in Mor^r (\mathcal T_{|L_k},Gr(n-d_k ,TM))$ is a lifting of $i_{|L_k}$, which can be chosen arbitrarily $C^0$-close to $N_k'$ and satisfies:
\begin{enumerate}
\item for all $j\ge k$ and $x\in L_k\cap L_j$, the space $N_j(x)$ is included in $N_k(x)$,
\item for every $x\in L_k$, $N_k(x)\oplus Ti(T_x\mathcal{L}_k)=T_{i(x)}M$.
\end{enumerate}
\end{lemm}
\begin{proof}
Let us construct $(N_k)_k$ by induction on $k$.

\emph{ Step $k=1$}

This step is obtained by applying the annex \ref{relstra} to the continuous lifting $N_1'$, of the $\mathcal T_{|L_1}$-controlled immersion $i_{|L_1}$, to obtain the lifting $N_1\in Mor^r(\mathcal T_{|L_1},Gr(n-d_1,TM))$ of $i_{|L_1}$ arbitrarily $C^0$-close to $N_1'$.

\emph{ Step $k\rightarrow k+1$}

Following the annex \ref{relstra}, there exists a $\mathcal T_{|L_{k+1}}$-controlled lifting $N_{k+1}^{k+1}$ of $i_{|L_{k+1}}$ in $Gr(n-d_{k+1},TM)$, arbitrarily $C^0$-close to $N_{k+1}'$. Thus, the lifting $N_{k+1}^{k+1}$ satisfies the condition $2$.

Let us change this lifting, in order that it satisfies the condition $1$.

For $j\le k$, let $p_{N_j}$ be the orthogonal projection of $i^*TM_{|L_j}$ onto the vector bundle induced by $N_j$.

Let us define the homotopy $p_j$ from the identity of $i^*TM$ to $p_{N_j}$:

 \[p_j\;:\; [0,1]\times i^*TM_{|L_j}\rightarrow i^*TM\]
 \[(t,u)\mapsto u+t\cdot (p_{N_j}(u)-u)\]
As $(L_j,X_j^c)$ is an open covering of $A\setminus \cup_{X_p<X_j} X_p$, it follows from the annex \ref{partstra}, that there exists a function $\rho_j\in Mor(\mathcal T_{p|A\setminus \cup_{X_p<X_j} X_p}, [0,1])$, whose restriction to an open neighborhood $L_j'$ of $X_j$ is equal to $1$ and whose support is included in $L_j$.

We now define by decreasing induction on $j\in\{1,\dots,k\}$, the map 
 \[N_{k+1}^{j} \::\; L_{k+1} \longrightarrow Gr(n-d_{k+1},TM) \]
\[\quad\quad\quad x \mapsto \left\{
          \begin{array}{ll}p_j \big(\{\rho_j(x)\},N_{k+1}^{j+1}(x)\big)
  & \qquad \mathrm{if}\quad x\in L_j \\
  N_{k+1}^{j+1}(x)
 & \qquad  \mathrm{else} \\
\end{array}
        \right. \]
which is well defined and can be constructed arbitrarily close to $N_{k+1}'$, by supposing the lifting $N_{j}$ close enough to $N_j'$ and the lifting $N_{k+1}^{j+1}$ close enough to $N_{k+1}'$. Moreover, we show by induction that $N_{k+1}^j$ is $C^r$-$\mathcal T_{|L_{k+1}}$-controlled, by remarking that $\cup_{X_p<X_j}X_p$ does not intersect $L_{k+1}$.

Let us suppose now the inductive construction achieved. The lifting $N_{k+1}:=N_{k+1}^0$ satisfies therefore condition $2$, if it is chosen close enough to $N_{k+1}'$. For $j\le k$, we restrict $(L_j, \mathcal L_j)$ to $L_j'$. The condition $1$ is then satisfied.\end{proof}

 For $k\in\{0,\dots ,N\}$, we denote by $\pi\;:\; F_k\rightarrow L_k$ the vector bundle induced by $N_k$.
We endow this bundle with the norm induced by the Riemannian metric on $M$.
\subsubsection{Construction of an adapted filtration}
 \emph{We denote by $K$ the compact $cl(A')$.}

\begin{propr}\label{Li'}
There exists a family of compact sets $(K_p)_{p=1}^{N+1}$ which satisfies:

 \begin{itemize}
\item [\ref{Li'}.1] $K=K_1\supset  K_2\supset \cdots \supset K_{N+1}=\emptyset  \; \mathrm{and} \: f^*(K_p)\subset int(K_p),\; \forall p\ge 0$,

\begin{center}
such that, for all $p\le N$, with $C_p:= cl(K_{p}\setminus K_{p+1})$, we have:
\end{center}

\item[\ref{Li'}.2] the compact set $C_p$ is included in the adapted neighborhood $V_p$.

\item [\ref{Li'}.3] 
For any $x\in C_p$, any $u\in Ti(T_x\mathcal L_p)^\bot$, the orthogonal projection of $T_xf(u)$ onto the subspace $Ti(T_{f^*(x)}\mathcal L_p)^\bot$ is nonzero.
\end{itemize}
\end{propr}

\begin{rema} For all $p\in\{1,\dots,N\}$, the compact set $K_p$ is equal to $\cup_{j=p}^N C_j$. It follows from \ref{Li'}.2 that the compact set $K_p$ is included in $\cup_{j\ge p}X_j$.\end{rema}
\begin{figure}[h]
    \centering
\includegraphics[width=7cm]{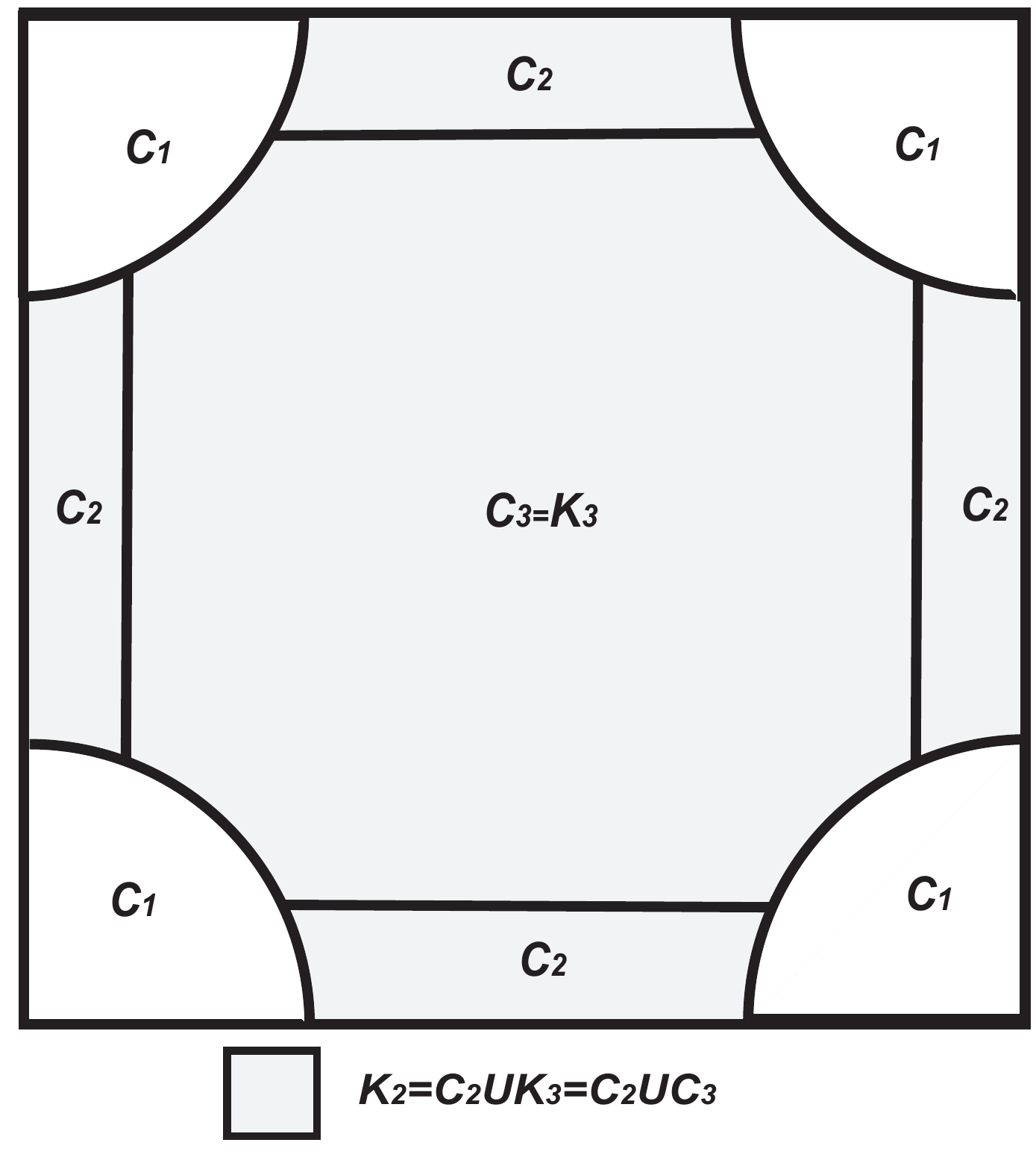}
    \caption{Compact sets $(C_k)_k$ for the simplicial stratification of a square, endowed with the trellis structure drawn figure \ref{coupe treillis}.}
    \label{fig:schemack}
\end{figure}

\begin{proof}
We will prove the existence of an open neighborhood $S_p$ of  $\cup_{j\le p} X_j\cap K$ in the topology induced by $K$, that satisfies 
\begin{equation}\label{0}\emptyset =S_0\subset S_1\subset \cdots \subset S_{N}=K\quad \mathrm{and}\quad {f^{*^{-1}}_{|K}}(cl(S_p))\subset S_p\end{equation}
such that $cl(S_{p}\setminus S_{p-1})$ can be chosen arbitrarily close to $ X_{p+1}\cap K\setminus S_{p-1}$ (the open subset $S_{p-1}$ being  fixed) and satisfies 
\begin{equation}\label{fond}
\bigcap_{n\in\mathbb N} f_{|K}^{*^{-n}}(S_{p}) =    \bigcup_{j\le p} X_j\cap K
\end{equation}

We define then $K_p:= K\setminus S_{p-1}$, for $p\ge 1$. Let us show that (\ref{0}) and (\ref{fond}) are sufficient to prove this property:

\noindent{\it Proof of \ref{Li'}.1}

The first part of \ref{Li'}.1 is obvious by the first part of (\ref{0}).

The compact set $K$ is sent by  $f^*$ into the interior of $K$ and, following the second inclusion of (\ref{0}),  we have
\[ K_p=K\setminus S_{p-1} \subset f_{|K}^{*^{-1}}\big(int(K)\big)\setminus f^{*^{-1}}_{|K}\big(cl(S_p)\big)=f_{|K}^{*^{-1}}\big(int(K)\setminus cl(S_p)\big)\]
\[\Rightarrow f^*(K_p)\subset int(K)\setminus cl(S_p)= int(K_p)\]

\noindent{\it Proof of \ref{Li'}.2}

The compact set $C_p$ is equal to $cl(S_p\setminus S_{p-1})$, which can be chosen arbitrarily close to the compact set $X_{p}\cap K_p=X_p\cap K\setminus S_{p-1}$, which is included in $V_{p}$.

\noindent{\it Proof of \ref{Li'}.3}

By normal expansion, for any $x\in K \cap X_p\setminus S_{p-1}$, any $u\in Ti(T_x\mathcal L_p)$, the orthogonal projection of $T_xf(u)$ onto $Ti(T_x\mathcal  L_p)^\bot$
 is nonzero.
 
 By compactness, statement \ref{Li'}.3 is then  satisfied if $C_p$ is enough close to $X_{p}\cap K_p$.   

\noindent{\it Proof of $(1)$ and $(2)$}

Let us construct, by induction on $p\in\{0,\dots,N\}$, the subset $S_p$ which satisfies $(1)$ and $(2)$. For the rest of the proof, we deal with the topology induced by $K$.

Let $p$ be an integer that satisfies the induction hypothesis. Let $U:=(K\cap V_{p+1})\cup S_p$. Following (\ref{fond}), every orbit which starts in $U\cap K$ without $\tilde K:= K\cap(\cup_{j\le p+1} X_j)$ leaves definitively $S_p$ and, by hypothesis $(iii)$ of theorem \ref{th2}, leaves also $V_{p+1}$. As the set $f_{|K}^{*^{-1}}(\tilde K)$ is equal to $\tilde K$, we have
\[ \bigcap_{n\ge 1}f_{|K}^{*^{-n}}(U)=\tilde K.\]

Let $V_0$ be a compact neighborhood of $\tilde K$ in $U$. We have also 
\[\bigcap_{n\ge 1}f_{|K}^{*^{-n}}(V_0)=\tilde K.\]
By compactness, there exists $M\ge 0$ such that $\bigcap_{n= 1}^Mf_{|K}^{*^{-n}}(V_0)$ is included in $V_0$. We now define 
\[V_1:=  \bigcap_{n= 0}^{M}f_{|K}^{*^{-n}}(V_0).\]
The compact set $V_1$ has its preimage by $f^*_{|K}$ which is included into itself. The decreasing sequence of preimages of $V_1$ converges to $\tilde K$. Moreover, $V_1$ is a neighborhood of $\tilde K$ (for the topology induced by $K$). We would like the preimage $f^{*^{-1}}_{|K}(V_1)$ to be included into the interior of $V_1$. This require the construction of a new neighborhood.

There exists $M'>0$ such that $f^{*^{-M'}}(V_1)$ is included in the interior of $V_1$. We chose a family of open subsets $(V^i)_{i=0}^{M'-1}$ that satisfy

 \[int\big(f^{*^{-M'}}_{|K}(V_1)\big)=:V^0\subset cl(V^0)\subset V^1\subset cl(V^1)\subset V^2\subset \cdots \subset V^{M'-1}:=int(V_1).\]

Let us define the following open neighborhood of $\tilde K$ in $V_0$:
 \[V_2:=\bigcup_{n=0}^{M'-1} f^{*^{-n}}_{|K}(V^n)\]

 We easily check that the preimage by $f^{*}_{|K}$ of $cl(V_2)$ is included in $V_2$, and that 
 \[\bigcap_{n\ge 0}f^{*^{-n}}_{|K}(V_2)=\tilde K\]
Let us define $S_{p+1}:=V_2\cup S_p$, which is a neighborhood of $\tilde K$ and satisfies (\ref{0}). Moreover, $f_{|K}^{*^{-k}}(S_{p+1})$ is equal to $f_{|K}^{*^{-k}}(V_2)\cup f^{*^{-k}}_{|K}(S_p)$, thus $\cap_{n\ge 0}f_{|K}^{*^{-n}}(S_{p+1})$ is equal to $K\cap(\cup_{l\le p+1} X_l)$, which is (\ref{fond}).
The hypothesis is satisfied with $cl(S_{p+1}\setminus S_p)$ arbitrarily close to $K\cap X_{p+1}\setminus S_p$, by replacing $S_{p+1}$ by $f_{|K}^{*^{-n}}(S_{p+1})\cup S_{p}$.
\end{proof}

\subsubsection{Uniformity of exiting chains}
Let $(L,\mathcal{L})$ be a lamination, $V$ be a subset of $L$ and $f^*$ be a continuous map from $V$ to $L$. A \emph{ $\epsilon $-pseudo-chain of $V$ which respects $\mathcal L$} is a sequence
$(p_n )_{ n=0}^N\in {V}^{N+1}$ such that, for all $n\in\{0,\dots ,N-1\}$, the points $p_{n+1}$ and $f^*(p_n)$ are in  a same plaque of $\mathcal{L}$ of diameter less than $\epsilon$. We say that $(p_n )_{ n=0}^N\in L^{N+1}$ {\it starts} from $p_0$, {\it arrives} to $p_N$, and is of \emph{length} $N$.

\begin{propr}\label{chaine}
Let $p\in \{1,\dots,N\}$ and let $\eta$ be the function on $V_p$ associated to $X_p$ in hypothesis $(iii)$ of  theorem \ref{th2}. For every open subset $V$, precompact in $V_p$, and every real $\eta'\in]0,\inf_{V}\eta[$, we have 
\[\cup_{j\ge 0} int(U_j)=V\setminus X_p\]
with $U_j$ the subset of points $x\in V_p$ such that there is no any $\eta'$-pseudo-chain of $V$, which respects $\mathcal L_p$, starts from $x$ and of length $j$.
\end{propr}

\begin{proof}

To show this property, it is sufficient to prove that, for every $x\in V\setminus X_p$, there exists $j\ge 0$ such that $x$ belongs to the interior of $U_j$. Let $W$ be a compact neighborhood of $x$ included in $V\setminus X_p$. Let $W_n$ be the subset $V$ consisting of the arriving points of $\eta'$-pseudo-orbits of $V$, of length $n$, starting from  $x'\in W$, and which respects $\mathcal L_p$.

If for $n$ large enough, the subset $W_n$ is empty, then $x$ belongs to the interior of $U_n$.

Else, we show a contradiction: there exists then a family $\big((x_i^k)_{i=0}^{N_k}\big)_k$ of $\eta'$-pseudo-chains of $V$ which respect $\mathcal L_p$, start from $W$, and such that $(N_k)_k$ converges to the infinity.
We complete $(x_i^k)_{i=0}^{N_k}$ to a family $(x_i^k)_{i\in \mathbb N}\in V^\mathbb N$ with $x_i^k:= x$ for every $i>N_k$.
As $V$ is precompact in $V_p$, by diagonal extraction, a subsequence converges to an $\eta$-pseudo-orbit of $V_p$ which respects $\mathcal L_p$ and starts at $x'\in W$. As $x'$ belongs to $V_p\setminus X_p$, this $\eta$-pseudo-orbit is included in $V_p\setminus X_p$. This contradicts hypothesis $(iii)$ of theorem \ref{th2}.\end{proof}

\subsection{Proof to corollary $\ref{cor2}$}\label{demcor}

Let $V_{A'}$ be a neighborhood of $cl(A')$, such that $f^*(cl(V_{A'}))$ is included into $A'$. By applying theorem \ref{th2} with $V_{A'}$ instead of $A'$, we may suppose that $f'\mapsto f'^*$ is a continuous map from $V_f$ to $End_{f^*\mathcal V'}^r(\mathcal T_{|V_A'})$. Nevertheless, we continue to use all the notations and statements of the preceding sections.

Let $(K_p)_p$ be the compact sets family provided by property \ref{Li'}. Let us show, by decreasing induction on $p$, that by taking $V_f$ sufficiently small, we may suppose that $i(f')_{|K_p}$ is injective, for every $f'\in V_f$. Therefore, the restriction $i(f')_{|K}$ is a homeomorphism onto its image and $i(f')_{|A'}$ is a $\mathcal T_{|A'}$-controlled embedding.

For $p=N$, we first remark that $K_N$ is a compact subset of $X_N$. Let $\epsilon>0$ less than the minimum on $K_N$ of the plaque-expansiveness function of $X_N$. By taking $\epsilon$ sufficiently small, we may suppose that restricted, to any $X_N$-plaque intersecting $K_N$ and with diameter less than $\epsilon$, the map $i(f')$ is injective, for every $f'\in V_f$.
  
We notice that the following map is continuous:
    \[\phi\;:\;  V_f\longrightarrow \mathbb R^+\]
    \[f'\mapsto \min_{(z,z')\in K_N^2,\; d(z,z')\geq \epsilon} d\big( i(f')(z),i(f')(z')\big)\]
As $\phi(f)$ is positive, by taking $V_f$ sufficiently small, we may suppose that $\phi$ is positive on $V_f$.
Let $(x,y)\in K_N^2$ and $f'\in V_f$ such that $i(f')(x)=i(f')(y)$. By commutativity of the diagram, this implies that,
for every $n\ge 0$, the points $i(f')(f'^{*^n}(x))$ and $i(f')(f'^{*^n}(y))$ are equal. It follows from \ref{Li'}.1 and from the continuity of the extension of $f'\mapsto i(f')$ that, by taking $V_f$ sufficiently small, the points $f'^{*^n}(y)$ and $f'^{*^n}(y)$ belong to $K_N$, for every $n\ge 0$. As $\phi(f')$ is positive, this implies that  the points $f'^{*^n}(y)$ and $f'^{*^n}(y)$ are $\epsilon$-distant. By taking $V_f$ sufficiently small, $(f'^{*^n}(x))_n$ and $(f'^{*^n}(y))_n$ are $\epsilon$-pseudo orbits which respect the lamination $X_N$. By plaque-expansiveness and injectivity of the restriction of $i(f')$ to the $\epsilon$-$X_N$-plaques, $x$ and $y$ are equal. This implies that the restriction of $i(f')$ to $K_N$ is injective, for every $f'\in V_f$.

 We suppose the injectivity to be shown on $K_{p+1}$. By proceeding as in the step $p=N$, one shows that the restriction of $i(f')$ to the compact set $K_p\cap X_p$ is injective, for every $f'\in V_f$.

 Let $(x,y)\in K_p^2$ and $f'\in V_f$ that satisfy 
\[i(f')(x)=i(f')(y).\]

By taking $V_f$ sufficiently small, we may suppose that, for every $f''\in V_f$, the compact sets $i(f'')(K_p\cap X_p)$ and $i(f'')(K_{p+1})$ are disjoint, and $f''^*(K_p)$ is contained in $K_p$.

If $x$ belongs to $ X_p$, by commutativity of the diagram, we have 

\[\forall n\geq 0,\; i(f')\circ f^{'*^n}(x)=i(f')\circ f^{'*^n}(y)\Rightarrow \forall n\ge 0,\; f^{'*^n}(y)\in K_p\setminus K_{p+1}\subset C_p\]

By taking $V_f$ sufficiently small and by the compactness of $C_p$ in $V_p$, we may suppose that $(f'^{*^n}(y))_n$ is an  $\eta$-pseudo-orbit which respects $\mathcal L_p$, with $\eta$ the function on $V_p$ provided by hypothesis $(iii)$ of theorem \ref{th2}. Therefore, following this hypothesis $(iii)$, $y$ belongs to $X_p\cap C_p$. Thus, $x$ and $y$ are equal.

Let us treat the case where nether $x$ nether $y$ belongs to  $X_p$. We fix a compact neighborhood $V_{C_p}$ of $C_p$ in $V_p'$, and we note that:
\begin{enumerate}
    \item By taking $V_f$ sufficiently small, it follows from the local inversion theorem and from the compactness of $V_{ C_p}$ that there exists $\epsilon>0$, which does not depend on $f'\in V_f$, such that the restriction of $i(f')$ to any plaque of $\mathcal L_p$, with diameter less than $\epsilon$ and nonempty intersection with $V_{ C_p}$, is an embedding.

    \item By taking $\epsilon$ and then $V_f$ sufficiently small, it follows from \ref{Li'}.3 that there exists $\epsilon>0$ such that for any pair $(x',y')\in V_{ C_p}^2$ satisfying $f'^{*}(x')=f'^{*}(y')$ and $d(x',y')<\epsilon$, the points $x$ and $y$ belong to a same plaque of $\mathcal L_p$ whose diameter is less than $\epsilon$. We can suppose moreover that the open set $V_{ C_p}$ contains the $\epsilon $-neighborhood of $C_p$.
    \item We notice that the following map is continuous:
    \[\phi\;:\;  V_f\longrightarrow \mathbb R^+\]
    \[f'\mapsto \min_{(z,z')\in K_p^2,\; d(z,z')\geq \epsilon} d\big( i(f')(z),i(f')(z')\big)\]
As $\phi(f)$ is positive, by taking $V_f$ sufficiently small, for every $ f'\in V_f$, the real number $\phi(f')$ is also positive.
\end{enumerate}
Since $i(f')(x)$ is equal to $i(f')(y)$, by commutativity of the diagram, for every  $ n\geq 0$, the point $i(f')\circ f'^{*^n}(x)$ is equal to $i(f')\circ f'^{*^n}(y)$. It follows from 3) that, for $ n\geq 0$, we have $d(f'^{*^n}(x),f'^{*^n}(y))< \epsilon$.

By tacking $V_f$ sufficiently small, it follows from hypothesis $(iii)$ and \ref{Li'}.2, as  neither $x$ nor $y$ belongs to $X_p$, that there exists a minimal integer $M$ such that $f'^{*^M}(x)$ and $f'^{*^M}(y)$ belong to $K_{p+1}$. Using the induction hypothesis, the point  $f'^{*^M}(x)$ is equal to $f'^{*^M}(y)$. Moreover, by definition of $M$, the points $f'^{*^{M-1}}(x)$ and $f'^{*^{M-1}}(y)$ belong to the $\epsilon$-neighborhood of $C_p$ and so to  $V_{ C_p}$. Using 2) then 1), we have 
\[f'^{*^{M-1}}(x)=f'^{*^{M-1}}(y)\]
By decreasing induction, using 3), then 2), and finally 1), we have 
\[\forall n\leq N,\;  f'^{*^n}(x)=f'^{*^n}(y)\]
Thus, $x$ and $y$ are equal.
\begin{flushright}
$\square$
\end{flushright}

\subsection{Proof of main theorem $\ref{th2}$}

\subsubsection{Fundamental property of dynamics on $K_p$}

We denote by $\exp$ the exponential map associated to a complete Riemannian metric on $M$.

Let $\epsilon\in C^\infty(M,\mathbb R^+_*)$ be a positive function less than the injectivity radius of the exponential map.
\[\mathrm{Let}\;Exp\;:\; i^*TM\rightarrow M\]
\[(x,v)\mapsto \exp_{i(x)}\left(\epsilon\circ i(x)\cdot \frac{v}{\sqrt{1+\|v\|^2}}\right)\]

For all $p\in \{1,\dots, N\}$ and $x\in  L_p$, let $\mathcal F_{px}$ be the submanifold $Exp(F_{px})$ and let $\mathcal F_{px}^{\eta'}$ be the submanifold $Exp(B_{F_{px}}(0,\eta'))$.

\label{prop fonda}

We will prove by decreasing induction on $p$ the
\begin{proprf} For every $p\in\{1,\dots,N+1\}$, there exist:
\begin{itemize}
\item a real $\eta'>0$ and a neighborhood $V_f$ of $f\in End^r(M)$, both arbitrarily small,
\item an open neighborhood $A_p$ of $K_p$, which is precompact in $\cup_{q\ge p} X_q$ and whose closure is sent by $f^*$ into $int(K_p)$,
\item a family of neighborhoods $\mathcal V^p:=(V_X)_{X\in \Sigma_{|A_p}}$ adapted to $f^*_{|A_p}$,
\item a continuous map 
\[V_f\rightarrow End_{f^*_{|A_p} \mathcal V^p}^r (\mathcal T_{|A_p})\times Mor^r(\mathcal T,M)\]
\[f'\longmapsto (f_p'^*,i_p(f'))\]
\end{itemize}
that satisfy:\begin{enumerate}
\item $f_p^*=f^*_{|A_p}$ and $i_p(f)=i$,
\item \label{hatf}
the diagram $\begin{array}{rcccl}
&&f'&&\\
&M&\rightarrow&M&\\
i_p(f')&\uparrow &&\uparrow &i_p(f')\\
&A_p&\rightarrow&A_p&\\
&&f_p^{'^*}&&\\
\end{array}$ commutes,
\item \label{Im} the restriction of $i_p(f')$ to $A_p$ is an immersion,

\item  \label{N} there exists a neighborhood of $f^*(C_k)$, which does not depend on $f'\in V_f$, whose points $x$ are sent by $i_p(f')$ into $\mathcal F_{kx}$.\\
\begin{center} For every $j\ge p$, let $X_j^p$ be the stratum of $\Sigma_{|A_p}$ associated to $X_j\in\Sigma'$.\end{center}

\item \label{cp} For every $j\ge p$, the subset $V_{X_j^p}$ is a neighborhood of $C_j$, and for all $x\in V_{X_j^p}$ and $f'\in V_f$, the point $f'^*_p(x)$ belongs  to the set \footnote{We recall that $\mathcal L_{jy}^\delta$ denote the union of the plaques of $\mathcal L_j$ containing $y\in L_j$ and of diameter less than $\delta>0$.} $\mathcal L_{jf^*(x)}^{\eta'}$.

\end{enumerate}
\end{proprf}

The link between $\eta'$ and $V_f$ is the following: $\eta'$ will be chose small enough, then  $V_f$ will be chosen small enough regarding to $\eta'$.

\subsubsection{Proof of main theorem \ref{th2} by admitting the fundamental property}\label{zzzzz}
Let us show that, for $p=1$, the above property is sufficient to show theorem \ref{th2}. For every $j\ge 1$, we denote by $X_j'$ the stratum of $\Sigma_{|A'}$ associated to $X_j$.

As $cl(A')$ is a compact subset included in $A_1$ sent by $f^*$ into $A'$, by taking $\eta'$ and  $V_f$ sufficiently small, we may suppose that, for all $f'\in V_f$, we have \[d(A'^c,f'^{*}_1(A'))>\eta'.\]

Therefore, it follows from fundamental property \ref{Im}, that we can define the following continuous map:
\[V_f\rightarrow End_{f^*_{|A'}\mathcal V'}^r(\mathcal T_{|A'})\times Im^r(\mathcal T_{|A'},M)\]
\[f'\mapsto \big(f'^*:=f'^*_{1|A'},i(f'):=i_1(f')_{|A'}\big)\]
\[\mathrm{with\; for}\; j\ge 1,\; V_{X_j'}:= V_{X_j^1}\cap A'\;\mathrm{and}\; \mathcal V':= (V_{X})_{X\in \Sigma_{|A'}}.\]

Conclusion $(i)$ of theorem \ref{th2} is a simple consequence of fundamental property  \ref{hatf}.

Conclusion $(iii)$ of theorem \ref{th2} for the stratum $X_p'$ can be shown by induction on $p\ge 1$.

For every $p\ge 1$, by restricting $V_{X_p^1}$, we may suppose that $V_{X_p^1}\cap K_p$ is precompact in $V_p$. We may also suppose that $\eta'$ is less than the minimum on $V_{X_p^1}\cap K_p$ of the function $\eta$ associated to $X_p$ in hypothesis $(iii)$.

The step $p=1$ is then obvious. We now consider $p>1$.

As $K_p$ is sent by $f^*$ into its interior, by taking $V_f$ and $\eta'$ sufficiently small, every $\eta'$-$f'^*$-pseudo-orbit of $V_{X_p'}$, which respects $\mathcal L_p$, and starts in $V_{X_p'}\cap K_p$ is included in $K_p$, and by hypothesis $(iii)$, is necessarily included in $X_p$.

In the other hand, by taking $V_f$ sufficiently small, it follows from  \ref{Li'}.1 that we have 
\begin{equation}\label{aaaa}A'\cap K_p\subset int\big(f'^{*^{-1}}(A'\cap K_p)\big)\end{equation}
and we can show that
\begin{equation}\label{aaaaa}\cup_{n\ge 0}f'^{*^{-n}}(A' \cap K_p)\supset V_{X_p'}.\end{equation}

Because, else there exists $x\in V_{X_p'}$ having its $f'^*$-orbit which does not intersect $K_p$. Let $q<p$ be maximal such that the orbit of $x$ intersects $C_q$. The point $x$ cannot belong to  $X_q$ ; so its orbit leaves necessarily $C_q$, by the hypothesis of induction, fundamental property \ref{cp} and the choice of $\eta'$. On the other hand, the orbit of $x$ intersects $K_q$, so eventually lands in $K_q$ but does not intersect $K_p$. Therefore its orbit intersects $C_{q'}$ with $p>q'>q$. This
is a contradiction with the maximality of $q$.

By (\ref{aaaa}) and (\ref{aaaaa}), we can build a continuous and positive function $\eta''$ on $V_{X'_p}$, which is less than  $\eta'$ and such that: for all $n\ge 1$, $x\in f'^{*^{-n}}(A'\cap K_p)\cap V_{X_p'}$, and $x_1\in V_{X_p'}\cap \mathcal L_{pf'^*(x)}^{\eta''(x)}$, the set $V_{X'_p}\cap \mathcal L_{pf'^*(x_1)}^{\eta''(x_1)}$ is included in $f'^{*^{-n+1}}(A' \cap K_p)\cap V_{X'_p}$. Such a function $\eta''$ satisfies conclusion $(iii)$ of theorem \ref{th2}, because all
$\eta''$-pseudo-orbits of $f'^*$ in $V_{X_p'}$ belong eventually to $V_{X_p'}\cap K_p$.

Only the proof of conclusion $(ii)$ remains. By fundamental properties \ref{hatf}, \ref{Im}, and \ref{cp}, the lamination $X_p'$, is immersed by $i(f')$ and preserved by $f'\in V_f$, for every $p\ge 1$.
By continuity of $f'\mapsto (i_1(f'),f'^*)$ and the $r$-normal expansion, we may restrict $V_f$ such that the endomorphism $f'$ uniformly $r$-normally expands the lamination $X^1_{p}$ over the compact set $X_p\cap K_p$. It follows from  (\ref{aaaa}) and (\ref{aaaaa}), that the endomorphism $f'$ $r$-normally expands the immersed lamination $ X_p'$. Thus, conclusion $(ii)$ of theorem \ref{th2} is satisfied.

\subsubsection{Step p=N+1}

By taking $A_{N+1}:=\emptyset$, $V_f:=End^r(M)$, and $i(f'):=i$ for every $f'\in V_f$, this step is obviously satisfied.

\subsubsection{Step p+1 $\Rightarrow$ step p}
Naively, the idea of the proof is, for every $f'\in V_f$, to glue $i_{p+1}(f')$ with its pull back given by the following lemma:
\begin{lemm}\label{prelem6}
By restricting $\eta'$ and then $V_f$, there exist an open precompact neighborhood  $\hat V_{C_p}$ of $C_p$ in $V_p$ 
such that, for any small neighborhood $\hat {A_p}$ of $K_p$, there exist
a neighborhood $V_i$ of $i_{|\hat A_p}\in Mor^r(\mathcal T_{|\hat A_p},M)$ and a continuous map
\[S^0\;:\; V_f\times V_i\rightarrow Mor^r(\mathcal T_{|\hat V_{C_p}},M)\]
satisfying:
\begin{enumerate}
\item the morphism $S^0(f,i_{|\hat A_p})$ is equal to $i_{|\hat V_{C_p}}$,

for all $x\in L_p$ and $f'\in V_f$,
\item \label{preconc4}  the preimage of $i'(\mathcal L_{pf^*(x)}^{\eta'})$ by $f'$ intersects $\mathcal F_{px}^{\eta'}$ at a unique point $S^0(f',i')(x)$, for all $f'\in V_f$ and $i'\in V_i$.

 Let $f'^*_{i'}(x)\in \mathcal L_{pf^*(x)}^{\eta'}$ be defined by 
 \[f'\circ S^0(f',i')(x)=i'\big(f'^*_{i'}(x)\big).\]
\item\label{preim} 
If $i'$ is a controlled immersion around $f_{i'}'^*(x)$, then $S_0(f',i')(x)$ is a controlled immersion around $x$.
\end{enumerate}

\end{lemm}

\begin{proof}
A small neighborhood $\hat V_{C_p}$ of $C_p$ is precompact in $V_p$. We may suppose  $\eta'>0$ and the compact neighborhood $cl(\hat V_{C_p})$ small enough such that, by \ref{Li'}.3, the restriction of $f'$ to $\mathcal F_{px}^{\eta'}$ is a diffeomorphism onto its image, and this image intersects transversally at a unique point the image of the plaque $\mathcal L_{pf^*(x)}^{\eta'}$ by $i'$, for all $x\in cl(\hat V_{C_p})$, $i'$ $C^r$-close to $i$ and $f'$ $C^r$-close to $f$.

Writing this intersection point in the form 
\[f'(v)=i'(x'),\]
\[\mathrm{we\; define}\; \left\{\begin{array}{l} 
f'^*_{i'}(x):= x'\in \mathcal L_{pf^*(x)}^{\eta'}\\
S^0(f',i')(x):=v\in \mathcal F_{px}^{\eta'}\end{array}\right.\]
\begin{figure}[h]
    \centering
        \includegraphics{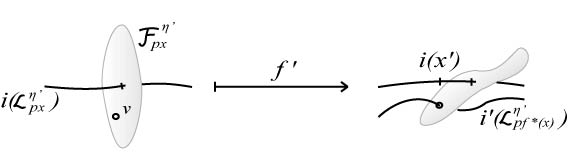}
        
    \caption{Definition of $S^0$.}
    \label{cexple}
\end{figure}

Such a map $S^0$ satisfies conclusions 1 and \ref{preconc4} of lemma \ref{prelem6}.
Let us show that $S^0$ takes continuously its values in the set of morphisms from the lamination $\mathcal L_{p|\hat V_{C_p}}$ into $M$.   

Let $x\in cl(C_p)$ and let $(U_y,\phi_y)\in \mathcal L_{p|\hat V_{C_p}}$ be a chart of a neighborhood of $y:=f^*(x)$.
We may suppose that $\phi_y$ can be written in the form 
\[\phi_y\;: \; U_y\rightarrow \mathbb R^{d_p}\times T_y\]
where $T_y$ is a locally compact metric space. Let $(u_y,t_y)$ be defined by
\[\phi_y(y)=:(u_y,t_y).\]
We remark that $\mathbb R^{d_p}$ is $C^r$-immersed by $\psi:= u\in \mathbb R^{d_p} \mapsto i\circ \phi^{-1}(u,t_y)$.

As we saw, the endomorphism $f$, restricted to a small neighborhood of $i(x)$, is transverse to the above immersed manifold, at $z:=f\circ i(x)$.
In other words,
\[Tf(T_{x}M)+T\Psi(T_{u_y}\mathbb R^{d_p})=T_zM.\]

Thus, by transversality, there exist open neighborhoods $V_{u_y}$ of $u_y\in\mathbb R^d$
 and $V_{i(x)}$ of $i(x)\in M$ such that the preimage by $f_{|V_{i(x)}}$ of $\psi(V_{u_y})$ is a
 $C^r$-submanifold.
Moreover, such a submanifold depends continuously on $f$ and $\psi$, with respect to the $C^r$-topologies.

More precisely, there exist neighborhoods  $V_{u_y}$ of $u_y$,
$V_f$ of $f\in End^r(M)$, $V_\psi$ of $\psi\in C^r(\mathbb R^{d_p},M)$, and $V_{i(x)}$ of $i(x)$
such that, for all $f'\in V_f$ and $\psi'\in V_\psi$, the map $f'_{|V_{i(x)}}$ is transverse to $\Psi'_{|V_{u_y}}$ and the preimage by $f'_{|V_{i(x)}}$ of $\psi'(V_{u_y})$ is a manifold which depends continuously on $f'$ and $\psi'$, in the compact-open $C^r$-topologies.

There exist neighborhoods $V_{t_y}$ of $t_y$ in $T_y$ and $V_i$ of $i\in Im^r(\mathcal T, M)$, such that
\[\psi_{i',t}\;:\; u\in \mathbb R^d\mapsto i'\circ \phi(u,t)\]
belongs to $V_\psi$, for all $t\in T_y$ and $i'\in V_i$.

Thus, the preimage by every $f'\in V_f$, restricted to $V_{i(x)}$, of the plaque $\mathcal L_{t}:= \phi^{-1}_y(V_{u_y}\times \{t\})$,
immersed by $i'\in V_i$, depends $C^r$-continuously on $f'$, $i'$, and $t'$.

Let $(U_x,\phi_x)$ be a chart of a neighborhood of $x$.
 Let us suppose that $\phi_x$ can be written in the form 
\[\phi_x\;: \; U_x\rightarrow \mathbb R^{d_p}\times T_x\]
where $T_x$ is a locally compact metric space. We define 
\[(u_x,t_x):=\phi_x(x)\; \mathrm{and}\; x_t:=\phi_x^{-1}(u_x, t), \; \forall t\in T_x.\]

For $V_{i(x)}$, and then $\eta'>0$ and $U_x$ small enough, the manifolds $(\mathcal F_{px'}^{\eta'})_{x'\in \mathcal L_{px_t}^{\eta'}}$
 are open subsets of leaves of a $C^r$-foliation on $V_{i(x)}$, which depends $C^r$-continuously on $t\in T_x$. We may suppose $U_x$ small enough to have its closure sent by $f^*$ into $\phi_y^{-1}(V_{u_y}\times V_{t_y})$.
 
For all $\eta'>0$ and then $V_i$ and $V_f$ small enough, each submanifold $\mathcal F_{px'}^{\eta'}$
 intersects transversally at a unique point  the submanifold $f'^{-1}_{|V_{i(x)}}\big(i'(\mathcal L_{t'})\big)$, where $t'$ is the second  coordinate of $\phi_y\circ f^*(x')$ and $x'$ belongs to $U_x$.
As we know $S^0(f',i')(x')$ is this intersection point.
 
In other words, $S^0(f',i')_{|\mathcal L_{x_t}^{\eta'}}$ is the composition of $i$ with the holonomy along the $C^r$-foliation $(\mathcal F_{px'}^{\eta'})_{x'\in \mathcal L_{px_t}^{\eta'}}$, from $i(\mathcal L_{px_t}^{\eta'})$ to the transverse section $f'^{-1}_{|V_{i(x)}}(i'(\mathcal L_{t'}))$, where $t'$ is the second coordinate of $\phi_y\circ f^*(x_t)$.
 
 Thus, the map $S^0(f',i')$ is $C^r$ along the $\mathcal L_p$-plaque contained in $U_x$.
 As these foliations and manifolds vary $C^r$-continuously with $x'\in U_x$, the map $S^0(f',i')$ is a
 $\mathcal L_{p|U_x}$-morphism into $M$.
 
  These foliations and manifolds also depend $C^r$-continuously on $x'\in U_x$, $i'\in V_i$, and $f'\in V_f$. Thus, the map
  \[S^0\;:\;(f',i')\in V_f\times V_i \mapsto S^0(f',i')\]
  is continuous into $Mor^r(\mathcal L_{p|U_x}, M)$.

  As, $C_p$ is compact, we get a finite open covers of $C_p$ by such open subsets $U_x$ 
  on which the restriction of $S^0$ satisfies the above regularity property. 
  By taking $V_i$ and $V_f$ small enough to be convenient for all the subsets of this finite covers, we get the continuity of the following continuous map:
  \[S^0\;:\;(f',i')\in V_f\times V_i \mapsto S^0(f',i')\in Mor^r(\mathcal L_{p|\hat V_{C_p}}, M).\]
where $\hat V_{C_p}$ is the union of the open covers of $C_p$.

As $S^0(f,i)=i$ is an immersion and $S^0$ is continuous, by restricting a slice $\hat V_{C_p}$, and by restricting $V_f$ and $V_i$, we may suppose that $S^0$ takes its values in the set of immersions from $\mathcal L_{p|\hat V_{C_p}}$ into $M$.

Let us show that $S^0$ is continuous from $V_f\times V_i$ into the set of $\mathcal T_{|\hat V_{C_p}}$-controlled morphisms. For this, we generalize the proof of the $\mathcal L_p$-regularity of $S^0$ and so we come back to the above notations. 

Let $x'\in \hat V_{C_p}$ and $U_x$ be one subset of the covers of $\hat V_{C_p}$, such that $x'$ belongs to $U_x$. Let  $X_j$ be the stratum that contains $x'$ and $(u'_{x},t'_x):= \phi_x(x')$. 

By property \ref{coher}, $\phi^{-1}(\mathbb R^{d_p}\times \{t_x\})$ is contained in $X_j$. By restricting a little slice $U_x$, a small enough neighborhood $T_{x'}$ of $t_x'\in T_x$ satisfies that $U_{x'}:= \phi_x^{-1}(\mathbb R^{d_p}\times T_{x'})$ is included in $V_j$ and is a distinguish open subset of the foliation $\mathcal L_{p|L_j\cap L_p}$ of $\mathcal L_{j|L_j\cap L_p}$.

That is why $T_{x'}$ can be supposed identifiable to $\mathbb R^{d_j-d_p}\times T_j$ (for some metric space $T_j$), such that the restriction 
\[\phi_{x|U_{x'}}\;:\; \rightarrow \mathbb R^{d_p}\times \mathbb R^{d_j-d_p}\times T_j\]
is a chart of the foliation $\mathcal L_{p|L_j\cap L_p}$ of $\mathcal L_{j|L_j\cap L_p}$.

Similarly, the manifolds $(\{v\}\times \mathcal F_{px''}^{\eta'})$ for $x''\in \mathcal L_{px_{(v,t)}}^{\eta'}$ and $v\in \mathbb R^{d_j-d_p}$ are open subsets of leaves of a $C^r$-foliation on $\mathbb R^{d_j-d_p}\times V_{i(x)}$, which depends $C^r$-continuously on $t\in T_j$.

By regularities of $f^*$ and $i'$, the following union is a $C^r$-submanifold: 
\[\bigcup_{v\in \mathbb R^{d_j-d_p}} \{v\}\times i'(\mathcal L_{t'_v}),\]
with $t'_v:= \phi_{y_2}\circ f^*\circ \phi_x^{-1}(0,v,t)$ and  $\phi_{y2}$ the second coordinate of $\phi_y$ and $t\in T_j$.

The map $(v,x'')\in \mathbb R^{d_j-d_p}\times V_{i(x)}\mapsto (v,f'(x''))$ is transverse to the above manifold since $f'_{|V_{i(x)}}$ is transverse to $i'(\mathcal L_{t'})$ for any $t'\in V_{t_y}$.

That is why the following  union is a submanifold:

\[\bigcup_{v\in \mathbb R^{d_j-d_p}} \{v\}\times f'^{-1}_{|V_{i(x)}}\big(i'(\mathcal L_{t'_v})\big).\]

For each $x''\in \mathcal L_{px_{(v,t)}}^{\eta'}$ and $v\in \mathbb R^{d_j-d_p}$, the above submanifold is transverse to the plaque  $(\{v\}\times \mathcal F_{px''}^{\eta'})$. The intersections is $(\{v\}\times \{S^0(f',i')(x'')\})$.

By holonomy, the map $S^0(f',i')$ is of class $C^r$ along the $\mathcal L_j$-plaque contained in $U_{x'}$. Moreover this foliation and this submanifold depend continuously on $t\in T_j$, $i'\in V_i$, and $f'\in V_f$. Thus, the restriction of $S^0$ to $U_{x'}$ is a continuous map from $V_f\times V_i$ (which do not have been restricted in this step) to $Mor^r(\mathcal L_{j|U_{x'}},M)$.

As this is true for any $x'\in \hat V_{C_p}$, we get the regularity of $S^0$.
   
By using tools defined later, we will proof in section \ref{preuve lem6} that conclusion \ref{preim} holds.  
                                                                       \end{proof}

In order to satisfy statements \ref{hatf}, \ref{N}, and \ref{cp} of the fundamental property, we have to pay attention on the way we glue $i_{p+1}(f')$ and $i'_p(f'):= S^0(f',i_{p+1}(f'))$, for $f'$ close to $f$.

 We shall begin by studying  the combinatorial topology.

\paragraph{Topological study}
This is the ``gluing area'':

\begin{propr}\label{delta}
Let $\Delta$ be the compact subset $C_p\cap K_{p+1}$.

There exists an open neighborhood $V_\Delta $ of $\Delta$, arbitrarily small which is precompact in  $\hat V_{C_p}\cap A_{p+1}$ and such that:\begin{itemize}
\item[\ref{delta}.1] $f^*(cl(V_\Delta))$ is included in $int(K_{p+1}\setminus V_\Delta)$,
\item[\ref{delta}.2] ${f^*}(cl(A_{p+1}))$ is disjoint from $cl(V_\Delta)$.
\end{itemize}
\end{propr}
\begin{proof}

As $\Delta$ is included in $K_{p+1}$, the open subset $A_{p+1}$ is a neighborhood of $\Delta$. Since $\Delta$ is included in $C_p$ the open subset $\hat V_{C_p}$ is a neighborhood of $\Delta $. Thus, a small enough neighborhood of $\Delta$ is included in $\hat V_{C_p}\cap A_{p+1}$.
 
As $\Delta$ is included in $K_{p+1}$, the endomorphism $f^*$ sends $\Delta$ into $int(K_{p+1})$, by \ref{Li'}.1. Moreover, $\Delta$ is included in $C_p\subset cl(K_{p+1}^c)$. Thus, a small enough neighborhood $V_\Delta $ of $\Delta$ satisfies \ref{delta}.1.

Since $\Delta$ is included in $cl(K_{p+1}^c)$ and since, by the induction hypothesis, $f^*(cl(A_{p+1}))$ is included in $int(K_{p+1})$, a small enough neighborhood $V_\Delta $ of $\Delta$ satisfies \ref{delta}.2.\end{proof}

Let $V_\Delta'$ be an open subset of $A$ that satisfies 
\[\Delta \subset V_\Delta'\subset cl(V_\Delta')\subset V_\Delta.\]

\begin{propr}\label{pro2}
For each $j\ge p$, there exist two precompact open neighborhoods $V_{C_j}'$ and $V_{C_j}$ of $C_j$, that satisfy $cl(V_{C_j}')\subset V_{C_j}$, such that 
 \[\mathrm{with}\quad A_p:=\bigcup_{j\ge p} V_{C_j}'\;\mathrm{and}\;
 A_{p+1}':=\bigcup_{j> p} V_{C_j},\; \mathrm{we\; have \;}: \]
 \begin{itemize}
 \item[\ref{pro2}.0] $A_p$ and $A_{p+1}'$ are neighborhoods of $K_p$ and $K_{p+1}$ respectively. Moreover, $A_{p+1}'$ is included in $A_{p+1}$,
\item[\ref{pro2}.1] $f^*$ sends $cl(A_p\cup V_{C_p})$ into $int(K_p)$ and

$f^*$ sends $cl(A_{p+1})$ into $int(K_{p+1})\setminus cl(V_{C_p}\cup V_\Delta)$,

 \item[\ref{pro2}.2] $cl(V_{C_j})\subset V_{X_j^{p+1}}$, for every $j> p$, and $cl(V_{C_p})\subset \hat V_{C_p}$,
 
\item [\ref{pro2}.3]
for any $x\in \hat V_{C_p}$, any $u\in Ti(T_x\mathcal L_p)^\bot$, the orthogonal projection of $T_xf(u)$ onto the subspace $Ti(T_{f^*(x)}\mathcal L_p)^\bot$ is nonzero.
 
\item[\ref{pro2}.4]  for every $j\ge p$, there exists a neighborhood of $f^*(cl(V_{C_j}))$ whose  points $x$ are sent by $i_{p+1}(f')$ into $\mathcal F_{jx}$, for every $f'\in V_f$,

 \item[\ref{pro2}.5] the intersection  $cl(V_{C_p})\cap cl(A_{p+1}')$ is included in $V'_{\Delta}$,
 
 \item[\ref{pro2}.6] we have $V_{C_p}\cup A_{p+1}'\cup int(A_p^c)=A$.

 \end{itemize}
\end{propr}

\begin{proof}
Since the union of compact subsets $(C_j)_{j\ge p}$ is equal to $K_p$ and since the union of compact subsets $(C_j)_{j\ge p+1}$ is equal to $K_{p+1}$, we easily get statement \ref{pro2}.0. When the neighborhoods $(V_{C_j})_{j\ge p}$ are small, the neighborhoods $A_p$ and $A'_{p+1}$ are close to respectively $K_p$ and $K_{p+1}$. Thus, for $(V_{C_j})_{j>p}$ small enough, the subset  $A_{p+1}'$ is included in $A_{p+1}$.

The fist part of \ref{pro2}.1 follows from \ref{Li'}.1 for $(V_{C_j}')_{j\ge p}$ and $V_{C_p}$ small enough. The second part of \ref{pro2}.1 is true for $V_{C_p}$ and $V_\Delta$ small enough, by the fundamental property which states that the compact subset $cl(A_{p+1})$ is sent into the interior of $K_{p+1}$.
 
Inclusion \ref{pro2}.2 is a consequence of fundamental property \ref{cp}, for $(V_{C_j})_{j>p}$ small enough. 

Inequality \ref{pro2}.3 is a consequence of \ref{Li'}.3, for $\hat V_{C_p}$ and $V_{\Delta}$ small enough.

Assertion \ref{pro2}.4 is a consequence of fundamental property \ref{N}, for neighborhoods $(V_{C_j})_{j> p}$ small enough.

To obtain statement \ref{pro2}.5, we fix $V_\Delta'$, then we take neighborhoods $(V_{C_j})_{j\ge p}$ small enough.

Statement \ref{pro2}.6 is obvious.

\end{proof}

We can fix definitively $(V_{C_j})_{j\ge p}$, $(V_{C_j}')_{j\ge p}$, $\hat V_{C_p}$, $V_{\Delta}$, and $V_{\Delta'}$.
 
\paragraph{Gluing lemma}\label{recol}
The idea of the proof is to glue $i_p':= S(f',i_{p+1}(f'))$ to  $i_{p+1}(f')$ over $V_\Delta$, and then to reapply the lemma \ref{prelem6} and so on. We will prove then that such a sequence converges to a certain controlled morphism $i_p(f')$.

In order to satisfy fundamental property \ref{N} and \ref{cp}, we have to take care on the way of gluing.
We must connect $i_{p+1}(f')(x)$ to $i_{p}'(f')(x)$ along the submanifold 
\[\mathcal F_{j x}^{f'}:=\mathcal F_{px}\pitchfork i_{p+1}(f')(\mathcal L_{jx}^{\eta'}),\; \forall j>p,\; x\in V_{C_j} \; and \; f'\in V_f.\]
The above intersections are well transverse by taking $\eta'$ and then $V_f$ sufficiently small.

The point $i_{p+1}(f')(x)$ belongs well to $\mathcal F_{jx}^{f'}$ lemma \ref{Ni} and fundamental property \ref{N}. Let us show that $i_{p}'(f')(x)$ belongs also to $\mathcal F_{jx}^{f'}$. 

By conclusion \ref{preconc4} of lemma \ref{prelem6}, this point belongs to $\mathcal F_{px}$. Let us show that $i_p'(f')(x)$ belongs to $i_{p+1}(f')(\mathcal L_{jx}^{\eta'})$.

It follows from conclusion \ref{preconc4} of lemma \ref{prelem6}, that the point $i'_p(f')(x)$ is sent by $f'$ into $i_{p+1}(f') \big(\mathcal L_{pf^*(x)}^{\eta'}\big)$.

By taking $\eta'$ smaller, the distance $d(f^*(V_{C_j}), L_j^c)$ is greater than $2\eta'$. Thus, by coherence of tubular neighborhoods, the plaque  $\mathcal L_{pf^*(x)}^{\eta'}$ is included in $\mathcal L_{jf^*(x)}^{\eta'}$.

\begin{equation}\label{46} \Rightarrow f'\Big(i'_p(f')(x)\Big)\in i_{p+1}(f') \big(\mathcal L_{jf^*(x)}^{\eta'}\big).\end{equation}

It follows from property \ref{pro2}.3, the coherence of tubular neighborhoods, property \ref{pro2}.2 and fundamental property \ref{hatf} that, for all $y$ close to $i(x)$ and $f'$ close to $f$, we have 
\begin{equation}\label{(c)}d\Big(y,i_{p+1}(f')\big(\mathcal L_{jx}^{\eta'}\big) \Big)\le d\Big(f'(y),i_{p+1}(f')\big(\mathcal L_{jf^*(x)}^{\eta'} \big)\Big).\end{equation}
Therefore, by restricting $V_f$, by (\ref{46}) and (\ref{(c)}), the point $y:=i'_p(f')(x)$ belongs to $ i_{p+1}(f')\big(\mathcal L_{jx}^{\eta'} \big)$.

Thus, $i_p'(f')(x)$ belongs to $\mathcal F_{jx}^{f'}$ for all $f'\in V_f$ and $x\in V_{C_j}$.
  
This is the gluing lemma:

\begin{lemm}\label{lem9}
By taking $\eta'>0$ smaller and then $V_f$ smaller, there exist a neighborhood $G$ of the graph of $i_{|V_\Delta}$ and a continuous map  
\[\gamma\;:\; V_f\rightarrow Mor^r\big((\mathcal T\times M)_{|G}\times [0,1],M\big)\]
such that, for all $f'\in V_f$ and $(x,y)\in G$:
\begin{enumerate}
\item  $\gamma(f')(x,y, 0)$ is equal to $i_{p+1}(f')(x)$,
\item  for every $t\in [0,1]$, the point $\gamma(f')(x,y, t)$ belongs to $\mathcal F_{j x}^{f'}$,  if $x$ belongs to $V_{C_j}'$, 
\item  if $y$ belongs to $\mathcal F_{j x}^{f'}$ for every $j>p$ such that $x$ belongs to $V_{C_j}'$, then $\gamma(f')(x,y, 1)$ is equal to $y$.
\end{enumerate}
\end{lemm}

Lemma \ref{lem9} will be proved in section \ref{structloc}. Let us proceed with the gluing. 

Let $\rho\in Mor^r(\mathcal T,[0,1])$ be a function with support in $V_\Delta$ and equal to 1 on $V_\Delta'$.

By taking $V_f$ sufficiently small,  we may define  
\[i^0(f'):A\rightarrow M\]
 \[x\mapsto \left\{\begin{array}{cl}
 \gamma(f')\Big(x,i_p'(f')(x),\rho(x)\Big)&\mathrm{if}\; x\in V_\Delta\\
i_{p+1}(f')(x)&\mathrm{if} \; x\in V_\Delta^c\end{array}\right..  \]

Since the support of $\rho$ is included in $V_\Delta$, by statement 1 of lemma \ref{lem9}, the map $i^0$ is continuous from $V_f$ into $Mor^r(\mathcal T,M)$.

By statement 3 of lemma \ref{lem9}, the morphism $i^0$ is equal to $i'_p$ on $cl( V_\Delta')$ and to $i_{p+1}$ on the complement of  $V_\Delta$.

\begin{propr}\label{plaque}
By restricting $\eta'$ and $V_f$, we may suppose that, for all $ f'\in V_f$, $ j>p$ and $ x\in V_{C_j}'$:
\begin{itemize}
\item[\ref{plaque}.1] the point $f'\circ i^0(f')(x)$ belongs to $i^0(f')\big(\mathcal L_{j{f^*(x)}}^{\eta'}\big)$,
\item[\ref{plaque}.2] the point $i^0(f')(x)$ belongs to $\mathcal F_{jx}^{f'}$.\end{itemize}
\end{propr}

\begin{proof}
Statement \ref{plaque}.2 is an obvious consequence of conclusion 2 of lemma \ref{lem9}, let us show statement \ref{plaque}.1.

For every $x\in V_{C_j}'\setminus V_\Delta$, the points $i^0(f')(x)$ and $i_{p+1}(f')(x)$ are equal. It follows from fundamental property \ref{hatf}, that the points $f'\circ i_{p+1}(f')(x)$ and $i_{p+1}(f')\circ f'^*_{p+1}(x)$ are equal. Since $x$ belongs to $A_{p+1}$, by restricting $V_f$, the point  $f'^*_{p+1}(x)$ never belongs to $V_\Delta$, thus the points $i_{p+1}(f')\circ f'^*_{p+1}(x)$ and $i^0(f')\circ f'^*_{p+1}(x)$ are equal.  Finally, by fundamental property 
\ref{cp} and property \ref{pro2}.2, the point $f'^*_{p+1}(x)$ belongs to $\mathcal L_{j{f^*(x)}}^{\eta'}$. Hence 
 \[ f'\circ i^0(f')(x)=f'\circ i_{p+1}(f')(x)=i_{p+1}(f')\circ f'^*_{p+1}(x)\]
 \[=i^0(f')\circ f'^*_{p+1}(x)\in i^0(f')\big(\mathcal L_{j{f^*(x)}}^{\eta'}\big)\]

By statement 2 of lemma \ref{lem9}, for all $x\in V_\Delta\cap V_{C_j}'$ and $f'\in V_f$, the point $i^0(f')(x)$ belongs to $i_{p+1}(f')(\mathcal L_{jx}^{\eta'})(x)$. Moreover, the points $i^0(f)(x)$ and $i(x)$ are equal. Thus, by taking $V_f$ sufficiently small, $f'$ sends $i^0(f')(x)$ into $i_{p+1}(f')(\mathcal L_{jf^*(x)}^{\eta'})$.

By \ref{delta}.1, the image of $cl(V_\Delta)$ by $f^*$ is disjoint from $cl(V_\Delta)$. Thus, on a neighborhood of $f^*(V_\Delta)$, the morphism $i^0$ is equal to $i_{p+1}$. Therefore, for $\eta'$ and $V_f$ small enough, the images of $\mathcal L_{j{f^*(x)}}^{\eta'}$ by $i^0(f')$ and by $i_{p+1}(f')$ are equal.

Therefore, $f'$ sends the point $i^0(f')(x) $ into $i_{p+1}(f')(\mathcal L_{jx}^{\eta'})= i^0(f')(\mathcal L_{jx}^{\eta'})$, for any $x\in V_\Delta\cap V_{C_j}'$. 
\end{proof}

Let us show, for all $f'\in V_f$, $q> p$, and $x\in  A_{p+1}\cap L_q$, that the differential  $\partial_{T_x\mathcal L_q} i^0(f')$ is injective. By fundamental property \ref{Im}, it is sufficient to check this for $x\in V_\Delta$. By fundamental property \ref{Im} and property \ref{plaque}.2, it is sufficient to prove that $\partial_{T_x\mathcal L_j} i^0(f')$ is injective for every $j>p$ such that $x\in V_{C_j}'$. For $f'=f$, the morphism $i^0(f')$ is equal to $i$, thus $\partial_{T_x\mathcal L_j} i^0(f')$ is injective. Since $V_{C_j}'$ is precompact in $L_j$ and since the map $i_{p|L_j}^0$ is continuous from $V_f$ into $Mor^1(\mathcal L_j,M)$, by restricting $V_f$, the differential $\partial_{T_x\mathcal L_j} i^0(f')$ is always injective.

\paragraph{Construction of $i_p$}

Let $\hat A_p$ be equal to $A_{p+1}'\cup V_{C_p}$. By property \ref{pro2}.6, $\hat A_p$ is a neighborhood of $cl(A_p)$.
\[\mathrm{Let}\; \mathcal M_p:= \{j\in Mor^r(\mathcal T_{|\hat A_p},M): j(x)\in \mathcal F_{px},\; \forall x\in L_p\}.\]
\[\mathrm{Let}\; \mathcal M_p^{f'}:= \{j\in \mathcal M_p,\; j_{|A_{p+1}'}=i^0(f')_{|A_{p+1}'}\}.\]
 Let $\mathcal M$ be the following set  
\[\prod_{f'\in V_f}\{f'\}\times \mathcal M_p^{f'}\]
endowed with the topology  induced by $End^r(M)\times Mor^r(\mathcal T_{|\hat A_p},M)$.

We notice that $(f',i^0(f')_{|\hat A_p})$ belongs to $\mathcal M$, for every $f'\in V_f$.

We are going to define $i_p$, by induction and by using the following lemma:
\begin{lemm}\label{lem6}
There exists a neighborhood $V_{f,i}$ of $(f,i)\in \mathcal M$  such that the following  map is well defined and continuous:
\[S\;:\; V_{f,i}\rightarrow V_{f,i}\]
\[(f',j)\mapsto (f',S_{f'}(j)),\]
\begin{center}
 where $S_{f'}(j)$ is equal to $i^0(f')$ on $A'_{p+1}$ and equal to $S^0(f',i')$ on $V_{C_p}$.
\end{center}

Moreover $S$ satisfies, for all $(f',j)\in V_{f,i}$:
\begin{enumerate}
\item \label{conc4} for every $x\in V_{C_p}$, the point $S_{f'}(j)(x)$ is the unique intersection point of $\mathcal F_{px}^{\eta'}$ with the preimage by $f'$ of $j(\mathcal L_{pf^*(x)}^{\eta'})$.

Let $f'^*_j(x)\in \mathcal L_{pf^*(x)}^{\eta'}$ defined by
\[f'\circ S_{f'}(j)=j\circ f'^*_j(x),\]
\item \label{conc5} for all $k\ge p$ and $x\in V_{C_p}\cap X_k$, if $\partial_{T X_k}j$ is injective at $f'_j(x)$, then
 $\partial_{T X_k}S_{f'}(j)$ is also injective at $x$,
\item for the distance defining the strong topology of $C^r$-morphism from the immersed lamination $\mathcal L_p$ (see section \ref{Top}), for every $\delta>0$, there exist a neighborhood $V_{f'}$ of $f'\in V_f$ and $M>0$ such that
for every $f''\in V_{f'}$, we have 
\[diam_{\mathcal L_p}\Big(
\{j_{|V_{C_p}}:\; (f'',j)\in S_{f''}^M(V_{f,i})\Big)<\delta\]
 
\end{enumerate}
\end{lemm}

For $V_f$ sufficiently small, we can now define, for all $k>0$, the continuous map
\[i^k\;: \;f'\in V_f\mapsto S_{f'}^k(i^0(f'))\in Mor^r(\mathcal T,M)\]

We are going to prove that, for every $f'\in V_f$, the sequence $(i^k(f'))_k$ converges to a morphism $i^\infty_p(f')\in Mor^r(\mathcal T_{|\hat A_p},M)$. Then $i_p(f')$ will be equal to $i^\infty_p(f')$ on $A_p$ and to $i$ on the complement of $A$.

\paragraph{Convergence of $(i^k)_k$}
Let us describe the values of $i^{k}_p$ on $V_{C_p}$, for every $f'\in V_f$.

By conclusion \ref{conc4} of lemma \ref{lem6}, for every $x\in V_{C_p}$, the point $i^k(x)= S_{f'}(i_p^{k-1}(f'))(x)$  depends only on $i^{k-1}(x_1)$, where $x_1:=f'^*_{i^{k-1}(f')}(x)$ is $\eta'$-close to $f^*(x)$ in a plaque of $\mathcal L_{p}$. By \ref{pro2}.1, the map $f^*$ sends $cl(V_{C_p})$ into $int(K_p)$, thus we may suppose that $x_1$ belongs to $K_p$.

If we suppose moreover that $x_1$ belongs to $A_{p+1}'^c$, then $x_1$ belongs to $K_p\setminus A_{p+1}'\subset V_{C_p}$. And so on, we can iterate this process which constructs an $\eta'$-pseudo chain $(x_i)_{i=0}^{n_x}$ of $f^*$, which respects the plaques of $\mathcal L_p$, defined by 
\begin{equation}\label{xn}
\left\{\begin{array}{c}
x_0=x\\
x_{i+1}:=f'^*_{i^{k-i-1}(f')}(x_{i}),\; \mathrm{and}\; f'\circ i^{k-i}(f')(x_i)= i^{k-i-1}(f')(x_{i+1})\end{array}\right.
\end{equation}

that we stop when ${i}$ is equal to $k$ or $x_{i}$ belongs to $ A_{p+1}'$. Therefore, we have $x_0=x,\dots,x_i\in K_p\setminus A_{p+1}',\dots,x_{n_x}\in A_{p+1}'\cap K_{p}$ or $n_x=k$.

We are going to prove that 
\begin{equation}\label{Uk}
\forall x\in  V_{C_p}, \quad i^k(f')(x)=S_{f'}^{n_x}(i^0(f'))(x)\end{equation}

For $n_x=k$, this equality is the definition of $i^k$.

For $n_x<k$, by \ref{pro2}.5, the point $x_{n_x}$ belongs to $A_{p+1}'$ thus, by  (\ref{xn}) and by decreasing induction on $i$ along the chain  $(x_i)_{i=0}^{n_x}$, we also have (\ref{Uk}) (since $n_{x_i}=n_x-i$). Moreover, $n_x$ does not change for a greater $k$.

 Thus, for every $x\in \hat A_p$,  the sequence $(i^k(f')(x))_k$ is eventually constant, by hypothesis $(iii)$ of the theorem, for $\eta'$ and then $V_{f}$ sufficiently small.

It follows from conclusion 3 of lemma \ref{lem6} and from the description of the values of $(i^k)_k$, that this sequence converges in $C^0(V_f,C^0({\hat A_p},M))$, to a certain map $i_p^\infty$.

Let $r\in Mor^r(\mathcal T,[0,1])$ be equal to $1$ on the $\eta'$-neighborhood of $A_p$ and to 0 on a neighborhood of the complement of $\hat A_p$ (we may reduce $\eta'$ if necessary).

\[\mathrm{Let}\;
 i_p:= f'\in V_f\mapsto\left[x\in A\mapsto \left\{\begin{array}{cl}
Exp\Big(r(x)\cdot Exp^{-1}_{i(x)}\big(i_p^\infty(x)\big)\Big)& \mathrm{if}\; x\in \hat A_p\\
i(x)&\mathrm{else}\end{array}\right.\right].\]

\paragraph{Properties of $i_p$}\label{prop:ip}

Let us begin by showing that $i_p$ is a continuous map from $V_f$ into $Mor^r(\mathcal T,M)$.
As the restriction of $i^\infty _p$ to  $ A_{p+1}'$ is a continuous map from $V_f$ into the space $Mor^r(\mathcal T_{|A'_{p+1}},M)$, and as $i$ and $r$ are $\mathcal T$-controlled $C^r$-morphisms, the restriction of $i_p$ to $ A_{p+1}'$ is a continuous map from $V_f$ into $Mor^r(\mathcal T_{| A_{p+1}'},M)$.
 
It follows from conclusion 3 of lemma \ref{lem6} that the restriction of $i_p$ to $V_{C_p}$ is continuous from $V_f$ into $Mor^r(\mathcal L_{p|V_{C_p}},M)$.

By taking $\eta'$ small enough, we may suppose this constant less than  $\inf_{V_{C_p}} \eta$, where $\eta$ is the function on $V_p$ provided by hypothesis $(iii)$ of the theorem. We define $U_j$ as the interior of the subset of points in $V_{C_p}$ which are not the starting point of any $\eta'$-pseudo-chain of $V_{C_p}$ which respects $\mathcal L_p$ and with length $j$. By property \ref{chaine}, the sequence of open subsets $(U_j)_j$ is increasing and its union is equal to $V_{C_p}\setminus X_p$. It follows from the description of the values of $(i^k)_k$ that, for $l>k>j$, the morphism $i^k(f')$ and $i_p^{\infty}(f')$ are equal on $U_j$. Thus,  $i_p$ is continuous from $V_f$ into $Mor^r(\mathcal T,M)$.

Let us show that $i_p$ satisfies fundamental property \ref{N}:

\begin{center}
for $j\ge p$ and $x$ in a neighborhood of $f^*(cl(V_{C_j}')$, which does not depends on $f'\in V_f$, the point $i_p(f')(x)$ belongs to $\mathcal F_{jx}$.
\end{center}

The above statement is obvious when $j$ is equal to $p$, by  lemma \ref{Ni} and fundamental property \ref{N} (at step $p+1$).

As $i_{p+1}$ and $i_p$ are equal on $A_p\setminus V_{C_p}\cup V_{\Delta}$, its is sufficient to prove that 
 $f^*(cl(V_{C_j}))$ is contained in $A_p\setminus cl(V_{C_p}\cup V_\Delta)$; this follows 
from \ref{pro2}.1 and the fact that $V_{C_j}$ is included in $A_{p+1}$.

Let us prove that $i_p$ satisfies fundamental property \ref{Im}:

\begin{center}
 the restriction of $i_p(f')$ to $A_p$ is an immersion, for every $f'\in V_f$.
\end{center}

By definition of $i_p$, on a neighborhood of $A_{p+1}'\cap A_p$, the map $i_p(f')$ is equal to $i^0(f')$.
We recall that  $A_{p+1}'$ is included in $A_{p+1}$ and that we showed that  $i_{p|A_{p+1}}^0(f')$ is an immersion. Thus, $i_{p|A_{p+1}'}(f')$ is an immersion.

Let us now study the restriction of  $i_{p}(f')$ to $V_{C_p}'\setminus A_{p+1}'$ which is equal to $A_p\setminus A_{p+1}'$. Let $x\in  (V_{C_p}'\setminus (X_p\cup A_{p+1}')$. We regard as before the pseudo-chain $(x_k)_{k=0}^{n_x}$ which respects the plaques of $\mathcal L_p$ and which is associated to $x$. We will show by decreasing induction on $k\in \{0,\dots,n_x\}$ that, when $x_k$ belongs to $X_l$, the tangent map $\partial_{T_{x_k}\mathcal L_l} S_{f'}^{n_x-k}(i^0(f'))$ is injective. 

For $k=n_x$, it follows from the fact that $i^0(f')_{|A_{p+1}'}$ is an immersion.

Let us suppose that $\partial_{T_{x_k}\mathcal L_l}S_{f'}^{n-k}(i^0(f'))$ is injective.
Therefore, it follows from conclusion \ref{conc5} of lemma \ref{lem6}, that $\partial_{T_{x_{k-1}}\mathcal L_l} S_{f'}^{n-k+1}(i^0(f'))$ is injective.

Thus, the restriction of $i_p$ to $A_p\setminus X_p$ is an immersion. As $A_p\cap X_p $ is precompact in $L_p\cap \hat A_p$, by continuity of $i_p$ and by restricting $V_f$, $i_{p|A_p}$ is a $\mathcal T$-controlled immersion.

\paragraph{Construction of a family of adapted neighborhoods $\mathcal V^p$}\label{consVp}

For $j\ge p$, by \ref{pro2}.2, we may suppose $\eta'$ small enough such that 
\begin{equation}\label{??}
    d(f^*(V_{C_j}'), L_j^c)>2\eta',
\end{equation}
and that the close subset $cl(\mathcal L_{jf^*(x)}^{\eta'})$ is a compact subset included in $\mathcal L_{jf^*(x)}^{2{\eta'}}$ which depends continuously on $x\in cl(V_{C_j}')$, in the space of nonempty compact subsets of $A$ endowed with the Hausdorff distance.  

We can now define $\mathcal V^p:= (V_{X_k^p})_{k=p}^N$ by 
\[V_{X_j^p}:=\left\{x\in L_j\cap A_p;\; \exists k\in \{p,\dots ,j\}\;:\; x\in V_{C_k}'\; \mathrm{and}\; cl\big(\mathcal L_{kf^*(x)}^{\eta'}\big)\subset L_j\right\}.\]

We remark that, for every $j\ge p$, the subset $V_{X_j^p}$ is open. Let us show that $V_{X_j^p}$ contains  $X_j^p:=X_j\cap A_p$. Let $x$ be a point of $X_j^p$. As $(V_{C_k}')_k$ covers $A_p$, there exists $k\in\{p,\dots,N\}$ such that $x$ belongs to $V_{C_k}'$. Since, $V_{C_k}'\subset L_k$ intersects $X_j$, the integer $k$ is not greater than $j$. Moreover, $f^*(x)$ belongs to $X_j$ and the compact subset $cl(\mathcal L_{kf^*(x)}^{\eta'})$ is included in $\mathcal L_{kf^*(x)}^{2\eta'}$. Thus, by property \ref{coher},
$cl(\mathcal L_{kf^*(x)}^{\eta'})$ is included in $X_j$, itself included in $L_j$. Therefore, $x$ belongs to $V_{X_j^p}$. This show that $X_j^p$ is contained in $V_{X_j^p}$.

Finally, we notice that  $V_{X_j^p}$ contains $V_{C_j}'$ and hence $C_j$, for every $j\ge p$. Therefore, a part of fundamental property \ref{cp} is shown.

\paragraph{Construction of $f'\mapsto f_p^{'^{*}}$}\label{fp}

Since for every $k\ge 0$, the morphism $i^{k+1}(f')$ is equal to $S_{f'}(i^{k}(f'))$, by conclusion \ref{conc4} of lemma \ref{lem6}, for $f'\in V_f$, $x\in V_{C_p}'$, and $k\ge 0$ we have 
\[f'\circ i^k(f')(x)\in i^k(f')\big(\mathcal L_{p_{f^*(x)}}^{\eta'}\big).\]
By taking the limit, as $k$ approaches the infinity, we get 
\begin{equation}\label{plaquep} f'\circ i_p(f')(x)\in cl\Big(i_p(f')\big(\mathcal L^{\eta'}_{pf^*(x)}\big)\Big)\subset i_p(f')\big(\mathcal L^{2\eta'}_{pf^*(x)}\big)\end{equation}
 
On $A_{p+1}'$, the map $i_p$ and $i^0$ are equal. 

Moreover, for $j>p$, the compact subset $cl(V_{C_j}')$ is contained in $A'_{p+1}$ which is sent into 
$K_p$, by \ref{pro2}.1.  Thus, on a neighborhood of the compact subset $f^*(cl(V_{C_j}'))$ the map $i_p$ and $i^0$ are equal. By the fact \ref{plaque}.1, we deduce 

for all $x\in V_{C_j}'$ and $f'\in V_f$, we have 
\begin{equation}\label{plaquej}
f'\circ i_p(f')(x)\in i_p(f')\big(\mathcal L_{j{f^*(x)}}^{\eta'}\big)
\end{equation}

Therefore, (\ref{plaquep}) and (\ref{plaquej}) imply that, by restricting $V_f$, we get 

for all $j\ge p$, $x\in V_{C_j}'$, and $f'\in V_f$, 

\begin{equation}\label{plaquefinal} f'\circ i_p(f')(x)\in i_p(f')\big(\mathcal L_{j_{f^*(x)}}^{\eta'}\big)\end{equation}

For every $k\ge p$, there exists a positive continuous function $\epsilon_k$ on $L_k$ such that, for any $x\in L_k$, the submanifolds  $(\mathcal F_{ky})_{y\in \mathcal L_{kx}^{\epsilon_k(x)}}$ are the leaves of a $C^r$-foliation on a neighborhood $U_{kx}$ of $i(x)$, which is transverse to $i(\mathcal L_{kx}^{\epsilon_k(x)})$.
For $y\in U_{kx}$, let $\pi_k^x(y)\in \mathcal L_{kx}^{\epsilon_k(x)}$ be the point such that $y$ belongs to 
$\mathcal F_{k\pi_k^x(y)}$.

Let $G_k:= \cup_{x\in L_k}\{x\}\times U_{kx}$. The set $G_k$ is an open neighborhood of the graph of $i_{|L_k}$.

 \[\mathrm{Let}\;\pi_k\;:\; G_k\rightarrow L_k\]
\[(x,y)\mapsto \pi_k^x(y)\]

\begin{lemm}\label{lem8} The map $\pi_k$ is $C^r$-$\big((\mathcal T\times M)_{|G_k},\mathcal T\big)$-controlled.
\end{lemm}
\begin{proof}
The map $N_k$ is $C^r$-$\mathcal T_{|L_k}$-controlled and hence $N_k^\bot$ is $C^r$-$\mathcal T_{|L_k}$-controlled.
Thus, the map $n_k$ from $L_k$ into the bundle over $M$ of linear endomorphism of $TM$, such that $n_k(x)$ is the orthogonal projection of $T_{i(x)}M$ onto $N_k^\bot(x)$, is also $C^r$-$\mathcal T_{|L_k}$-controlled.

The map $(x,y)\in G_k\mapsto Exp_{i(x)}^{-1}(y)$ is well defined and $C^r$-$(\mathcal T_{|L_k}\times M)_{|G_k}$-controlled. 

Thus, the map $(x,y)\in G_k\mapsto n_k(Exp_{i(x)}^{-1}(y))$ is $C^r$-$(\mathcal T_{|L_k}\times M)_{|G_k}$-controlled.

Let $(x_0,y_0)$ be in $G_k$ such that $x_0=\pi_k^{x_0}(y_0)$.  Let $U_{x_0}$ be a neighborhood of $x_0$ such that $TM_{|U_{x_0}}$ is canonically diffeomorphic to $U_{x_0}\times R^n$. We may suppose that in this identification 
$Ti(T_{x_0}\mathcal L_j)$ is sent onto $\{x_0\}\times \mathbb R^{d_k}\times \{0\}$.
For $U_{x_0}$ small enough, we may suppose that the canonical projection  $p_1\;:\; TM_{|U_{x_0}}\cong U_{x_0}\times \mathbb R^{d_k}\times \mathbb R^{n-d_k}\rightarrow \mathbb R^{d_k}$,  sends bijectively $Ti(T_x\mathcal L_k)$ onto $\mathbb R^{d_k}$, for every $x\in U_{x_0}$.

Let $X_l$ be the stratum which contains $x_0$. We may also suppose that the subset $U_{x_0}$ is a distinguish open subset of the $\mathcal L_{k|L_k\cap L_l}$-foliation of $\mathcal L_{l|L_l\cap L_k}$. In other words, there exists $\phi\;:\; U_{x_0}\rightarrow \mathbb R^{d_l}\times \mathbb R^{d_l-d_k}\times T$ which is a chart of $\mathcal L_k$ and $\mathcal L_l$.

Let us finally identify an open neighborhood of $y_0\in M$ to $\mathbb R^n$.
 
For $t\in T$, let us regard the following map:
\[\Psi_t\;:\; \mathbb R^{d_k}\times \mathbb R^{d_l-d_k}\times \mathbb R^n\rightarrow \mathbb R^{d_k}\]
\[(u,v,y)\mapsto p_1\circ n_k\circ Exp_{i(x)}^{-1}(y)\]
\begin{center}
with $x=\phi(u,v,t)$.
\end{center}

For $t$ close enough to $t_0$, the map $\Psi_t$ is well defined and of class $C^r$ on a neighborhood of $(u_0,v_0,y_0)$, with 
$\phi(x_0)=(u_0,v_0,t_0)$. Moreover $\Psi_t$ depends continuously on $t$.

We remark that $\psi_t(u,v,y)$ is equal to $0$ if and only if $(u,v,t)$ is equal to $\phi\circ \pi_{x_t}^k(y)$, where $x_t$ is equal to $\phi^{-1}(u_0,v_0,t)$.

The value of  $\Psi_{t_0}$ at $(u_0,v_0,y_0)$ is $0$. 
For $\epsilon_k$ small enough, the transversality of $\mathcal F_k$ with the image of $i$, implies that 
the derivative $\partial_{u}\Psi_{t_0}(u_0,v_0,y_0)$ is bijective. Hence by the implicit function theorem, there exist  neighborhoods $T_0$ of $t_0\in T$, $U_0$ of $u_0\in \mathbb R^{d_p}$, $V_0$ of $v_0\in \mathbb R^{d_k}$ and $Y_0$ of $y_0\in R^n$, and a continuous family $(\rho_t)_{t\in T_0}$ of $C^r$-maps from $V_0\times Y_0$ into $U_0$, such that 
\[\left\{\begin{array}{c} \rho_{t_0}(v_0,y_0) = u_0\\
\forall (u,v,t, y)\in U_0\times V_0\times T_0\times Y_0,\; \Psi_t(u,v,y)=0\Leftrightarrow u= \phi_t (v,y)\end{array}\right.\]

Thus, $\pi_{x_t}^k(y)$ is equal to $\phi^{-1}(\rho_t(v,y),v,t)$.  Therefore, the regularity of $\pi^k$ follows from the regularity of $(\rho_t)_t$.
\end{proof}

By taking $V_f$ small enough, for every $f'\in V_f$ the set \[\big\{\big(f^*(x),f'\circ i_{p}(f')\big);\; x\in V_{C_k}'\big\}\] is included in $G_k$. Thus, the following map is well defined, for every $f'\in V_f$,
 \[V_{C_k}'\rightarrow L_k\]
 \[x\mapsto \pi_k^{f^*(x)}\circ f'\circ i_p(f')(x)\]

It follows from assertion (\ref{plaquefinal}) and fundamental property \ref{N}, by taking $V_f$ small enough, for all $k\geq p$, $x\in V_{C_k}'$, and $f'\in V_f$,  we have 
\[i_p\circ \pi_k^{f^*(x)}\circ f'\circ i_p(f')(x):=f'\circ i_p(f')(x).\]
Thus, for every $l\ge p$,  on the intersection of $V_{C_k}'$ with $V_{C_l}'$, we have 
\[i_p\circ \pi_k^{f^*(x)}\circ f'\circ i_p(f')(x)=i_p\circ \pi_l^{f^*(x)}\circ f'\circ i_p(f')(x)\]
As $i_{p|A_p}$ is an immersion, by taking $V_f$ small enough, we always have
\[\pi_k^{f^*(x)}\circ f'\circ i_p(f')(x)=\pi_l^{f^*(x)}\circ f'\circ i_p(f')(x)\]
Therefore, we can define 
\[V_f\stackrel{C^0}{\longrightarrow}C^0\big(A_p,A_p\big)\]
\[f'\longmapsto \big(x\mapsto f_p^{'^*}(x)=\pi_j^{f^*(x)}\circ f'\circ i_p(f')(x),\; \mathrm{if}\; x\in V_{C_j}'\big)\]

Hence, the map $f'\mapsto (i_p(f'),f'^*_p)$ satisfies fundamental properties 1 and \ref{hatf}.

By smoothness of the maps which define $f'\mapsto f'^*_p$, to show that this last is continue from $V_f$ into $End_{f^*_{|A_p} \mathcal V^p}^r(\mathcal T_{|A_p})$, it is sufficient to prove fundamental property \ref{cp}: \[\forall k\ge p,\;\forall x\in V_k^p,\; \forall f'\in V_f,\; f'^*_p(x)\in \mathcal L_{kf'^*(x)}^{\eta'}.\]

By definition of $V_{X_k^p}$, if $x$ belongs to $V_{X_k^p}$, then there exists $j\in \{p,\dots ,k\}$, such that $x$ belongs to $V_{C_j}'$ and the plaque $\mathcal L_{jf^*(x)}^{\eta'}$ is included in $L_k$. As the point $f'^*_p(x)$ belongs to $\mathcal L_{jf^*(x)}^{\eta'}$, it belongs also to $\mathcal L_{kf^*(x)}^{\eta'}$. \begin{flushright}
$\square$
\end{flushright}

\subsubsection{Proof of lemma \ref{prelem6} and \ref{lem6}}
\label{preuve lem6}

\paragraph{Proof of the injectivity of $TS^0$}
Let us prove, for every $j\ge p$ and $x\in X_j\cap \hat V_{C_p}$, that the partial derivative $\partial_{T_x X_j}( S_{f'}(i'))$ is injective when 
$\partial_{T_{{f'^*_{i'}}(x)} X_j}(i')$ is injective, for all $f'\in V_f$ and $i'\in V_i$.

As we have ever seen that $S^0(f',i')$ is an immersion of $\mathcal L_{p|\hat V_{C_p}}$ into $M$, it remains to prove that, for $u\in T_xX_j\setminus T_x\mathcal L_p$, the vector $T(S^0(f',i'))(u)$ is nonzero, for $j>p$.

By lemma \ref{lem8}, the following maps is $(\mathcal T_{|\hat V_{C_p}},\mathcal T_{|L_p})$-controlled:
\[f'^*_{\sigma}\;:\;\hat V_{C_p}\rightarrow L_p\]
\[x\mapsto \pi^{f^*(x)}_p\big( f'\circ S^0(f',i')(x)\big)\]

 By definition of $S^0$, we have for every $x'\in \hat V_{C_p}$,
\[f'\circ S^0(f',i') (x')=i'\circ {f'_{i'}}^*(x').\]
This implies that
\[\partial_{T_xX_j}\big( f'\circ S^0(f',i') \big)=\partial_{T_xX_j}\big( i'\circ {f'_{i'}}^*\big).\]

Thus, it is sufficient to prove that the vector $T\big( i'\circ {f'_{i'}}^*\big )(u)$ is nonzero.
As we assume that $\partial_{T_{f'^*_{i'}(x)}X_j}( i')$ is injective, it is sufficient to prove that $T{f'_{i'}}^*(u)$ is nonzero.

As ${f'_{i'}}^*$ belongs to $Mor_{f^*}(\mathcal L_{p|\hat V_{C_p}},\mathcal L_p)$, it is sufficient to show that  $T_x{f}^*(u)$ does not belong to $T_{f^*(x)}\mathcal L_p$. By property \ref{pro2}.3 and injectivity of $Ti$, the vector $Tf\circ T_xi(u)$ does not belong to $Ti(T_{f^*(x)}\mathcal L_p)$. Therefore, by commutativity of the diagram, the vector $T_xf^*(u)=(T_{f^*(x)}i)^{-1}\circ T(f\circ i)(u)$ does not belong to $T_{f^*(x)}\mathcal L_p$.

 \paragraph{Definition of $S_{f'}$}
\[\mathrm{For}\;  f'\in V_f,\; \mathrm{let}\; S_{f'}\;:\; i'\in V_i\cap \mathcal M_p^{f'}\longmapsto
\left[  
x\mapsto \left\{\begin{array}{cl} 
i^0(f')(x)& \mathrm{if}\; x\in A_{p+1}'\setminus V_{C_p}\\
S^0_{f'}(i')&\mathrm{if}\; x\in V_{C_p}\end{array}\right.
\right]\]
where $S^0_{f'}:=S^0(f',\cdot)$ was defined in lemma \ref{prelem6}.

By definition, the restriction to $V_\Delta'$ of 
 $i^0(f')$ is equal to the one of $i'_{p}(f')$.
Since the image by $f^*$ of $V_\Delta$ is disjoint from $cl(V_{C_p}\cup V_\Delta)$, by restricting $V_f$, for $f'\in V_f$ and $i'\in V_i\cap \mathcal M_p^{f'}$, the maps $i'_{p}(f')$ and $S^0_{f'}(i')$ are equal on $V_\Delta'$.  Thus, $i^0(f')$ and $S^0_{f'}( i')$ are equal on $A_{p+1}'\cap V_{C_p}\subset V_\Delta'$.  

Therefore, the values of $S_{f'}$ are in $\mathcal M_p^{f'}$.

Let $\tilde K_p$ be the compact set $cl(V_{C_p}\cap X_p)$. Let $U $ be a small open neighborhood of $i(\tilde K_p)$ in $M$ such that there exists a continuous extension $\chi$ of the section of the Grassmannian of $d_p$-plan of $TM_{|i(\tilde K_p)}$ defined by\footnote{It follows from property \ref{propertyA} that such a section is well defined.}  
\[\chi(y)=Ti(T_x\mathcal L_p),\; \mathrm{if}\; x\in \tilde K_p\;\mathrm{ and}\; y=i(x)\]

We endow $M$ with an adapted metric satisfying property \ref{property B}, for the normally expanded lamination $X_p$ over the compact subset $\tilde K_p$. Let $\exp$ be the exponential map associated to this metric.

Hence, by restricting $U$, there exist $\lambda>1$ and a open cone field $C$ over $U$ of $i(\tilde K_p)$ such that:
\begin{itemize}
\item for all $x\in i(\tilde K_p)$, $u\in \chi(x)$, $v\in \chi(x)^{\bot}\setminus \{0\}$, we have 
 \[\lambda\cdot \|T_xf(u)\|^r< \|p\circ T_xf (v)\|,\]
with $p$ the orthogonal projection of $T_{f(x)}M$ onto $\chi_{f(x)}^\bot$, 
\item for each $x\in \tilde K_p$, the space $Ti(T_x\mathcal L_p)^\bot$ is a maximal vector subspace included in $C(x)$,
\item there exist $\delta>0$ and $\lambda<1$ satisfying for all $x\in U$ and $u\in cl(C(x))\setminus \{0\}$ with norm less than $\delta$,
\[v:=\exp_{f(x)}^{-1}\circ f\circ \exp(u)\in C(f(x))\quad\mathrm{and}\quad \lambda \|v\|>\|u\|.\]
\end{itemize}

\paragraph{The $C^0$-contraction}
Let $U^*$ be a neighborhood of $X_p\cap \hat A_p$ in $V_{C_p}\setminus V_\Delta$ such that $cl(U^*\cup f^*(U^*))$ is sent by $i$ into $U$.

For $V_i$ and $V_f$ small enough, any $f'\in V_f$ and $i'\in V_i$ satisfy  
\[S^0(f',i')(U^*)\subset U\; \mathrm{and}\; i'\circ f'^*_{i'}(U^*)\subset U.\]

Consequently, for $\eta'>0$, $\delta>0$ and then $V_i$ and $V_f$ small enough, for all $f'\in V_f$, $(i',i'')\in \big(V_i\cup S^0_{f'}( V_i)\big)^2$, and $x\in U^*$, we have:
\begin{itemize}
\item for every $u\in C_{i'(x)}$ with norm less than $\delta$, the vector $\exp_{f'(x)}^{-1}\circ f'\circ \exp(u)$ belongs to $C_{f'(x)}$ and has norm greater than $ \|u\|/\lambda$, 

\item The following number is equal to $0$ if and only if $i'(x)$ is equal to $i''(x)$:
\[d_x(i',i''):=\sup_{u\in cl(C_{i'(x)}),\; \|u\|< \delta}\Big\{\|u\|;\; \exp(u)\in i''(\mathcal L_{px}^\delta)\Big\}\]
\item the number $d_x(i',i'')$ is also equal to
\[\sup_{u\in cl(C_{i'(x)}),\; \|u\|< \|Tf_{|U}\|\cdot \delta}\Big\{\|u\|;\; \exp(u)\in i''(\mathcal L_{px}^{\delta+2\eta'})\Big\}.\]
\end{itemize}

By definition of $S^0$ and $f'^*_{\cdot}$, the map $f'$ sends $S^0_{f'}(i')(x)$ to $i'\circ f'^*_{i'}(x)$.
For every $u\in C_{S^0(f',i')(x)}$ with norm less than $\delta$, if $\exp(u)$ belongs to $S^0_{f'}(i'')(\mathcal L_{px}^\delta)$, then it is sent by $f'$  into $i''(\mathcal L_{pf^*(x)}^{\delta+\eta})\subset i''(\mathcal L_{pf'^*_{i'}(x)}^{\delta+2\eta})$.

 Thus, we have    
\begin{equation}\label{C0:contraction} d_x(S^0_{f'}(i'),S^0_{f'}(i''))\le \lambda 
d_{f'^*_{i'}(x)}(i',i'').\end{equation}

If $d$ denote the Riemannian distance on $M$, we notice that $d_x$ and \[d_x^{\infty}\; :\; (i',i'')\mapsto d(i'(x),i''(x))\] are uniformly equivalent on $U^*$.

Therefore, the map
\[d_0\; : \; (i',i'')\mapsto \sup_{x\in U^*} \big(d_x(i',i'')+d_x(i'',i')\big)\]
defines a semi-distance on $V_i\cap S^0_{f'}( V_i)$, for every $f'\in V_{f}$.

Moreover, this semi-distance is equivalent to the semi-distance

\[d_{\infty}(i',i'')=\sup_{x\in U^*}d_x^\infty(i',i'').\]

In fact, the equivalence is uniform for $f'\in V_f$, since $d_x$ and $d_x^\infty$ do not depend on $f'\in V_f$.

As $S_f(i)$ is equal to $i$, by (\ref{C0:contraction}), for any $i'\in V_i$, we have

\[d_0(i,S^0_f(i))=d_0(S^0_f(i),S^0_f(i'))\le {\lambda} d_0(i,i').\]
Thus, for any $\epsilon\in ]0,\delta[$, the intersection of the $\epsilon$-ball centered at $i$ with $S_{f}^N (\mathcal M_p^{f'}\cap V_i)$ is sent by $S_f$ into the $\epsilon/\lambda$-ball centered at $i$, with respect to the $d_0$-distance.

By continuity of $f'\mapsto S_{f'}$, by restricting $V_f$, for every $f'\in V_f$, the intersection of $V_i$ with the $\epsilon$-ball centered at $i$ is sent by $S^0_{f'}$ into the $\epsilon$-ball centered at $i$.   

Moreover, by equation (\ref{C0:contraction}), the map $S^0_{f'}$ is $\lambda$-contracting on $V_i$, for any $f'\in V_f$, with respect of the semi-distance $d_0$.

\paragraph{the $1$-jet space}\label{r=1}
Contrarily to what happened for the $C^0$-topology, the map $S_f$ is generally not contracting for the $C^1$-topology. However, we are going to show that the backward action of $Tf$ on the Grassmannian of $d_p$-plans of $TM$ is contracting, at the neighborhood of the distribution induced by $TX_p$. 

The bundle of linear maps from $\chi$ into $\chi^\bot$, denoted by $P^1$,  
 is canonically  homeomorphic (via the graph map) to an open neighborhood of $\chi$ in the Grassmannian of $d_p$-plans of $TM_{|U}$.
 In this identification the zero section is sent to $\chi$.
 
 We endow $P^1$ with the norm subordinate to the adapted Riemannian metric of $M$.
 
 By normal expansion, for $U$ small enough and $f'$ close enough to $f$, for every $x\in U$ sent by $f'$ into some $y\in U$, any $l\in P^1_y$ small enough, the preimage by $T_xf'$ of the graph of $l$ is the graph of a small $l'\in P^1_x$.
  
Let us show the following lemma:
\begin{lemm}
For all $\epsilon>0$, then $U$ and then $V_f$ small enough, for all $f'\in V_f$, $x\in U\cap f'^{-1}(U)$ and $l\in P^1_{f'(x)}$ with norm not greater than $\epsilon>0$, then the norm of $\phi_{f'x}(l):=l'\in P^1_x$ is less than $\epsilon$. Moreover, the map $\phi_{f'x}$ is $\lambda$-contracting.\end{lemm}
\begin{proof}
Let $\pi_v$ and let $\pi_h$ be the orthogonal projections of $TM_{|U}$ onto respectively $\chi^\bot$ and $\chi$.
Let $Tf'_h:=\pi_h\circ Tf'$ and $Tf'_v:=\pi_v\circ Tf'$ which are defined on $f'^{-1}(U)$.

For any vector $e'\in \chi(x)$, the point $(e',l'(e'))$ is sent by $T_xf'$ onto 
\[\big(Tf'_h(e',l'(e')),Tf'_v(e',l'(e'))\big).\]
Let $e:=  Tf'_h(e',l'(e'))$. By definition of $l'$, the point $(e',l'(e'))$ is sent by $T_xf'$ onto
$(e,l(e))$.

Therefore, we have $l(e)= Tf'_v(e',l'(e'))$ and $l(e)=l\circ Tf'_h(e',l'(e'))$.
\[\Rightarrow Tf'_v(e',l'(e'))= l\circ Tf'_h(e',l'(e'))\]
\[\Rightarrow (Tf'_v-l\circ Tf'_h)(l'(e'))=(l\circ Tf'_h-Tf'_v)(e')\]
By normal expansion, for $\epsilon$, $U$ and $V_f$ small enough, the map $(Tf'_v-l\circ Tf'_h)_{|\chi^\bot}$ is bijective. Consequently,
\[l'(e')=(Tf'_v-l\circ Tf'_h)_{|\chi^\bot}^{-1}(l\circ Tf'_h-Tf'_v)(e')\]
Hence, the expression of $l'=\phi_{f'x}(l)$ depends algebraically  on $l$, and the coefficients of this algebraic expression depend continuously on $f'$ or $x$, with respect to some trivialisations.

When $f'$ is equal to $f$ and $x$ belongs to $i(\tilde K_p)$, the map 
\[\phi_{fx} \;:\;l'\mapsto (Tf_v-l\circ Tf_h)_{|\chi^\bot}^{-1}(l\circ Tf_h)\]
is $\lambda$-contracting for $l$ small, by normal expansion.

Thus, for $U$ small enough, for all $x\in U\cap f^{-1}(U)$ and $l \in B_{P^1_{f^*(x)}}(0,\epsilon)$, the tangent map of $\phi_{fx}$ has a norm less than $\lambda$.
Therefore, by restricting a slice $U$, for $V_f$ small enough, the tangent map $T\phi_{f'x}$ has a norm less than $\lambda$ and hence $\phi_{f'x}$ is a $\lambda$-contraction on $cl\big({B}_{P_y}(0,\epsilon)\big)$. 

As, for $x\in i(\tilde K_p)$, the map $\phi_{fx}$ vanishes at $0$, for $U$ and $V_f$ small enough, the norm $\phi_{f'x}(0)$ is less than 
\[(1-\lambda)\cdot \epsilon.\]

Consequently, by $\lambda$-contraction, the closed $\epsilon$-ball centered at the $0$-section is sent by $\phi_{f'}$ into the $\epsilon$-ball centered in $0$, for all $x\in f'^{-1}(U)\cap U$ and $f'\in V_f$.
 \end{proof}

\paragraph{Proof of the lemma \ref{lem6} when $r=1$}
Let us  begin by proving the existence of a neighborhood $V_{f,i}$ sent by $S$ into itself.

Following the arguments of sections \ref{prop:ip}, for $V_i$ and $V_f$ small enough, there exists $N>0$ such that the subset $V_{C_p}\setminus U^*$ is included in $U_N$\footnote{$U_N$ is the set of points $x\in V_{C_p}$ such that there is no any $\eta'$-chain of $V_{C_p}$ which respects $\mathcal L_{p}$, which start at $x$, and with length $N$.}.
In particular, the restriction of $S^{N}_{f'}(i')$ to $U^{*c}$ does not depend on $i'$ but only on $f'$, for $(f',i')$ in the domains of $S^N$.

For $\epsilon>0$ and then $V_f$ small enough, the set  

\[V^{f'\epsilon}:= 
\Big\{i'\in \mathcal M_p^{f'}; \quad i'=S^N_{f'}(i^0(f'))\; \mathrm{on} \; U^{*c};\; d_{0}(i,i')< \epsilon,\]
\[\mathrm{and}\; \; Ti(T_x\mathcal L_{p|U^*})\in B_{P^1_x}(0,\epsilon),\; \forall x\in U^*\Big\}\]
is nonempty and included in $V_i$.
 Therefore, by the  
two last steps and by the equality of $S$ and $S^0$ on $U^*\subset V_{C_p}\setminus V_\Delta$, the map $S_{f'}$ sends $V^{f'\epsilon}$ into it self. Thus, the following neighborhood of $(f,i)\in\mathcal M$ is sent by $S$ into itself:
\[V_{f,i}:=\bigcup_{f'\in V_f,\; 0\le k\le N}\{f'\}\times S^{-k}_{f'}(V^{f'\epsilon})\]

We remark that, for any $f'\in V_f$, the restriction of $d_0$ to $V^{f'\epsilon}$ is a distance, which is equivalent
 to the distance defining the strong $C^0$-topology, for which $S_{f'}$ is $\lambda$-contracting.
 
 By restricting $V_i$ and $V_f$, and hence $V_{f,i}$, we may suppose that the set 
\[F:= \bigcup_{(f',i')\in V_{f,i}} i'(U^*)\cup f'\circ i'(U^*)\]
has its closure included in the interior of the set

\[O:= \bigcap_{(f',i')\in V_{f,i}} U\cap f'^{-1}(U).\]
Let $r\in C^0 (U,[0,1])$ be equal to $1$ on $F$ and to $0$ on $O^c$.

\[\mathrm{Let}\; r\phi\;:\; V_f\times cl\big(B_{\Gamma^0P^1|U}(0,\epsilon)\big)
\rightarrow cl\big(B_{\Gamma^0P^1|U}(0,\epsilon)\big)\]
\[(f',\sigma)\longmapsto\left[ r\cdot \phi_{f'}\; :\;x\mapsto r(x)\cdot \phi_{f'x}\big(\sigma\circ f'(x)\big)\right]\]

where $\Gamma^0P^1|U$ is the set of the continuous sections of the bundle $P^1$ restricted to $U$.

For every $f'\in V_f$,  the map $r\phi_{f'}$ is still $\lambda$-contracting and hence has a fixed continuous section $\sigma_{f'}$, which depends continuously on $f'\in V_f$.

Thus, for every $f'\in V_f$, there exists a neighborhood $V_{f'}$ of $f'\in V_f$, such that, 
\[\|\sigma_{f'}-\sigma_{f''}\|_\infty<\delta, \; \forall f''\in V_{f'}.\]

By uniform continuity of $\sigma_{f'}$, there exists $e>0$ such that 
\[\forall(x,y)\in F^2,\;\mathrm{if}\; d(x,y)<e\; \mathrm{then}\; d(\sigma_{f'}(x),\sigma_{f'}(y))<\delta.\]

Let $N'> N$ such that $\lambda^{N'-N}\epsilon$ is less than $\delta$ and $e$. Thus,  for every $f''\in V_{f'}$, the diameter of $r\phi_{f''}^{N'}(B(0,\epsilon))$ is less that $\delta$ and the $C^0$-diameter of $\{i'\in V_i:\; (f'',i')\in S^{N'}(V_{f,i})\}$ is less than $e$.

For $N''>N'$ large enough, for all $f''\in V_{f'}$  and $((f'',i''),(f'',i'))\in S^{N''}(V_{f,i})^2$, we have for every $x\in V_{C_p}$:
\begin{itemize}
\item either $f''^{k} (i''(x))$ and $f''^{k}(i'(x))$ belong to $F$, for all $x\in\{0,\dots, N'\}$, hence $Ti''(T_x\mathcal L_p)$ and $Ti'(T_x\mathcal L_p)$ belong to $r\phi_{f''}^{N'}(B(0,\epsilon))$,
\item either $i''(x)$ is equal to $i'(x)$, and so $Ti''(T_x\mathcal L_p)$ and $Ti'(T_x\mathcal L_p)$ are equal.
\end{itemize}

In the first case, $d(Ti''(T_x\mathcal L_p),Ti'(T_x\mathcal L_p))$ is not greater than 

\[d\Big(Ti''(T_x\mathcal L_p),\sigma_{f''}\big(i''(x)\big)\Big)+ 
d\Big( \sigma_{f''}(i''(x)),  \sigma_{f'}(i''(x))\Big)+ 
d\Big(\sigma_{f'}(i''(x)), \sigma_{f'}(i'(x))\Big)\]\[+
d\Big( \sigma_{f'}(i'(x)), \sigma_{f''}(i'(x))\Big)+
d\Big(\sigma_{f''}\big(i'(x)),Ti''(T_x\mathcal L_p\big)\Big)\]
\[\Rightarrow d(Ti''(T_x\mathcal L_p),Ti'(T_x\mathcal L_p))\le 5\delta.\]
 
The last inequality concludes the proof of conclusion $3$, when $r=1$.

\paragraph{General case: $r\ge 1$}
As we deal with the $C^r$-topology, we should generalize the Grassmannian concept as follow:

For $x\in U$, let $G_x^r$ be the set of the $C^r$-$d_p$-submanifolds of $M$, which contain $x$, quotiented by the following $r$-tangent relation:

Two such submanifolds $N$ and $N'$ are equivalent if there exits a chart $(U,\phi)$ of a neighborhood of $x\in M$, which sends $N\cap U$ onto $\mathbb R^{d_p}\times \{0\}$ and sends $N'\cap U$ onto the graph of a map from $\mathbb R^{d_p}$ into $\mathbb R^{n-d_p}$, whose $r$-first derivatives vanish at $\phi(x)$.

We notice that for $r=1$, this space is the Grassmannian of $d_p$-plans of $T_xM$.

As we are interested in the submanifold ``close'' to the embedding of the small $\mathcal L_p$-plaques by $i$,  we restrict our study to the manifolds whose tangent space at $x$ is complement  to $\chi(x)^\bot$.

The preimage of such submanifolds by the map $\exp_x$ is a graph of some $C^r$-maps $\overline{l}$ from $\chi(x)$ to $\chi(x)^\bot$. 

Moreover, their $r$-tangent equivalence class can be identified to the Taylor polynomial of $\overline{l}$ at $0$:
\[\overline{l}(u)=T_0\overline{l}(u)+\frac{1}{2}T^2_0\overline{l}(u^2)+\cdots +\frac{1}{r!} T^r_0\overline{l}(u^r)+o(\|u\|^r)\]
where $u$ belongs to $\chi(x)$, the $k^{th}$-derivative $T_0^k\overline{l}$ belongs to the space $L_{sym}^k(\chi(x),\chi(x)^\bot)$ of $k$-linear symmetric space from $\chi(x)^k$ to $\chi(x)^\bot$. We notice that we abused of notation by writing $u^k$ instead of $(u,\dots,u)$.

The map $l(u):= \sum_{k=1}^n \frac{1}{k} T^k_0\overline{l}(u)$ is an element of the vector space 
\[P_x^r:= \prod_{k=1}^r L_{sym}^j (\chi(x),\chi(x)^\bot).\]

Conversely, any vector $l\in P_x^r$, that we will write in the form 
\[l\::\;u\in \chi(x)\mapsto \sum_{k=1}^r l_k(u^k)\]
is the class of the following $C^r$-$d_p$-submanifold:
 \[\exp\Big(\big\{(u+l(u));\; u\in \chi(x)\; \mathrm{and} \; \|u\|<r_i(x)/2\big\}\Big),\]
 where $r_i$ is the injectivity radius of $\exp$.

The linear map $l_1$ from $\chi(x)$ to $\chi(x)^\bot$ will be called the \emph{linear part of $l$}. We notice that $l_1$ belongs to $P^1_x$.

We denote by $P^r$ the vector bundle over $U$, whose fiber at $x$ is $P^r_x$.

 By normal expansion, for $U$ small enough and $f'$ close enough to $f$, for any point $x\in U$ sent by $f'$ into some $y\in U$, any $l\in P^r_y$ whose linear part is small enough, the preimage  by $f'$ of a representative of $l$ is a representative of vector $\phi_{f'x}(l)\in P^r_x$, which only depends on $l$.

Let us show the following lemma:
\begin{lemm}
For every $\epsilon>0$ small enough and then  $V_f$ and $U$ small enough, for all $f'\in V_f$, $x\in U\cap f'^{-1}(U)$, and $l\in P^r_{f'(x)}$ with linear part of norm not greater than $\epsilon>0$, the norm of the linear part of $\phi_{f'}(l)$ is less than $\epsilon$. Moreover the map $\phi_{f'}$ is $\lambda$-contracting for a norm on $P^r$ which does not depends on $x$ or $f'$.\end{lemm}

\begin{proof}
We have showed in $r=1$, that $\phi_{f'x}$ sends the set of vectors of $P'$, and hence those of $P^r$, of linear part not greater than $\epsilon$ into it self. Let us show the $\lambda$-contraction of $\phi_{f'x}$. 

Let $l':=\phi_{f'}(l)$. Let $J_x^rf'$ be the $r$-jet map of $f'$ at $x$ (see \cite{Michor}).

We recall that that the $r$-jet $J_x^rf'$ of $f$ at $x$ is a vector of 
\[\prod_{j=1}^r L_{sym }^j (T_xM, T_{f'(x)}M)\] such that, if we denote by $f_j'$ its component in 
$L_{sym}^j(T_xM,T_{f'(x)}M)$, we have 

\[\exp^{-1}_{f'(x)}\circ f'\circ \exp_x(u)=\sum_{j=1}^r f_j(u^j) + o(\| u\|^r),\; \mathrm{for} \; u\in \chi(x)\]
By definition of $l':= \phi_{f'x}(l)$, for any $u'\in \chi(x)$, there exist $u\in \chi(f'(x))$ such that 

\begin{equation}\label{mise en valeur du jet de f}
J_x^rf'(u'+l'(u'))=(u+l(u))+o(\|u\|^r)\end{equation} 
  
We recall that $\pi_v$ and $\pi_h$ denote the orthogonal projection of $TM_{|U}$ onto respectively $\chi^\bot$ and $\chi$.

By (\ref{mise en valeur du jet de f}), we have 
\[u:= \pi_h\circ J^r_xf'(u'+l'(u'))+o(\|u'\|^r)\; \mathrm{and}\; l(u)=\pi_v\circ J^r_xf'(u'+l'(u'))+o(\|u'\|^r).\]

Thus, we have 
\begin{equation}\label{1eq}
l\circ \pi_h\circ J^r_xf'(u'+l'(u'))=\pi_v\circ J^r_xf'(u'+l'(u'))+o(\|u'\|^r).\end{equation}
We have 
\begin{equation}\label{2eq} J^rf'_x(u'+l'(u'))=\sum_{I\in R} f'_{|I|}\Big[\prod_{k\in I} l'_k(u'^k)\Big]+o(\|u'\|^r),\end{equation}
where $R$ is the set $\cup_{k=1}^r\{0,\dots ,r\}^k$, $l'_0(u'^0)$ is equal to $u'$, and for $I\in R$, $|I|$ is the length of $I$. 

Let $f_{kv}$ and $f_{kh}$ be respectively the linear maps $\pi_v\circ f_k$ and $\pi_h\circ f_k$ respectively, for every $k\in \{0,\dots, r\}$.

It follows from equations (\ref{1eq}) and (\ref{2eq}) that 
\begin{equation}\label{agd}l\Big(\sum_{I\in R} f'_{|I|h}\big[ \prod_{k\in I} l'_k(u'^k)\big]\Big)
=\sum_{I\in R} f'_{|I|v}\big[\prod_{k\in I}l_k'(u'^k)\big]+o(\|u'\|^k).\end{equation}

On the one hand, we have

\begin{equation}\label{ad}
\sum_{I\in R} f'_{|I|v}\big[\prod_{k\in I}l_k'(u'^k)\big]=\sum_{m=1}^r\left[ \sum_{I\in R,\; \Sigma I=m}f'_{|I|v} \big[\prod_{k\in I}l'_k(u'^k)\big]\right]+o(\|u'\|^r)\end{equation}

 with, for every $I\in  R$, $\Sigma I$ equal to $\sum_{j\in I} j$ plus the number of times that $0$ belongs to $I$.
 
On the other, as we have

\[l\Big(\sum_{I\in R} f'_{|I|h}\big[ \prod_{k\in I} l'_k(u'^k)\big]\Big)=
 \sum_{a=1}^r l_a\Big(\sum_{I\in R} f'_{|I|h}\big[ \prod_{k\in I} l'_k(u'^k)\big]\Big)^a.\]
As, 
\[\Big(\sum_{I\in R} f'_{|I|h}\big[ \prod_{k\in I} l'_k(u'^k)\big]\Big)^a=\sum_{(I_\alpha)_\alpha\in R^a} 
\prod_{\alpha} f'_{|I_\alpha|h}\big[\prod_{k\in I_\alpha} l_k'(u'^k)\big]\]
Thus, the polynomial map $l\Big(\sum_{I\in R} f'_{|I|h}\big[ \prod_{k\in I} l'_k(u'^k)\big]\Big)$ is equal to  
\begin{equation}\label{ag}
\sum_{m=1}^r
\sum_{(I_\alpha)_{\alpha\in A}\in R^*,\; \sum_\alpha\Sigma I_\alpha=m} 
l_{|A|} \Big[\prod_{\alpha\in A} f'_{|I_\alpha|h}\big[\prod_{k\in I_\alpha} l_k'(u'^k)\big]\Big]\end{equation}
with $R^*:= \cup_{a=1}^r R^a$.

By identification, it follows from equations  (\ref{agd}), (\ref{ad}), and (\ref{ag}) that for every $m\in \{1,\dots, r\}$

\[ \underbrace{\sum_{(I_\alpha)_{\alpha\in A}\in R^*,\; \sum_\alpha\Sigma I_\alpha=m} 
l_{|A|} \Big[\prod_{\alpha\in A} f'_{|I_\alpha|h}\big[\prod_{k\in I_\alpha} l_k'(u'^k)\big]\Big]}_{l'_m\; \mathrm{only\; occurs\; for}\; (I_\alpha)_\alpha=((m));\;  l_m\;  \mathrm{only\; for}\; (I_\alpha)_\alpha\in \{\{0\},\{1\}\}^m}=
\underbrace{ \sum_{I\in R,\; \Sigma I=m}f'_{|I|v} \big[\prod_{k\in I}l'_k(u'^k)\big]}_{\mathrm{here} \;l'_m\; \mathrm{only\; occurs\; for}\; I=(m)}\]
Thus,  there exists an algebraic function $\phi$, such that $f'_{1v}\circ l'_m(u'^m)$ is equal to
\[\sum_{(i_\alpha)_{\alpha=1}^M\in \{0,1\}^m} l_m\Big(\prod_{\alpha=1}^m f_{1h}'\circ l_{i\alpha}'(u'^{i_\alpha})\Big)+
l_1\circ f_{1h}'\circ l'_m(u'^m)+ \phi\big((l_i)_{i<m},(l_i')_{i<m},(f_i)_{i=1}^r\big)\]

Since the linear par $l_1$ of $l$ is small, we have 
 \[l_m'=(f'_{1v}-l_1\circ f'_{1h})_{|\chi^\bot}^{-1}
 \Big[ \phi\big((l_i)_{i<m},(l_i')_{i<m},(f_i)_{i=1}^r\big)+ 
 \sum_{I\in \{0,1\}^m} l_m\circ \prod_{k\in I} f'_{1h}\circ l_k\Big].\]
For $x\in i(\tilde K_p)$ and $f'=f$, we have $f'_{1h|\chi^\bot}=0$.

Thus,  
\[\sum_{I\in \{0,1\}^m} l_m\prod_{k\in I} f'_{1h}\circ l_k=l_m\circ  \big(f_{1h} \big)^m.\] 
It follows from the $r$-normal expansion that the map
\[C\;:\; l_m\mapsto  (f'_{1v}-l_1\circ f'_{1h})_{|\chi^\bot}^{-1} \circ l_m\circ  \big(f_{1h}\big)^m\]
is $\lambda$-contracting, when $l_1$ is small.

By induction, the map $l_s'$ is an algebraic function of only  $(l_k)_{k\le s}$ and $(f'_k)_{j\le s}$, for $s\le m$.
 Thus, 
 \[l_m'=C_m(l_m)+\phi\big((l_i)_{i<m},(f_i)_i\big)\]

This implies that for a norm on $P^r$, the map $\phi_{fx}$ is contracting, for $U$ small enough and then $f'$ $C^r$-close to $f$.
 
Then we prove conclusion 3 of the lemma, by replacing $P^1$ by $P^r$ in the proof of the case $r=1$ done in section \ref{r=1}. 
\end{proof}

\subsubsection{Proof of gluing lemma \ref{lem9}}\label{structloc}
For every $j>p$, let us admit the following lemma:

\begin{lemm}\label{sous-lem9} There exist of a neighborhood $G_j$ of the graph of $i_{|cl(V_{C_j})}$ and a continuous map 
\[\phi_j\;:\; V_{f'} \rightarrow Mor^r((\mathcal T\times M)_{|G_j}\times [0,1],M)\] 
\[f'\longmapsto \big[(x,y,t)\mapsto \phi_{jf'}(x,t)(y)\big]\]
 such that for every $f'\in V_f$ and $x\in cl(V_{C_j})$:
   \begin{itemize}
 \item[$a.$] For every $t\ge 1/N$, $\phi_{jf'}(x,t)$ is a retraction of $G_{jx}:= G_j\cap (\{x\}\times M)$ onto $\mathcal F_{jx}^{f'}:= \mathcal F_{px}^{\eta'}\pitchfork i_{p+1}(f')(\mathcal L_{jx}^{\eta'})$, 
\item[$b.$] the map $\phi_{jf'}(x,0)$ is equal to the identity, 
\item[$c.$] the restriction of $\phi_{jf'}(x,t)$ to $\mathcal F_{jx}^{f'}$ is the identity for every $t\in [0,1]$.
\end{itemize}
\end{lemm}
Let $(r_j)_{j>p}$ be a partition of the unity subordinate to the  covering $(V_{C_j}\setminus \cup_{p<k<j} cl(V'_{C_k}))_j$ of $A'_{p+1}$.

By restricting $V_f$, we can define for all $(x,y)$ in a neighborhood $G$ of the graph of $i_{|V_\Delta}$, $t\in [0,1]$, and $f'\in V_f$, the point 
\[\gamma_{f'}(x,y,t):= \phi_{Nf'}(x,r_N(x))\circ \cdots \circ \phi_{(p+1) f'}(x,r_1(x))\circ \psi_{f'}(x,y,t),\]
 with $\psi_{f'}(x,y,t):= Exp\big(t\cdot Exp^{-1}_{i_{p+1}(f')(x)}(y,t)\big)$ and 
 $\phi_{jf'}(x,0)(y)=y$, for all  $(x,y)\in G$, $t\in[0,1]$, and $f'\in V_f$.
 We notice that, by $b$,  $\gamma_{f'}$ is a $(\mathcal T\times M)_{|G}\times [0,1]$-controlled $C^r$-morphism which depends continuously on $f'\in V_f$.
 
 Let us check properties 1-2-3 of lemma \ref{lem9}.
 
 1) The point $\psi_{f'} (x,y,0)$ is equal to $i_{p+1}(f')(x)$. Since $i_{p+1}(f')(x)$ belongs to each submanifold $\mathcal F_{jx}^{f'}$ such that $x$ belongs to $V_{C_j}$, by $c$, we get 1).

 2)  Let $(x,y)\in G$ such that $x$ belongs to $V_{C_j}'$, for some $j\in \{p+1,\dots ,N\}$. Therefore $r_k(x)$ is equal to zero, for every $k>j$. Thus, the sum $\sum_{p<k\le j} r_k(x)$ is equal to $1$. Consequently, there exists $k\in \{p+1,\dots ,j\}$ such that $r_k(x)$ is greater than $1/N$. Therefore, by $a$, the point
 \[z:=  \phi_{kf'}(x,r_k(x))\circ \cdots \circ \phi_{1f'}(x,r_1(x))\circ \psi_{f'}(x,y,t)\]
 belongs to $\mathcal F_{kx}^{f'}$. 
 
 By coherence of the tubular neighborhoods, $z$ belongs to $\mathcal F_{mx}^{f'}$, for all $m>k$ such that 
 $x$ belongs to $ V_{C_m}$. 
 
 By $c$, the point $\phi_{mf'}(x,r_m(x))(z)$ is $z$. Consequently, $\gamma_{f'}(x,y,t)$ is equal to $z$ which belongs to $\mathcal F_{jx}^{f'}$. 
 
 3) This last property is obvious, by definition of $\psi$ and $c$.
 
 It remains to prove lemma \ref{sous-lem9}. 
 
 \begin{proof}[Proof of lemma \ref{sous-lem9}]
 We can construct, for every $f'\in V_f$, a map $\pi_{kf'}\;:\; G_k \mapsto V_{C_k}$ such that $\pi_{kf'}(x,y)$ is the unique intersection point of the transverse intersection of $i_{p+1}(f')(\mathcal L_{kx}^{\eta'})$ with the submanifold $Exp(F_{kx}+u)$, where $u$ is the orthogonal projection of $Exp^{-1}_x(y)$ on $F_{kx}^\bot$.
 
 By using the implicit function theorem, as in lemma \ref{lem8}, one shows that $\pi_{kf'}$ is $(\mathcal T\times M)_{|G^k}$-controlled and depends continuously on $f'\in V_f$.
      
 Let $\phi_{jf'}^0\;:\; G^j\times [0,1] \rightarrow M$ be defined by 
 
 \[\phi_{jf'}^0(x,y,t)= Exp\big(t\cdot Exp_x^{-1}(\pi_{jf'}(x,y)) +(1-t)Exp^{-1}_x(y)\big).\]
 
 The map $(\phi_{jf'}^0)_{f'}$ are $C^r$-$(\mathcal T\times M)_{|G_j}\times [0,1]$-controlled and depends continuously on $f'\in V_f$. 
 
 Moreover, the image by $ \phi_{jf'}^0$ of $\mathcal F_{px}\times [0,1]$ is contained in $\mathcal F_{px}$.
 
 Let $P\;:\; G^j\times [0,1] \rightarrow M$ be defined by $P(x,y,t)=Exp\big((1-t)p^x(Exp^{-1}_x(y)\big)$ where $p^x\;:\; T_{i(x)}M\rightarrow F_{px}$ is the orthogonal projection.
 
 As $P$ is a composition of controlled maps, $P$ is a controlled map.
 
 Finally we define 
 \[\phi_{jf'}(x,t)\; :\; y\mapsto \phi_{jf'}^0(x,P(x,y,\rho(t)),\rho(t)),\]
 where $\rho$ is a $C^\infty $ real function equal to $1$ on $[1/N,+\infty [$ and to $0$ on $\mathbb R^-$.
 
 Thus, for all $f'\in V_f$ and $x\in$, the map $\phi_j(f',x,t)$ preserves $\mathcal F_{px}$ for every $t\in [0,1]$. 
 
 For $t=0$, the map $\phi_j(f',x,t)$ is well equal to to the identity and hence property $b$ is satisfied. 
 
 For $t\in[1/N,1]$, the point $\phi_j(f',x,t)$ is the composition of a retraction into $\mathcal F_{px}$ with a retraction onto $i_{p+1}(f')(\mathcal L_{jx}^{\eta'})$ which preserves $\mathcal F_{px}$. Hence, $\phi_{j}(f',x,t)$ satisfies property $a$.
 
 For any $t\in[0,1]$, the map $P(x,\cdot,t)_{|\mathcal F_{px}}$ is equal to the identity and the restriction of $\phi_{jf'}^0(x,\cdot,t)$ to ${i_{p+1}(f')(\mathcal L_{jx}^{\eta'})}$ is also equal to the identity, hence the property $c$ is satisfied.  
\end{proof}

\appendix{}
\newpage\section{Analysis on laminations and on trellis}
\subsection{Partition of unity}
\subsubsection{Partition of unity on a lamination}\label{partlam}
\begin{propr}\label{lcl cpct}
\begin{enumerate}\item Let $L$ be a second countable locally compact metric space. There exists an increasing sequence of compact subsets $(K_n)_{n\ge 0}$ whose union is equal to $L$ and such that, for every $n\ge 0$, the compact subset $K_n$ is included in the interior of $K_{n+1}$.
\item Let $(L,\mathcal L)$ be a lamination. There exists a locally finite open covering $(V_i)_i$ of $L$, such that each open subset $V_i$ is precompact in a distinguish open subset.
\end{enumerate}
\end{propr}
\begin{proof}

 1) By local compactness of $L$, for every $x\in L$, we can define the supremum $r_x$ of $r\in ]0,1[$ such that the ball $B(x,r)$ is precompact. As $L$ is second countable, there exists a family $(x_i)_{i\in\mathbb N}$ dense in $L$. Thus, for each $x\in L$, there exists a point $x_i$ at a distance less than  $r_x/8$ from $x$. Therefore, the ball $B(x_i, r_x/4)$ is included in $B(x,r_x/2)$. As the last ball is precompact, the ball $B(x_i, r_x/4)$ is also precompact; this implies that $r_{x_i}\ge r_x/4$.
We remark that $x$ belongs to the ball $B(x_i,r_x/8)$ which is included in $B(x_i,r_{x_i}/2)$. Thus, the family of precompact balls  $(B(x_i,r_{x_i}/2))_i$ is a covering of $L$.
 
Let $K_n:= \cup_{0\le i\le n} cl\big(B(x_i,r_{x_i}/2)\big)$. The family of compact subsets $(K_n)_n$ is increasing and its union is equal to $L$. For every $n\ge 0$, the family $(K_n\setminus int(K_{n+p}))_{p\ge 0}$ is a decreasing sequence of compact subset whose intersection is empty:
\[\bigcap_{p\ge 0} K_n\setminus int(K_{n+p})= K_n\setminus \bigcup_{p\ge 0} int(K_{n+p})\subset
K_n\setminus \bigcup_{i\ge 0} B(x_i,r_{x_i}/2)=\emptyset.\]
Consequently, there exists $p\ge 0$ such that $K_n\setminus int(K_{n+p})$ is empty; in other words $K_n$ is included in the interior of $K_{n+p}$. Thus, by considering a subsequence of $(K_n)_n$, we may suppose that $K_n$ is included in the interior of $K_{n+1}$.

2) Let $(K_n)_n$ be the sequence of compact subsets given by 1). We denote by $C_n$ the compact subset $K_n\setminus int(K_{n-1})$ (with $K_{-1}:=\emptyset$). For each $n\ge 0$ and $x\in C_n$, there exists $r_x^n>0$ such that $B(x,r_x^n)$ is disjoint from $K_{n-2}$, included in $K_{n+1}$ and with (compact) closure included in a distinguish open subset of $\mathcal L$. By compactness of $C_n$, there exists a finite family $(x_i)_{i\in I_n}$ of points of $C_n$, such that $(B(x_i,r_{x_i}^n))_n$ covers
$C_n$. Thus, the family $(V_i)_i:= (B(x_i,r_{x_i}^n))_{n\ge 0,\; i\in I_n}$ is a locally finite covering of $L$ such that each open subset $V_i$ is included in a distinguish open subset.

\end{proof}

\begin{prop}\label{part1} Let $(L,\mathcal L)$ be a $C^r$-lamination, for some $r\ge 1$.
\begin{enumerate}
\item For all $\eta>0$ and $x\in L$, there exists a nonnegative function $\rho\in Mor^r( \mathcal L,\mathbb R)$ whose support is included in $B(x,\eta)$ and such that $\rho(x)$ is positive.
\item Given a locally finite open covering $(U_i)_{i\in I}$ of $L$, there exists $(\rho_i)_i\in Mor^r(\mathcal L,\mathbb R^+)^I$ such that $\sum_i \rho_i =1$ and such that the support of $\rho_i$ is included in $U_i$. We will say that  $(\rho_i)_i$ is a partition of unity subordinate to $(U_i)_i$.
\item The subset of morphisms from the lamination $(L,\mathcal L)$ to $\mathbb R$ is dense in the space of the continuous functions on $L$ endowed with the $C^0$-strong topology. \end{enumerate}

\end{prop}
\begin{proof}

1) Let $(U,\phi)\in \mathcal L$ be a chart of a neighborhood of $x$, which can be written in the form 
\[\phi\;:\; U\rightarrow V\times T\]
where $V$ is a open subset of $\mathbb R^d$ and $T$ a metric space. We denote by $\phi_1$ and $\phi_2$ the coordinates of $\phi$. We can suppose that $\phi_1(x)=0$. Let $\rho_1\in C^\infty(V,\mathbb R^+)$ be a nonnegative function with compact support, such that $\rho_1(0)$ is nonzero and the preimage by $\phi$ of $supp(\rho_1)\times \{\phi_2(x)\}$ is included into the ball $B(x, \eta)$.

By compactness, there exists a neighborhood $\tau$ of $\phi_2(x)$ in $T$ such that the preimage by $\phi$ of $supp(\rho_1)\times \tau$ is included in the ball $B(x, \eta)$. Let $\rho_2$ be a nonnegative continuous function on $T$, with support in $\tau$ and nonzero at $\phi_2(x)$. We define then 
\[\rho\;:\; y\mapsto \left\{\begin{array}{cl} \rho_1\circ\phi_1(y)\cdot \rho_2\circ\phi_2(y) &\mathrm{if}\; y\in U\\
0&\mathrm{else}\end{array}\right.\]
We note that the function $\rho$ satisfies the requested properties.

2) Let us begin by admitting this statement when $I$ is finite.
Let $(K_n)_n$ be a sequence of compact subsets of $L$, given by property \ref{lcl cpct}.1.
Let $K_{-1}:=K_{-2}:=\emptyset$.  Thus, for each $n\ge 0$, there exists a function $r_n\in Mor^r(\mathcal L,[0,1])$ equal to 1 on $K_n\setminus K_{n-1}$ and 0 on $K_{n-2}\cup K_{n+1}^c$. Let $(U_i)_{i\in I_n}$ be a finite subcovering of the covering $(U_i)_{i\in I}$ of $K_{n+1}\setminus K_{n-2}$. Thus, there exists $(\rho_i^n)_{i\in I_n}$ a partition of unity subordinate to the open covering $(U_i)_{i\in I_n}$ of $\cup_{i\in I_n} U_i$. Let 
\[\rho_i:=\frac{\sum_{\{n:\; I_n\ni i\}} r_n\cdot \rho_i^n}{\sum_n r_n}\in Mor^r(\mathcal L,\mathbb R^+),\]
whose support is in $U_i$ and satisfies
 \[\sum_i \rho_i=\frac{\sum_n r_n\sum_{i\in I_n}\rho_i^n}{\sum_n r_n}=\frac{\sum_n r_n}{\sum_nr_n}=1.\]
Consequently $(\rho_i)_i$ is a partition of unity subordinate to $(U_i)_i$. It is now sufficient to prove the existence of a partition of unity when $I$ is finite.

Let us show, by induction on the cardinality of $I$, that it is sufficient to prove this proposition when the  cardinality of $I$ is equal to 2.
If the cardinality of $I$ is $k+1>2$, by the induction hypothesis, there exists  a partition of unity  $(r_0,r_{k+1})$ subordinate to $(\cup_{j\le k} U_j,U_{k+1})$ and a partition of unity $(r_j)_{j=1}^k$ subordinate to $(U_j)_{j=1}^{k}$ on the restriction of $\mathcal L$ to $\cup_{j\leq k} U_j$.
Then we note that $((r_0\cdot r_j)_{j=1}^k,r_{k+1})$ is a partition of unity subordinate to $(U_j)_{j=1}^{k+1}$.

We now suppose that the covering $(U_j)_j$ is constituted by only two subsets $U_1$ and $U_2$. Let us define two close subsets $F_1$ and $F_2$ included in respectively $U_1$ and $U_2$ such that the union of $F_1$ with $F_2$ is equal to $L$.

If, for example, $U_1$ is equal to $L$, we choose $F_1:=L$ and $F_2:=\emptyset$. If neither $U_1$ neither $U_2$ is equal to $L$, we define 
\[F_1:= \{x\in L;\; d(x,U_1^c)\ge d(x,U_2^c)\}\quad \mathrm{and}
\quad F_2:=\{x\in L;\; d(x,U_1^c)\le d(x,U_2^c)\}\]
Obviously, these two subsets cover $L$. Suppose, for the sake of contradiction, that $F_2$ is not included in  $U_2$. Thus, there exists a point $x$ which belongs to $U_2^c\cap F_2$ and so satisfies 
\[d(x,U_1^c)\le d(x,U_2^c)=0.\]
Consequently $x$ belongs to the intersection of $U_1^c$ with $U_2^c$ which is empty, this is a contradiction. In the same way, we prove that  $F_1$ is included in $U_1$.

Let us now construct two nonnegative functions $r_1\in Mor^r(\mathcal L,\mathbb R)$ and $r_2\in Mor^r(\mathcal L,\mathbb R)$, such that the functions $r_1$ and $r_2$ are nonzero at all points of respectively $F_1$ and $F_2$, and have their support included in respectively $U_1$ and $U_2$.

 The following functions will then satisfy statement 2):
\[\rho_1:= \frac{r_1}{r_1+r_2}\quad \mathrm{and}\quad \rho_2:= \frac{r_2}{r_1+r_2}.\]

Let us construct, for example, the function $r_1$.

It follows from property \ref{lcl cpct}, that there exists an increasing sequence of compact subsets $(K_n)_n$ whose union is equal to $L$, and such that for any $n\ge 0$, the interior of $K_{n+1}$ contains $K_n$. Let $C_n:= K_n\setminus int(K_{n-1})$, with $K_{-1}=\emptyset$. Let $D_n:= C_n\cap F_1$. For every $x\in D_n$, there exists $\eta_x^n>0$ such that the ball $B(x,\eta_x^n)$ does not intersect $K_{n-2}$ and is included in $K_{n+1}\cap U_1$. Let $\rho_x^n$ be the function given by the first
statement of this proposition with $\eta=\eta_x^n$. We denote by $U_x^n$ the subset of points at which this function is nonzero. We note that the family of open subsets $(U_x^n)_{x\in D_n}$ is a covering of the compact subset $D_n$. Hence, there exists a finite subcovering $(U_{x_i}^n)_{i\in I_n}$. We remark that the family $(U_{x_i}^n)_{\{n\ge 0,\;i\in I_n \}}$ is a locally finite covering of $F_1$ and of union included in $U_1$. Thus, the following function is appropriate:
\[r_1:= \sum_{n\ge 0,\; i\in I_n} \rho_{x_i}^n.\]

3) Let $f\in C^0(L,\mathbb R)$ and $\epsilon>0$. Let us construct a function $f'\in Mor^r(\mathcal L,\mathbb R)$ satisfying  \[\sup_{x\in L}|f(x)-f'(x)|\le \epsilon.\]

Let $(U_i)_i$ be a locally finite covering of $L$ by precompact distinguish open subsets. For each $i$, let $\phi_i\;:\; U_i\rightarrow \mathbb R^d\times T_i$ be a chart. We denote by $\phi_{i1}$ and $\phi_{i2}$ its coordinates. Let $(\rho_i)_i\in Mor^r(\mathcal L,\mathbb R)^\mathbb N$ be a partition of unity subordinate to $(U_i)_i$. Let $W_i:= U_i\setminus \rho^{-1}(\{0\})$ which is precompact in $U_i$.

Let $r\in C^{\infty}(\mathbb R^d, \mathbb R^+)$ be a function whose support is included in the unity ball and whose integral is equal to 1.

For each $i$, let $\epsilon_i>0$ small such that the following function is well defines:
 \[f_i: W_i\rightarrow \mathbb R\]
 \[x\mapsto \frac{1}{\epsilon_i^d}\int_{B(0,\epsilon_i)} f\Big(\phi_{i}^{-1}\big(\phi_{i1}(x)+y,\phi_{i2}(x)\big)\Big)\cdot r\Big(\frac{y}{\epsilon_i}\Big)dy\]
and satisfies $\sup_{W_i} |f_i-f|<\epsilon$.

From the classical properties of the convolutions, the following function satisfies all the required properties:
 \[x\mapsto \sum_{\{i;\; x\in U_i\}} \rho_i(x)\cdot f_i(x).\]

\end{proof}

\subsubsection{Partition of unity controlled on a stratification of laminations}\label{partstra}
\begin{prop}\label{lem11}Let $(A,\Sigma)$ be a stratified space endowed with a $C^r$-trellis structure $\mathcal T$, for some $r\ge 1$.
\begin{enumerate}
\item For all $\eta>0$ and $x\in A$, there exists a nonnegative function $\rho\in Mor^r(\mathcal T,\mathbb R)$ whose support is included in $B(x,\eta)$ and such that $\rho(x)$ is positive.
\item For any $\eta>0$ and any function $\rho_0$ continuous on $A$, there exists a  $C^r$-$\mathcal T$-controlled function $\rho$ on $A$ such that 
\[\sup_{x\in A} |\rho(x)-\rho_0(x)|\le \eta.\]
\item Given an open covering $(U_i)_{i\in I}$ of $A$, there exists $(\rho_i)_i\in Mor^r(\mathcal T,\mathbb R^+)^I$ such that the support of $\rho_i$ is included in $U_i$ and $\sum_i \rho_i =1$. We say that $(\rho_i)_i$ is a partition of unity subordinate to $(U_i)_i$.
\end{enumerate}\end{prop}
\begin{proof}

1-2) Let us show statement 1 and 2 in the same times. We will replace all the propositions about the sign of the constructed functions by, respectively, the propositions about the distance to $\rho_0$ of the constructed functions.

We denote by $(X_p)_p$ and $(L_p,\mathcal L_p)_p$ the lamination obtained
from $\Sigma$ and $\mathcal T$, by property \ref{arbo}. For any $k\ge 0$, let $U_k:=\cup_{p\le k}L_p$.

We are going to construct, by induction on $k\ge 0$, a continuous function $\rho_k$ on $A$ such that:
\begin{itemize}
\item for $j\le k$, $\rho_{k|L_j}$ is morphism from $\mathcal L_j$ to $\mathbb R$,
\item for $j\le k$, the restriction to $U_j$ of $\rho_k$ is equal to the one of $\rho_j$,
\item $\rho_k$ is nonzero at $x$, nonnegative on $A$ and with support included in $B(x,(1-2^{-k-1})\cdot \eta)$.

\noindent (resp. $\sup_A | \rho_k-\rho|\le (1-2^{-k-1})\cdot \eta$)\end{itemize}

For the step $k=0$, we simply choose a continuous function $\rho_{0}$ on $A$, nonnegative, with support included in $B(x,\eta/2)$ and such that $\rho_{0}(x)>0$ (resp. for the step $k=0$, we chose the function $\rho_0$ given in the hypotheses).

We suppose the induction hypothesis satisfied for $k\ge 0$. By property \ref{lcl cpct},
there exists a locally finite open covering $(W_i)_i$ of $L_{k+1}$, such that each open subset $W_i$ is precompact in a distinguish open subsets of $\mathcal L_{k+1}$. By splitting each of these open subsets into smaller, we may also suppose that the diameter of $W_i$ is less than the distance from $W_i$ to the complement of $L_{k+1}$.

For each $j\le k+1$, we fix a Riemannian metric on $(L_j,\mathcal L_j)$. For any open subset $W$ in $L_j$ and $\lambda\in Mor^r(\mathcal L_{j|W},\mathbb R)$, we define 
\[\|\lambda\|_{Mor^r(\mathcal L_{j|W},\mathbb R)}:=\sup_{x\in W} \left(\sum_{s=1}^r\|\partial_{\partial_x^s\mathcal L_j} \lambda\|\right),\]
where the norm $\|\cdot\|$ is subordinated to the induced norm by the Riemannian metric on $T\mathcal L_j$ and to the Euclidean norm on $\mathbb R$.

We chose then a partition of unity $(\lambda_i)_i\in Mor^r(\mathcal L_{k+1},\mathbb R^+)^\mathbb N$ subordinate to $(W_i)_i$. For each $i$, we define 
\[\epsilon_i:=\frac{\eta}{2^{k+2+i}}\cdot \min\left(1,\frac{diam(W_i)}{\|\lambda_i\|_{Mor^r(\mathcal L_{k+1},\mathbb R)}},diam (W_i)\right)>0.\]\\

For each $i$, we apply the following lemma, that we will show at the end:

\begin{lemm}\label{lem10}
There exists a function $\rho'_i\in Mor^r(\mathcal L_{k+1|W_i}, \mathbb R)$ such that:
\begin{enumerate}
\item If the closure $W_i$ is included in $L_j$, for any $j\le k$, we have then 
\[||\rho_{k|W_i}-\rho'_i||_{Mor^r( \mathcal L_{j|W_i},\mathbb R)}<\epsilon_i\]
And, in the case of the first statement, we have moreover 
\item the support of $\rho'_i$ is included in the $\epsilon_i$-neighborhood of the support of $\rho_{k|W_i}$,
\item the function $\rho'_i$ is nonnegative, and if $x$ belongs to $W_i$, then $\rho'_i(x)$ is positive.
 
\end{enumerate}
\end{lemm}

\[\mathrm{Let}\;\rho_{k+1}\;:\; y\mapsto \left\{
\begin{array}{cr}
\sum_i \lambda_i(y)\cdot \rho_i'(y)& \mathrm{if}\; y\in L_{k+1}\\
\rho_k(y)                          &  \mathrm{else}\\
\end{array}\right.
\]

In the first statement case, we have well defined a nonnegative function which is positive at $x$. As for each $i$ the support of $\rho_i'$ is included in the $\frac{\eta}{2^{k+2}}$-neighborhood of $\rho_k$, the support of $\rho_{k+1}$ is included in the $\frac{\eta}{2^{k+2}}$-neighborhood of the support of $\rho_k$, so in $B(x,(1-2^{-k-2})\cdot \eta)$.

In the second statement case, for $y\in L_{k+1}$, the number $|\rho_{k+1}(y)-\rho(y)|$ is less than 
\[|\rho_k(y)-\rho(y)|+|\rho_k(y)-\rho_{k+1} (y)|\le (1-2^{-k-1})\eta+\sum_i \lambda_i(y)\cdot \epsilon_i \le (1-2^{-k-2})\cdot\eta\]
and for $y\in L_{k+1}^c$, the number $|\rho_{k+1}(y)-\rho(y)|$ is equal to $|\rho_{k}(y)-\rho(y)|$ which is less than
$(1-2^{-k-2})\cdot\eta$.\\

Now, let us show that, for $j\le k+1$, the function $\rho_{k+1|L_j}$ is $C^r$-morphism from  $\mathcal L_j$ to $\mathbb R$.

By local finiteness of the covering $(W_i)_i$, the map $\rho_{k+1|L_{k+1}}$ belongs to $Mor^r(\mathcal L_{k+1},\mathbb R)$. Thus, for $j\le k$, the function $\rho_{k+1|L_{k+1}\cap L_j}$ belongs to $Mor^r(\mathcal L_{j|L_{k+1}\cap L_j},\mathbb R)$. Moreover, for $y\in L_{k+1}$,
\[|\rho_k(y)-\rho_{k+1} (y)|\le \sum_i \lambda_i(y)\cdot |\rho_i'(y)-\rho_k(y)|\le \sum_{i;\; x\in W_i} \epsilon_i\le \sum_{i;\; x\in W_i} \frac{\eta \cdot diam W_i}{2^{i+2}},\]
\begin{equation}\label{rec1} \Rightarrow |\rho_k(y)-\rho_{k+1} (y)|\le  \eta\cdot d(y,L_{k+1}^c)
\end{equation}
Hence, the function $\rho_{k+1}$ is continuous.

For any $i\le k$ and $x_0\in L_i\setminus L_{k+1}$, there exists $r >0$ such that the ball $B(x_0,r)$ is included in $L_i$. If any $W_j$ intersects $B(x_0,r/2)$, then the closure of $W_j$ is contained in $L_i$. Thus, for every $y\in B(x_0,r/2)\cap W_j$, the number $\|\partial_{T\mathcal L_i}(\rho_k-\rho_{k+1}) (y)\|$ is less than 
\[\sum_{\{j;\;W_j\ni y\}}\sum_{k=0}^s C^k_s\underbrace{ \left\|\partial_{T\mathcal L_i}^{s-k}\lambda_j(y)\cdot \partial_{T\mathcal L_i}^{k}\big(\rho_{j}'(y)-\rho_k(y)\big)\right\|}
_{\le \frac{\eta}{2^{k+2+j}}\cdot diam W_j }.\]

As $diam W_j\le d(y, x_0)$, we have 
\begin{equation}\label{rec2}
    \|\partial_{T\mathcal L_i}(\rho_k-\rho_{k+1}) (y)\|\le \eta\sum_{k=0}^s C^k_s\cdot d(y,x_0).
\end{equation}

From equations (\ref{rec1}) and (\ref{rec2}), the restriction $\rho_{k+1|L_i}$ is a $C^r$-morphism  from $\mathcal L_i$ into $\mathbb R$, for each $i\le k$.
 
As $\Sigma$ is locally finite, the family $(L_k)_k$ is also locally finite. Thus, the sequence $(\rho_k)_k$ is locally eventually constant. Let $\rho$ be the limit of $(\rho_k)_k$. Therefore, this sequence satisfies, for any $k\ge 0$, that $\rho_{|L_k}$ belongs to $Mor^r(\mathcal L_k,\mathbb R)$. Hence, for all $X \in\Sigma$, the restriction of $\rho$ to $L_X$ is a $C^r$-morphism from $\mathcal L_X$ to $\mathbb R$. Consequently, $\rho$ is a  $\mathcal T$-controlled $C^r$-morphism. Moreover, the first (resp. second)
statement is checked.

     3) We do exactly the same proof as for proposition \ref{part1}.2, by replacing '$L$' by '$A$' and '$Mor^r(\mathcal       L,\mathbb R)$' by '$Mor^r(\mathcal T,\mathbb R)$'.
\end{proof}

    \begin{proof}[proof of lemma \ref{lem10}]

    Let $(U,\phi)$ be a chart of $\mathcal L_{k+1}$ such that the closure of $W_i$ is included in $U$. Let $d_{k+1}$ be         the dimension of $\mathcal L_{k+1}$, $V$ be an open subset of $\mathbb R^{d_{k+1}}$ and $\tau$ be a locally compact         metric space, such that 
    \[\phi\;:\; U{\longrightarrow} V\times \tau\]
    \[x\mapsto (\phi_1 (x),\phi_2(x))\]

    Let $r\in C^\infty (\mathbb R^{d_{k+1}},\mathbb R^+)$ be a function with support included in the unity ball,                        nonnegative on this ball and with integral on $\mathbb R^{d_{k+1}}$ equal to 1.
    \[\mathrm{For\; any}\; x'\in W_i,\; \mathrm{let}\; \rho'_i(x')= \frac{1}{\mu^{d_{k+1}}}\cdot \int_{y\in B(0,\mu)} \rho(z)\cdot r\left(\frac{y}{\mu}\right)dy,\]
    with $z:=\phi^{-1}\big(\phi_1(x')-y,\phi_2(x')\big)$ and $\mu>0$ small enough for $\phi$ to be well defined.


    It follows from the classical properties of convolutions that the function $\rho'_i$ is a morphism from $\mathcal           L_{k+1|W_i}$ to $\mathbb R$.

    Let us prove 1). For this, we now assume that $x'$ belongs to $W_i\subset cl(W_i)\subset L_j$. By taking $\mu$ small                enough, the point $z$ (defined above) always belongs to $L_i$.
 
    \[\Rightarrow\; \partial_{T_{x'}\mathcal L_{i}}^s\rho'_i= \frac{1}{\mu^{d_{k+1}}}\cdot \int_{y\in B(0,\mu)}                           \partial_{T_{z}\mathcal L_{i}}^s(\rho\circ z) \cdot  r\left(\frac{y}{\mu}\right)\: dy,\; \forall s\in \{1,\dots ,r\}.\]

    As, for $\mu>0$ small and $y\in B(O,\mu)$, the map $x\mapsto z$ is $C^r$-close to the identity, the function $x\mapsto \rho\circ z$ is $C^r$-close to $\rho$. Therefore, by taking  $\mu$ sufficiently small, we have 
    \[||\rho_{|W}-\rho'_i||_{Mor^r(\mathcal L_{i|W}),\mathbb R)}<\epsilon_i.\]

    In the first assertion case, for $\mu$ small enough, conclusion 2) is well satisfied. Conclusion 3) is                  obvious.

\end{proof}

\subsection{Density of smooth liftings of a smooth map}

Throughout this section, we denote by $G$ and $M$ two Riemannian manifolds and $p\;:\; G\rightarrow M$ a $C^\infty$-bundle.

Given a family of numbers $(r_k)_{k=1}^n\in [0,1]^n$ and given a family of points  $(m_k)_{k=1}^n$ that belong to a same fiber $G_x$ of $G$ and each other sufficiently close, using the Riemannian metric we may define \cite{CM} the centroid  $cent\{(m_k)_{k=1}^n,(r_k)_{k=1}^n\}\in G_x$
of the family of points $(m_k)_{k=1}^n$ weighted by the masses $(r_k)_{k=1}^n$ respectively.
This centroid is a $C^\infty$-map from the product of the product bundle $G^n$ over $M$, with $[0,1]^n$, to $G$.
The centroid does not depend on the indexation in $\{1,\dots,n\}$. Finally, if we add some points with wedges equal to zero, the centroid remains the same.

\subsubsection{Density of smooth liftings of a morphism of a lamination}\label{rellam}

Let $(L,\mathcal L)$ be a lamination and $i$ be a morphism from $(L,\mathcal L)$ to $M$.
\begin{prop}\label{rel} The subset of liftings of $i$ in $F$ which are $C^r$-morphisms from $(L,\mathcal L)$ to $G$ is dense in the space of the continuous liftings of $i$ endowed with the strong $C^0$-topology.\end{prop}
\begin{proof}

Let $N$ be a continuous lifting of $i$ and let $\epsilon$ be a positive number. Let us show the existence of a lifting $N'\in Mor^r(\mathcal L,G)$ of $i$ such that 
\[\sup_{x\in L} d(N(x),N'(x))\le \epsilon.\]

By property \ref{lcl cpct}, we may construct a locally finite covering $(U_k)_k$ of $L$, such that for each $k$, $N(U_k)$ is included into a precompact distinguish open subset $V_k$ of the bundle $G$. This means that there exists a trivialization $\phi_k$ of class $C^\infty$ from $V_k$ onto $p(V_k)\times \mathbb R^d$:
\[\phi_k\;:\; V_k\stackrel{\sim}{\rightarrow} p(V_k)\times \mathbb R^d.\]
As $N$ is a lifting of $i$, for each $k$, there exists a continuous map $F_k$ from $U_k$ into $\mathbb R^d$, such that 
\[\phi_k\circ N_{|U_k}\;:\; U_k{\rightarrow} p(U_k)\times \mathbb R^d\]
\[x\mapsto(i(x),F_k(x)).\]

By proposition \ref{part1} 2), there exists a partition of unity $(\rho_k)_k\in Mor^r(\mathcal L,[0,1])^\mathbb N$ subordinate to $(U_k)_k$.

Thus, by proposition \ref{part1} 3), there exists for each $k$, a morphism $F_k'\in Mor^r(\mathcal L_{|U_k}, \mathbb R^d)$ close enough to $F_{|U_k}$ such that:
\begin{itemize}
\item  the following morphism of laminations to be well defined:
\[N'\;:\; L\rightarrow G\]
\[x\mapsto cent\Big\{\big(F'_k(x)\big)_{\{k;\; x\in U_k\}},\big(\rho_k(x)\big)_{\{k;\; x\in U_k\}}\Big\},\]
\item for each $x\in L$,
\[d(N'(x),N(x))\le \epsilon.\]
\end{itemize}
Finally, we note that $N'$ is a lifting of $i$.\end{proof}

\subsubsection{Density of smooth controlled liftings of a controlled morphism}\label{relstra}
Let $(A,\Sigma)$ be a stratified space endowed with a trellis structure $\mathcal T$ and let $i$ be a $\mathcal T$-controlled morphism into $M$.
\begin{prop}\label{rellem} The subset of $C^r$-liftings of $i$ into $F$ which are $\mathcal T$-controlled is dense in the space of continuous liftings of $i$ endowed with the strong $C^0$-topology.\end{prop}
\begin{proof}
We do exactly the same proof as in proposition \ref{rel}, by replacing '$L$' by '$A$', '$\mathcal L$' by '$\mathcal T$', and proposition \ref{part1} by proposition \ref{lem11}.\end{proof}

\newpage\section{Adapted metric}
In this chapter we proof the following proposition stated in section \ref{section:P1}, where some notations used below are defined.
\label{proof P1}

\noindent{\bf Proposition \ref{P1}.}  {\it Let $(L,\mathcal{L})$  be a lamination  and let $(M,g)$ be a Riemannian $C^{\infty}$-manifold.
Let $f\in End^1 (M)$, $i\in Im^1(\mathcal{L},M)$, and $f^*\in End^1(\mathcal{L})$.

 If $f$ $r$-normally expands the lamination $\mathcal L$ immersed by $i$ over $f^*$, for every compact subset $K$ of $L$ stable by $f^*$ ($f^*(K)\subset K$), there exist a Riemannian metric $g'$ on $M$ and $\lambda'<1$, such that for the norm induced by $g'$ on $i^*TM$ and every $v\in (i^*TM/T\mathcal L)_{|K}\setminus \{0\}$, we have 
\[\max\big( 1,\|T_{\pi(v)}f^{*}\|^r\big)\cdot \|v\|<\lambda'\cdot \|[i^*Tf](v)\|\]

We say that $g'$ is an adapted metric to the normal expansion of $f$ on $K$.}

\begin{proof}
The existence of an adapted metric when $f$ is a diffeomorphism has been proved recently by Nikolaz Gourmelon \cite{Go}. In the following proof, we adapt some of his ideas.

Let $B$ the compact set $i(K)$ of $M$ and let $F$ be the vector bundle $TM_{|B}\rightarrow B$. 
Let $F'$ be the vector bundle over $B$ whose fiber at $y\in B$ is $Ti(T_x\mathcal L)$ if $x$ is sent by $i$ to $y$. By property \ref{propertyA}, $F'$ is a continuous vector bundle, well defined, even if $i_{|K}$ is not injective.  are $(F_x')_{x\in B}$. We endow $F$ with the norm induced by the Riemannian metric of $M$.

We denote by $T$ the restriction of $Tf$ to the bundle $F$, which is a bundle morphism over $f$.
As $T$ preserves the subbundle $F'$, this morphism defines a morphism, denoted by $[T]$, on the quotient bundle $F/F'$ over $B$.

For $x\in B$ and $n\ge 0$, we define 
\[m\big([T]^n(x)\big):=\min_{u\in (F/F')_x,\;\|u\|=1}\big(\|[T]^n(u)\|\big)\]

By $r$-normal expansion and compactness of $B$, there exist $N>0$ and $a<1$ such that for every $x\in B$,
\[\max\big(1,\| T^N_{|F'}(x)\|^r\big)<a^{2N}\cdot m([T]^N(x))\]

Therefore, there exists a function $r$ on $B$, continuous and greater than 1, such that for every $x\in B$ 
\[\frac{1}{a}  \sqrt[N]{\big\|T_{|F'}^N(x)\big\|^r}<r(x)<a\cdot\sqrt[N]{m\big([T]^N(x)\big)}\]

We denote by $R_n$ the continuous function on $B$ defined by 
\[R_n\;:=\;x\mapsto\prod_{i=0}^{n}r(f^i(x))\]

We use now the following lemma, that we will prove at the end:
\begin{lemm}\label{eq1}
There exists $c>0$ such that, for all $x\in B$ and $n\ge 0$, we have 
\[\frac{\big\|T^{n}_{|F'}(x)\big\|}{\sqrt[r]{R_n(x)}}\le c\cdot a^{n}\quad \mathrm{and}\quad \frac{m\big([T]^{n}(x)\big)}{R_n(x)}\ge c^{-1}\cdot a^{-n}\]
\end{lemm}

So there exists  $M\ge 0$ such that, for every $x\in B$, $\frac{m\big([T]^{M+1}(x)\big)}{R_{M+1}(x)}$ is greater than $\frac{1}{r(x)}$. For every $(x,u)\in F$, let $u_1$ be the orthogonal projection  of $u$ onto $F'_x$ and let $u_2$ be the equivalence class of $u-u_1$ in $(F/F')_x$. By lemma \ref{eq1}, the following Euclidean norm is well defined and depends continuously on $(x,u)$:

\[\|(x,u)\|'^2:=\sum_{n=0}^\infty\frac{\| T^n(x,u_1)\|^2}{R_n(x)^{\frac{2}{r}}}+
\sum_{n=0}^M\frac{\|[T]^n(x,u_2)\|^2}{R_n(x)^2}\]

We remark that we have 
\[\|T(x,u_1)\|'^2=\sum_{n=0}^\infty\frac{\| T^{n+1}(x,u_1)\|^2}{R_n(f(x))^{\frac{2}{r}}}
=r(x)^{\frac{2}{r}}\cdot\sum_{n=1}^\infty\frac{\| T^{n}(x,u_1)\|^2}{R_{n}(x)^{\frac{2}{r}}}\le r(x)^{\frac{2}{r}}\cdot\|(x,u_1)\|'\]
Hence, the norm induced by $\|\cdot \|'$ of $T_{|F'_x}$ is less than $\sqrt[r]{r(x)}$.

If $u_2\in (F/F')_x$ is nonzero, we have 
\[\|[T](x,u_2)\|'^2= \sum_{n=0}^M\frac{\|[T]^{n+1}(x,u_2)\|^2}{R_n(f(x))^2}=r(x)^2\cdot \sum_{n=1}^{M+1}\frac{\|[ T]^{n}(x,u_2)\|^2}{R_n(x)^2}\]
\[= r^2(x) \cdot\left(\|(x,u_2)\|'^2+ \frac{\|[T]^{M+1}(x,u_2)\|^2}{R_{M+1}(x)^2}-\frac{\|(x,u_2)\|^2}{r(x)^2}\right)>r^2(x) \cdot\|(x,u_2)\|'^2\]

Therefore, the real number  $\|[ T](x)^{-1} \|'^{-1}$ is greater than $r(x)>1$.

It follows from the two last conclusions that for every $x\in B$,
\[\big\| [T](x)^{-1}\big\|'.\max \big(1,\big\|T_{|F'}(x)\big\|'^r\big)< 1\]

By compactness of $B$, there exists a upper bound $\lambda'<1$ such that for $x\in B$, we have 
\[\big\| [T](x)^{-1}\big\|'.\max \big(1,\big\|T_{|F'}(x)\big\|'^r\big)< \lambda'\]

We extend the Euclidean norm $\|\cdot \|'$  on $F=TM_{|B}$ to a continuous Riemannian metric $g''$ on $TM$. We chose then a $C^\infty$-Riemannian metric $g'$ on $M$, close enough to $g''$ to have, with the norm induced by $g'$ on $i^*TM$:
\[\forall v\in (i^*TM/T\mathcal L)_{|K}\setminus \{0\},\quad \max\big( 1,\|T_{\pi(v)}f^{*}\|^r\big)\cdot \|v\|<\lambda'\cdot \|[i^*Tf](v)\|.\]
\end{proof}
\begin{proof}[Proof of lemma \ref{eq1}]
 \[\mathrm{Let}\; C:=\max_{x\in B}\left(\|T_{|F'}(x)\|,\|[ T](x)^{-1}\|,r(x)\right)>1\quad\mathrm{and}\quad c:=C^{4N}\cdot a^{-2N}\]

For every $n\in \mathbb N$, let $q\in \mathbb N$ and $p\in\{0,\dots,N-1\}$ such that $n=q \cdot N+p$.
For $x\in B$, we have 
\[ R_n(x)=\prod_{i=0}^{N-1}\prod_{j=0}^{q-1}r(f^{i+jN}(x))\cdot \prod_{k=0}^{p}r(f^{q N+k}(x)).\]
The first inequality of this lemma, when $q\ge 1$, is obtained by the following calculus:
\[\sqrt[r]{R_n(x)}\ge \prod_{i=0}^{N-1}\prod_{j=0}^{q-2}\frac{\sqrt[N]{\big\|T_{|F'}^{N}(f^{i+jN}(x))\big\|}}{a}
\ge \prod_{i=0}^{N-1} \frac{\sqrt[N]{\big\|T_{|F'}^{N(q-1)}(f^{i}(x))\big\|}}{a^{q-1}}\]
 \[\Rightarrow  \sqrt[r]{R_n(x)}\ge \prod_{i=0}^{N-1} \frac{\sqrt[N]{\big\|T_{|F'}^{n}(x)\big\|}}{a^{q-1}\cdot C^{2}}\ge
C^{-2N}\cdot a^{2N}\frac{\big\|T_{|F'}^{n}(x)\big\|}{a^n}
\ge c^{-1}\cdot \frac{\big\|T_{|F'}^{n}(x)\big\|}{a^n}.\]
If $q=0$, then $n< N$ and $\sqrt[r]{R_n(x)}\ge 1\ge  c^{-1}\cdot \frac{\big\|T_{|F'}^{n}(x)\big\|}{a^n}$.

The second inequality of this lemma, when $q\ge 1$, is obtained by the following calculus:
\[R_n(x)\le \prod_{i=0}^{N-1}\prod_{j=0}^{q-2}\Big(
a\cdot\sqrt[N]{m\big([T]^{N}(f^{i+jN}(x))\big)}
\Big)\cdot C^{N+p}\]
\[\Rightarrow R_n(x)\le \prod_{i=0}^{N-1} \Big(a^{q-1}\cdot \sqrt[N]{m\big([T]^{ N(q-1)}(f^{i}(x))\big)}\Big)\cdot C^{2N}\]
 \[\Rightarrow R_n(x)\le \prod_{i=0}^{N-1} \Big({a^q\cdot C^{2}\cdot \sqrt[N]{m\big([ T]^{n}(x)\big)}}\Big)\cdot a^{-N}\cdot C^{2N}\le c\cdot a^n\cdot  m\big([T]^{n}(x)\big).\]
If $q=0$, then $n< N$ and $R_n(x)\le C^N\le c\cdot a^n\cdot  m\big([T]^{n}(x)\big)$.

\end{proof}

\newpage\section{Plaque-expansiveness}\label{pppppplaque}

The definition of the plaque-expansiveness in the diffeomorphism context and the endomorphism context are different and recalled in section \ref{plaque-expansiveness}.

The plaque-expansiveness is satisfied in all the known examples of compact lamination normally expanded or hyperbolic. Nevertheless, we do not know if every compact lamination, normally  expanded or hyperbolic are plaque-expansive.
Moreover, we do not know if this hypothesis is necessary for a lamination to be persistent (as an embedded lamination).
\subsection{Plaque-expansiveness in the diffeomorphism context}
In the diffeomorphism context, up to our knowledge there exist essentially two results, both were proved in \cite{HPS}.

In order to state the first result, let us recall that a lamination $(L,\mathcal L)$ embedded into a manifold is \emph{locally} a saturated subset of a $C^1$-foliation if for every $x\in L$ there exists a $C^1$-foliation $\mathcal F$, on a neighborhood $U$ of $x$, such that $\mathcal L_{|U\cap L}$ is equal to $\mathcal F_{|U\cap L}$.

\begin{propr}[Hirsch-Pugh-Shub]\label{propHPS}
Let $(L,\mathcal L)$ be a compact lamination embedded  into a manifold $M$. Let $f$ be a diffeomorphism normally hyperbolic to this lamination. Then $f$ is plaque-expansive if $(L,\mathcal L)$ is locally a saturated subset of a $C^1$-foliation. \end{propr}

The second result was generalized in \cite{RHU} and require the definition of the \emph{Lyapunov stability}:

\begin{defi}Let $f$ be a diffeomorphism of a manifold $M$ preserving a compact lamination $(L,\mathcal L)$ embedded into $M$.
The diffeomorphism $f$ is \emph{ Lyapunov stable} along $\mathcal L$ if for every small $\epsilon>0$, there exists $\delta>0$ such that, for all $x\in L$ and $n\ge 0$, the plaque\footnote{Recall that we denote by $\mathcal L_x^\delta$ the union of the plaques whose diameter is less than $\delta$ and which contains x} $f^n(\mathcal L_x^\delta)$ is included in $\mathcal L_x^\epsilon$.\end{defi}

We remark that if the restriction of $f$ to the leaves of $\mathcal L$ is an isometry then $f$ is Lyapunov stable along $\mathcal L$.
\begin{prop}[Rodriguez Hertz- Ures]\label{Urez} Let $(L,\mathcal L)$ be a compact lamination embedded into a manifold $M$. Let $f$ be a diffeomorphism of $M$ which preserves $(L,\mathcal L)$.
 
\begin{itemize}
\item If $f$ is Lyapunov stable along $\mathcal L$ and normally expands this lamination, then $f$ is plaque-expansive.
\item If $f$ is normally hyperbolic on this lamination and if $f$ and $f^{-1}$ are Lyapunov stable along $\mathcal L$, then $f$ is plaque-expansive.\end{itemize}\end{prop}
\subsection{Plaque-expansiveness in the endomorphism context}
In the endomorphism context, we have generalized a little bit the above result:
\begin{prop}\label{Myprop}
Under the hypotheses of theorem \ref{th1}, we suppose moreover that the lamination $(L,\mathcal L)$ is embedded. Let $\mathcal L':= \mathcal L_{|L'}$. We suppose that there exist $A>0$ and $\delta>0$ such that, for every $x\in L'$, the subset $\mathcal L_x'^A$ is precompact in the leaf of $x$, and we have for any $n\ge 0$
\[f^{n}(\mathcal L_{x}'^\delta)\subset \mathcal L_{f^{n}(x)}'^A.\]
Then $f$ is plaque-expansive at $(L',\mathcal L')$.\end{prop}
\begin{proof}
This proof uses several ideas from \cite{RHU}, in particular the one
where we consider the forward iterates of pseudo-orbits.

As $L'$ is precompact, we may suppose that the metric of $M$ satisfies property \ref{property B} for the compact subset $K=cl(L')$. We denote by $\exp$ the exponential map associated to this metric. Thus, there exists a cone field over $L'$ in $TM_{|L'}$ such that, for each $x\in L'$, $T_x\mathcal L^\bot$ is a maximal vector subspace included in $C(x)$ and satisfies moreover:

There exist a small $\epsilon_0>0$ and $\lambda>1$ such that, for all $x\in L'$ and $u\in C(x)$ with norm less than $\epsilon_0$, we have 
\begin{equation}\label{dilacona}
v:=\exp_{f(x)}^{-1}\circ f\circ \exp_{x}(u)\in C(f(x))\quad\mathrm{and}\quad \|v\|\ge \lambda \|u\|.\end{equation}
By precompactness, for $\epsilon_0>0$ sufficiently small, there exists $\eta>0$ such that for every $(x,y)\in L'^2$ satisfying $y=\exp_x(u)$, with $u\in C(x)$ of norm in $[\epsilon_0/\sup_{L'}\|Tf\|,\epsilon_0]$, we have
\begin{equation}\label{disteta} d(\mathcal L_x'^A,\mathcal L_y'^A)>\eta.\end{equation}

Let $p\in \mathbb N$ such that $\lambda^p\cdot \eta > \epsilon_0$.

We can also suppose that $\delta$ is less than  $\frac{\epsilon_0}{\sup_{L'}\|Tf\|^p}$.
\begin{fact}\label{thefact} There exists a small $\epsilon\in ]0,\epsilon_0[$ such that, for every pair of $\epsilon$-pseudo-orbits $(x_n)_n$ and $(y_n)_n$ which respect $\mathcal L'$ and satisfy 
\[d(x_n,y_n)<\epsilon\quad\mathrm{and} \quad y_n\notin \mathcal L_{x_n}'^\epsilon,\;\forall n\ge 0,\]
there exists a sequence $(z_n)_n\in L'^{\mathbb N}$ such that, for every $n\ge 0$, $z_n$ belongs to the intersection of $\exp_{x_n}(B_{C(x_n)}(0,\delta))$ with a small plaque containing $y_n$ (but not $x_n$).

For $\epsilon>0$ small enough, $f^p(z_n)$ belongs to $\mathcal L_{z_{n+p}}'^\delta$ and $f^p(x_n)$ belongs to $\mathcal L_{x_{n+p}}'^\delta$.\end{fact}

We have proved this proposition if there do not exist such pseudo-orbits $(x_n)_n$ and $(y_n)_n$.
 We suppose, for the sake of contradiction, that there exist such sequences $(x_n)_n$ and $(y_n)_n$, and so $(z_n)_n$.

The fact \ref{thefact} implies that, for all $k\ge 0$ and $j\ge 0$, the sequences $(f^k(x_{pn+j}))_n$ and $(f^k(z_{pn+j}))_n$ are $A$-pseudo-orbits of $f^p$ which respect $\mathcal L'$.

For $k\ge 0$, let $M_k:= \sup_n d(f^k(x_n),f^k(z_n))$. The number $M_0$ belongs to the interval $]0,\epsilon_0/\sup_{L'}{\|Tf\|^p}[$. Moreover, if $M_j<\epsilon_0$ for every $j\le k$, by (\ref{dilacona}) and the fact \ref{thefact}, the number $M_{k+1}$ belongs to $[\lambda M_k, \sup_{L'}\|Tf\|M_k]$. Thus, there exists $k_0\ge 0$ such that $M_{k_0+p}$ belongs to $]\epsilon_0/\sup_{L'}{\|Tf\|},\epsilon_0]$ and $M_j$ is less than $\epsilon_0$ for $j\le k_0+p$. Hence, there exists $n_0\ge 0$ such that 
        \[d(f^{k_0+p}(x_{n_0}),f^{k_0+p}(z_{n_0}))\in \left[\frac{\epsilon_0}{\sup_{L'}\|Tf\|},\epsilon_0\right].\]
Therefore, by (\ref{disteta}), we have 
        \[d(f^{k_0}(x_{n_0+p}),f^{k_0}(z_{n_0+p}))>\eta\]
Consequently, as $\lambda^p\eta$ is greater than $\epsilon_0$, we have 
         \[d(f^{k_0+p}(x_{n_0+p}),f^{k_0+p}(z_{n_0+p}))>\epsilon_0\]
This contradicts $M_{k_0+p}\le \epsilon_0$.
\end{proof}
\begin{rema} Under the hypotheses of theorem \ref{th1}, if the leaves of $\mathcal L$ are the connected components of the fibers of a bundle, then $f^*$ is plaque-expansive at $\mathcal L'$, by proposition \ref{Myprop}.\end{rema}

The following is equivalent, in the endomorphism context, to property \ref{propHPS}:
\begin{propr}
Under the hypotheses of theorem \ref{th1}, we suppose moreover that $(L,\mathcal L)$ is embedded. We denote by $\mathcal L'$ the lamination $\mathcal L_{|L'}$.  If $\mathcal L$ is locally a saturated subset of a $C^1$-foliation, then $f$ is plaque-expansive at $(L',\mathcal L')$. \end{propr}

\begin{proof} We suppose that $M$ is endowed with a metric which satisfies property \ref{property B} for the compact subset $K=cl(L')$. We denote by $\exp$ the exponential map associated to this metric. Thus, there exists a cones field $C$ on $L'$ in $TM_{|L'}$ such that, for each $x\in L'$, $T_x\mathcal L^\bot$ is a maximal vector subspace included in $C(x)$ and which satisfies moreover:

There exist $\epsilon_0>0$ and $\lambda>1$ such that, for all $x\in L'$ and $u\in C(x)$ with norm less than $\epsilon_0$, we have 
\begin{equation}\label{dilacone}
v:=\exp_{f(x)}^{-1}\circ f\circ \exp_{x}(u)\in C(f(x))\quad\mathrm{and}\quad \|v\|\ge \lambda \|u\|.
\end{equation}
Moreover, for $\epsilon_0$ small enough, by the $C^1$-foliation hypothesis, there exists a number $C>0$ such that, for all $(x,y)\in L'^2$, if
$y$ belongs to $\exp(C(x)\cap B_{T_xM}(0,\epsilon_0))$, the distance $d(\mathcal L_x^{\epsilon_0},\mathcal L_y^{\epsilon_0})$ is greater than $C d(x,y)$. Let $p\ge 0$ such that $C\lambda^p>2$. Then there exists $\epsilon_1\in]0,\epsilon_0[$ small enough such that for every $\epsilon_1$-pseudo-orbit $(x_n)_n$ which respects $\mathcal L'$, the sequence $(x_{np})_n$ is an $\epsilon_0$-pseudo-orbit of $f^p$ which respects $\mathcal L'$.
 
There exists $\epsilon\in ]0,\epsilon_1[$ such that, for every pair $\Big((x_n)_n,(y_n)_n\Big)$ of $\epsilon$-pseudo-orbits of $f$, which respects $\mathcal L'$ and satisfies 
\[d(x_n,y_n)<\epsilon,\quad \forall n\ge 0,\]
there exists $z_n\in \mathcal L_{y_n}'^{2\epsilon}$ such that $z_n$ belongs to $\exp(C(x_n)\cap B(0_{x_n},\epsilon_0))$, the distance $d(z_n,x_n)$ is less than $\epsilon_1$ and $(z_n)_n$ is an $\epsilon_1$-pseudo-orbit of $f^*$ which respects $\mathcal L'$.
 
Consequently, the sequences $(z_{pn})_n$ and $(x_{pn})_n$ are $\epsilon_0$-pseudo-orbits of $f^p$ which respects the plaques of $\mathcal L'$, such that $z_{pn}$ belongs to $\exp(C(x_{pn})\cap B(0_{x_{pn}},\epsilon_0))$ and such that the distance $d(z_{pn},x_{pn})$ is less than $\epsilon_0$.

Thus, the distance  $d(f^p(z_{pn}),f^p(x_{pn}))$ is greater than $\lambda^p d(z_{pn},x_{pn})$. Moreover, the distance $d(f^p(z_{pn}),f^p(x_{pn}))$ is less than $d(z_{p(n+1)},x_{p(n+1)})/C$. Therefore, 
$d(z_{p(n+1)},x_{p(n+1)})$ is twice greater than $d(z_{pn},x_{pn})$. We conclude that $d(z_{pn},x_{pn})$ is greater than $2^nd(z_0,x_0)$ and less than $\epsilon_0$, this implies the equality of $x_0$ and $z_0$.
Thus, $x_0$ belongs to $\mathcal L_{y_0}'^{2\epsilon}$.
\end{proof}

\begin{rema}

Under the hypotheses of theorem \ref{th2}, if for each stratum $X\in \Sigma_{|A'}$ the hypotheses of the proposition or the above property are satisfied on a precompact subset $L'$ of $X$, such that 
\begin{equation}\label{derniere equation}f^*(cl(L'))\subset L',\; cl(L')\subset int\Big(f^{*^{-1}}\big(cl(L')\big)\Big)\;\mathrm{and}\; \cup_{n\ge 0} f_{|A'}^{*^{-n}}(cl(L')=X,\end{equation}
we can reduce and extend the plaque-expansiveness constant on $L'$ to a continuous positive function $\epsilon$ on $X$, for which $f^*$ is plaque-expansive at $X$ (as in section \ref{zzzzz}). We note that there always exists an precompact open subset satisfying condition (\ref{derniere equation}): for example we can take $int(K_p)\cap X_p'$ for the stratum $X_p'$, with the notation of the demonstration of theorem \ref{th2}.
\end{rema}

\nocite{*}
\bibliographystyle{alpha}
\bibliography{references}

\def\polhk#1{\setbox0=\hbox{#1}{\ooalign{\hidewidth
  \lower1.5ex\hbox{`}\hidewidth\crcr\unhbox0}}}
  \def\polhk#1{\setbox0=\hbox{#1}{\ooalign{\hidewidth
  \lower1.5ex\hbox{`}\hidewidth\crcr\unhbox0}}}
  \def\polhk#1{\setbox0=\hbox{#1}{\ooalign{\hidewidth
  \lower1.5ex\hbox{`}\hidewidth\crcr\unhbox0}}} \def\cprime{$'$}
\begin{thebibliography}{AGZV85}

\bibitem[AGZV85]{MR777682}
V.~I. Arnol{\cprime}d, S.~M. Guse{\u\i}n-Zade, and A.~N. Varchenko.
\newblock {\em Singularities of differentiable maps. {V}ol. {I}}, volume~82 of
  {\em Monographs in Mathematics}.
\newblock Birkh\"auser Boston Inc., Boston, MA, 1985.
\newblock The classification of critical points, caustics and wave fronts,
  Translated from the Russian by Ian Porteous and Mark Reynolds.

\bibitem[BDV05]{BDV}
Christian Bonatti, Lorenzo~J. D{\'{\i}}az, and Marcelo Viana.
\newblock {\em Dynamics beyond uniform hyperbolicity}, volume 102 of {\em
  Encyclopaedia of Mathematical Sciences}.
\newblock Springer-Verlag, Berlin, 2005.
\newblock A global geometric and probabilistic perspective, Mathematical
  Physics, III.

\bibitem[Bek91]{Bek}
K.~Bekka.
\newblock C-r\'egularit\'e et trivialit\'e topologique.
\newblock In {\em Singularity theory and its applications, Part I (Coventry,
  1988/1989)}, volume 1462 of {\em Lecture Notes in Math.}, pages 42--62.
  Springer, Berlin, 1991.

\bibitem[Ber]{PB}
Pierre Berger.
\newblock Persistence des stratifications de laminations normalement dilatées,
  phd thesis, université paris xi.

\bibitem[BST03]{BST}
J{\'e}r{\^o}me Buzzi, Olivier Sester, and Masato Tsujii.
\newblock Weakly expanding skew-products of quadratic maps.
\newblock {\em Ergodic Theory Dynam. Systems}, 23(5):1401--1414, 2003.

\bibitem[Cer61]{C}
Jean Cerf.
\newblock Topologie de certains espaces de plongements.
\newblock {\em Bull. Soc. Math. France}, 89:227--380, 1961.

\bibitem[dM73]{dM}
W.~de~Melo.
\newblock Structural stability of diffeomorphisms on two-manifolds.
\newblock {\em Invent. Math.}, 21:233--246, 1973.

\bibitem[Dou62]{D}
Adrien Douady.
\newblock Vari\'et\'es \`a bord anguleux et voisinages tubulaires.
\newblock In {\em S\'eminaire Henri Cartan, 1961/62, Exp. 1}, page~11.
  Secr\'etariat math\'ematique, Paris, 1961/1962.

\bibitem[Ghy99]{MR1760843}
{\'E}tienne Ghys.
\newblock Laminations par surfaces de {R}iemann.
\newblock In {\em Dynamique et g\'eom\'etrie complexes (Lyon, 1997)}, volume~8
  of {\em Panor. Synth\`eses}, pages ix, xi, 49--95. Soc. Math. France, Paris,
  1999.

\bibitem[Gou]{Go}
Nikolaz Gourmelon.
\newblock Adapted metrics for dominated splittings.
\newblock {\em Preprint}.

\bibitem[G{\'S}98]{grac}
Jacek Graczyk and Grzegorz {\'S}wi{\polhk{a}}tek.
\newblock {\em The real {F}atou conjecture}, volume 144 of {\em Annals of
  Mathematics Studies}.
\newblock Princeton University Press, Princeton, NJ, 1998.

\bibitem[Hir76]{H}
Morris~W. Hirsch.
\newblock {\em Differential topology}.
\newblock Springer-Verlag, New York, 1976.
\newblock Graduate Texts in Mathematics, No. 33.

\bibitem[HPS77]{HPS}
M.~W. Hirsch, C.~C. Pugh, and M.~Shub.
\newblock {\em Invariant manifolds}.
\newblock Springer-Verlag, Berlin, 1977.
\newblock Lecture Notes in Mathematics, Vol. 583.

\bibitem[Kar77]{CM}
H.~Karcher.
\newblock Riemannian center of mass and mollifier smoothing.
\newblock {\em Comm. Pure Appl. Math.}, 30(5):509--541, 1977.

\bibitem[Lyu97]{Lyu}
Mikhail Lyubich.
\newblock Dynamics of quadratic polynomials. {I}, {II}.
\newblock {\em Acta Math.}, 178(2):185--247, 247--297, 1997.

\bibitem[Ma{\~n}78]{Manethese}
Ricardo Ma{\~n}{\'e}.
\newblock Persistent manifolds are normally hyperbolic.
\newblock {\em Trans. Amer. Math. Soc.}, 246:261--283, 1978.

\bibitem[Ma{\~n}88]{Mane}
Ricardo Ma{\~n}{\'e}.
\newblock A proof of the {$C\sp 1$} stability conjecture.
\newblock {\em Inst. Hautes \'Etudes Sci. Publ. Math.}, (66):161--210, 1988.

\bibitem[Mat73]{Ma}
John~N. Mather.
\newblock Stratifications and mappings.
\newblock In {\em Dynamical systems (Proc. Sympos., Univ. Bahia, Salvador,
  1971)}, pages 195--232. Academic Press, New York, 1973.

\bibitem[Mic80]{Michor}
Peter~W. Michor.
\newblock {\em Manifolds of differentiable mappings}, volume~3 of {\em Shiva
  Mathematics Series}.
\newblock Shiva Publishing Ltd., Nantwich, 1980.

\bibitem[Mil06]{Mi}
John Milnor.
\newblock {\em Dynamics in one complex variable}, volume 160 of {\em Annals of
  Mathematics Studies}.
\newblock Princeton University Press, Princeton, NJ, third edition, 2006.

\bibitem[MT06]{MT}
Claudio Murolo and David Trotman.
\newblock Semidifférentiabilité et version lisse de la conjecture de fibration
  de whitney.
\newblock {\em Advanced Studies in Pure Mathematics}, (43):271--309, 2006.

\bibitem[PS70]{PS}
J.~Palis and S.~Smale.
\newblock Structural stability theorems.
\newblock In {\em Global Analysis (Proc. Sympos. Pure Math., Vol. XIV,
  Berkeley, Calif., 1968)}, pages 223--231. Amer. Math. Soc., Providence, R.I.,
  1970.

\bibitem[RHRHU]{RHU}
F.~Rodriguez~Hertz, M.A. Rodriguez~Hertz, and R.~Ures.
\newblock A survey on partially hyperbolic dynamics.
\newblock {\em arXiv:math.DS}.

\bibitem[Rob71]{Ri}
J.~W. Robbin.
\newblock A structural stability theorem.
\newblock {\em Ann. of Math. (2)}, 94:447--493, 1971.

\bibitem[Rob76]{Rs}
Clark Robinson.
\newblock Structural stability of {$C\sp{1}$} diffeomorphisms.
\newblock {\em J. Differential Equations}, 22(1):28--73, 1976.

\bibitem[Rud80]{Ru}
Walter Rudin.
\newblock {\em Analyse r\'eelle et complexe}.
\newblock Masson, Paris, 1980.
\newblock Translated from the first English edition by N. Dhombres and F.
  Hoffman, Third printing.

\bibitem[Shu69]{Shubthese}
Michael Shub.
\newblock Endomorphisms of compact differentiable manifolds.
\newblock {\em Amer. J. Math.}, 91:175--199, 1969.

\bibitem[Shu78]{Sh}
Michael Shub.
\newblock {\em Stabilit\'e globale des syst\`emes dynamiques}, volume~56 of
  {\em Ast\'erisque}.
\newblock Soci\'et\'e Math\'ematique de France, Paris, 1978.
\newblock With an English preface and summary.

\bibitem[Sma67]{Sm}
S.~Smale.
\newblock Differentiable dynamical systems.
\newblock {\em Bull. Amer. Math. Soc.}, 73:747--817, 1967.

\bibitem[Tho64]{Th}
R.~Thom.
\newblock Local topological properties of differentiable mappings.
\newblock In {\em Differential Analysis, Bombay Colloq.}, pages 191--202.
  Oxford Univ. Press, London, 1964.

\bibitem[Tro79]{Trot}
David J.~A. Trotman.
\newblock Geometric versions of {W}hitney regularity for smooth
  stratifications.
\newblock {\em Ann. Sci. \'Ecole Norm. Sup. (4)}, 12(4):453--463, 1979.

\bibitem[Via97]{V}
Marcelo Viana.
\newblock Multidimensional nonhyperbolic attractors.
\newblock {\em Inst. Hautes \'Etudes Sci. Publ. Math.}, (85):63--96, 1997.

\bibitem[Whi65a]{W1}
Hassler Whitney.
\newblock Local properties of analytic varieties.
\newblock In {\em Differential and Combinatorial Topology (A Symposium in Honor
  of Marston Morse)}, pages 205--244. Princeton Univ. Press, Princeton, N. J.,
  1965.

\bibitem[Whi65b]{W2}
Hassler Whitney.
\newblock Tangents to an analytic variety.
\newblock {\em Ann. of Math. (2)}, 81:496--549, 1965.

\bibitem[Yoc95]{Y}
Jean-Christophe Yoccoz.
\newblock Introduction to hyperbolic dynamics.
\newblock In {\em Real and complex dynamical systems (Hiller\o d, 1993)},
  volume 464 of {\em NATO Adv. Sci. Inst. Ser. C Math. Phys. Sci.}, pages
  265--291. Kluwer Acad. Publ., Dordrecht, 1995.

\end{thebibliography}
\bigskip
\bigskip

Pierre Berger\\
Institute for Mathematical Sciences\\
Stony Brook University\\
Stony Brook, NY 11794-3660

\noindent{pierre.berger@normalesup.org}

\noindent {http://www.math.u-psud.fr/~pberger/}

\noindent {http://www.math.sunysb.edu/~berger/}

\end{document}